\documentclass[11pt,a4paper,twoside]{article}
\usepackage{amssymb}
\usepackage{amsfonts}
\usepackage{geometry}
\usepackage{color}
\usepackage{mathtools}
\usepackage{amsmath,amsthm,amssymb,amsfonts,enumitem}
\usepackage{CJK}
\usepackage{latexsym}
\usepackage{graphicx}
\usepackage{pgfpages}
\usepackage{mathrsfs}
\usepackage{ytableau}
\usepackage{enumerate}
\usepackage{cite}
\usepackage[hyphens]{url}
\usepackage[bookmarks=false,hidelinks]{hyperref}
\usepackage[affil-it]{authblk}
\allowdisplaybreaks

\newtheorem*{hypothesis}{\bf (H)}
\newtheorem{definition}{Definition}[section]
\newtheorem{theorem}[definition]{Theorem}
\newtheorem{lemma}[definition]{Lemma}
\newtheorem{proposition}[definition]{Proposition}
\newtheorem{corollary}[definition]{Corollary}
\newtheorem{remark}[definition]{Remark}
\newtheorem{condition}{Assumption}

\newtheorem{example}[definition]{Example}

\numberwithin{equation}{section}

\begin{document}
	\title{Entrance measures for semigroups of time-inhomogeneous SDEs: possibly degenerate and expanding
}
	\author[a]{
		Chunrong Feng}
	\author [a] {Baoyou Qu}
	\author[a,b]{Huaizhong Zhao}
	\affil[a]{Department of Mathematical Sciences, Durham University, DH1 3LE, UK}
	\affil[b] {Research Centre for Mathematics and Interdisciplinary Sciences, Shandong University, Qingdao 266237, China}
	
	\affil[ ]{chunrong.feng@durham.ac.uk, baoyou.qu@durham.ac.uk, huaizhong.zhao@durham.ac.uk}
	\date{}
	
\maketitle

\begin{abstract}
	In this article, we solve the problem of the long time behaviour of transition probabilities of time-inhomogeneous Markov processes and give a unified approach to stochastic differential equations (SDEs) with periodic, quasi-periodic, almost-periodic forcing and much beyond. We extend Harris's ``small set'' method to the time-inhomogeneous situation with the help of Hairer-Mattingly's refinement of Harris's recurrence to a one-step contraction of probability measures under the total variation distance $\rho_{\beta}$ weighted by some appropriate Lyapunov function and a constant $\beta>0$. We show that the convergence to an entrance measure under $\rho_{\beta}$ implies the convergence both in the total variation distance and the Wasserstein distance $\mathcal{W}_1$. For SDEs with locally Lipschitz and polynomial growth coefficients, in order to establish the local Doeblin condition, we obtain a nontrivial lower bound estimates for the fundamental solution of the corresponding Fokker-Planck equation. The drift term is allowed to be possibly non-weakly-dissipative, and the diffusion term can be degenerate over infinitely many large time intervals. This causes the system to be expanding over many periods of large time durations, the convergence to the entrance measure is generally only subgeometric. As an application we obtain the existence and uniqueness of quasi-periodic measure. We then lift the quasi-periodic Markovian semigroup to a cylinder on a torus and obtain a unique invariant measure and its ergodicity.
	 
	\medskip
	
	\noindent

	{\bf MSC2020 subject classifications:} Primary 60H10, 60B10; secondary 37A50.

\end{abstract}

\tableofcontents

\section{Introduction to the problem and general results}
Large time behaviour of the laws of time-homogeneous Markov processes and their ergodic theory have been in a central place of probability theory for many years and enormous important results have been obtained. The same problem for time-inhomogeneous Markov processes is much less understood and remains as a challenging and important question. The feature that the law of a certain process at certain time conditioned on the process at an earlier time, unlike the time-homogeneous case, does not only depend on the duration between the two certain times, but also on the actual time at which the condition is taken, creates a lot of difficulties to its analysis. In the time-homogeneous case, if two probability measures pushed by the semigroup can contract at a certain time duration, then they contract at any other time duration, irrespective to the starting time of the time duration concerned. The rate of contraction is the same at all times. Moreover, due to the contraction, the transition probability has a limit distribution which is the invariant measure of the Markovian semigroup. It gives a statistical equilibrium of stochastic system. But in the time-inhomogeneous case, these important features no longer hold, so a lot of arguments cannot work and the study of large time behaviour is largely missing.

Time-inhomogeneous Markov processes arise e.g. when coefficients of stochastic differential equation vary with time. Periodic stochastic and quasi-periodic stochastic systems are two important special cases. They appear in mathematical modelling of many real world problems such as the temperature variant, transition of the ice age and the interglacial period, sunspot activities, economic cycles (big and small), etc (Benth-Benth \cite{benth2007volatility}, Benzi-Parisi-Sutera-Vulpiani \cite{benzi1982stochastic}). In these problems, time-inhomogeneity is one of the key features. 
It is also notable that stochastic climate models are essentially time-inhomogeneous (Flandoli-Tonello \cite{flandoliintroduction}, Majda-Timofeyev-Vanden Eijnden \cite{Majda-cpam2001}). The study of their large time behaviour and statistical equilibrium is of great importance.

In the periodic stochastic system case, when we consider discrete time slices with equal gap of one period of the system, though the discrete system becomes a time-homogeneous Markov chain and could be mixing with a stationary measure (PS-mixing), the continuous time system process is never mixing. Random periodic paths were studied in Zhao-Zheng \cite{Zhao-Zheng2009}, the relation with periodic measures and  their ergodic theory were obtained in Feng-Zhao \cite{feng2020random}. In the random quasi-periodic case (Feng-Qu-Zhao \cite{feng2021random}), even the discrete time PS-mixing property cannot hold, let alone more general cases.

It is worth mentioning here that the relevance of random periodic paths, periodic measures and their ergodic theory to theoretical and applied problems arising in stochastic dynamical systems and applications has began to attract attentions of many researchers. In particular, there has been progress in the study of a wide variety of topics e.g. bifurcations (Wang \cite{Wang-bifurcation2014}), random attractors (Bates-Lu-Wang \cite{Bates-Lu-Wang2014}), stochastic resonance (Cherubini-Lamb-Rasmussen-Sato \cite{Cherubini-nonlinearity2017}, Feng-Zhao-Zhong \cite{feng2019existence,Feng-Zhao-Zhong2021},
Feng-Liu-Zhao \cite{Feng-Liu-Zhao2021})), random horseshoes (Huang-Lian-Lu \cite{huang2016ergodic}), modelling El Ni$\tilde{\rm n}$o phenomenon (Chekroun-Simonnet-Ghil \cite{Chekroun2011}), stochastic oscillations (Engel-Kuehn \cite{Engel-cmp2021}), large deviations (Gao-Liu-Sun-Zheng \cite{Gao-Liu-Sun-Zheng2022}), ergodic periodic measure approach to time series (Feng-Liu-Zhao \cite{Feng-LiuYj-Zhao2023}), linear response and homogenizations (Branicki-Uda \cite{Branicki-Uda2021}, Uda \cite{Uda2021}), random almost-periodic solutions (Cheban-Liu \cite{Cheban-Liu2020}, Raynaud de Fitte \cite{Raynaud2020almost}), random periodic solutions of certain functional differential equations (Gao-Yan \cite{Gao-Yan2018}) and certain stochastic differential equations and stochastic partial differential equations (Dong-Zhang-Zheng \cite{Dong-Zhang-Zheng2020}, Liu-Lu \cite{Liu-Lu2021,Liu-Lu2022}, Song-Song-Zhang \cite{Song-Song-Zhang2019}). 

The study of the long time behaviour of time-inhomogeneous Markovian semigroup was attempted when the system is globally dissipative in a number of special situations such as Cheban-Liu \cite{Cheban-Liu2020} for almost-periodic case, Feng-Qu-Zhao \cite{feng2021random} for quasi-periodic case and Majid-R\"{o}ckner \cite{majid2021} for Ornstein-Uhlenbeck processes. In the dissipative cases, solutions starting at different initial conditions can converge to a single trajectory, observed also in the time-homogeneous case (Da Prato-Zabczyk \cite{da1996ergodicity}, Mattingly \cite{mattingly1999}, Schmalfuss \cite{Schmalfuss2001}). The pathwise approach clearly do not work for weakly dissipative systems. As far as we know, in the situation of time-inhomogeneous Markovian semigroup, when a stochastic system is only weakly dissipative and locally Lipschitz, only stochastic differential equations with periodic forcing was solved (Feng-Zhao-Zhong \cite{feng2019existence}),  more general cases including SDEs with quasi-periodic and almost-periodic forcing remained open, let alone more general cases. Note also that the method used in \cite{feng2019existence} cannot be extended easily to more general cases when periodic pattern cases to exist as regarding the one-step dynamics as a Markov chain is the main ingredient in this approach.

In this paper, we solve the problem of long time behaviour and entrance measure for non-stationary stochastic differential equations under very mild conditions. Our results give a unified systematic approach to many SDEs that are currently under intensive studies in the random periodic, random quasi-periodic and random almost-periodic cases. But our results go much beyond these results available so far e.g. in the quasi-periodic case \cite{feng2021random}, the almost-periodic case \cite{Cheban-Liu2020}, they were only solved under the strictly dissipative conditions. Our main results are given as follows.

Consider a Markov process on a metric space $\mathbb{X}$ with transition probability $P(t,s,x,\cdot), t\geq s, s\in \mathbb{R}, x\in \mathbb{X}$, $P(t,s)$ is the corresponding semigroup as specified in \eqref{0110-5}, its entrance measure on $\mathbb{X}$ is defined as in Definition \ref{Def 1}.
\begin{condition}
	\label{A1}
  \begin{description}
    \item [(i)] There exist a function $V: \mathbb{X}\to [0,\infty)$ and nonnegative functions $\gamma(t,s)$, $K(t,s)$ for all $t\geq s$ such that 
		\begin{equation*}
			P(t,s)V(x)\leq \gamma(t,s) V(x)+K(t,s), \ \text{ for all } \ x\in \mathbb{X}.
		\end{equation*}
		\item [(ii)]
		For some $R>0$, there exist $\eta(t,s)\in [0,1)$ for all $t\geq s$ and $\nu\in \mathcal{P}(\mathbb{X})$ such that 
		\begin{equation*}
			\inf_{V(x)\leq R}P(t,s, x,\cdot)\geq \eta(t,s)\nu(\cdot).
		\end{equation*}
  \end{description}
\end{condition}

Functions $V$ appearing in Assumption \ref{A1} (i) is normally referred as a Lyapunov function. For a Lyapunov function and $\mu_1,\mu_2\in \mathcal{P}(\mathbb{X})$ with $\int_{\mathbb{X}}V(x)\mu_i(dx)<\infty, \ i=1,2$, define 
\begin{equation*}
	\rho_{\beta}(\mu_1, \mu_2):=\int_{\mathbb{X}}\big(1+\beta V(x)\big)|\mu_1-\mu_2|(dx),
\end{equation*} 
where $\beta>0$ is a constant and $|\mu_1-\mu_2|$ is the total variation of the signed measure $\mu_1-\mu_2$. Moreover, for a set $\mathcal{A}\subset \mathbb{R}$, define
\begin{equation*}
	\mathcal{M}_{\mathcal{A}}:=\bigg\{\mu:\mathbb{R}\rightarrow \mathcal{P}(\mathbb{X})\bigg| \sup_{t\in \mathcal{A}}\int_{\mathbb{X}}V(x)\mu_{t}(dx)<\infty\bigg\}.
\end{equation*}
In particular, when $\mathcal{A}=\mathbb{R}$, denote $\mathcal{M}_{\mathcal{A}}=\mathcal{M}$.

A decreasing sequence $\{t_n\}_{n\geq 0}\subset \mathbb{R}$ is called a well-controlled partition if $t_n\downarrow -\infty$ and both 
\begin{equation}\label{1205-1}
	\gamma:=\max_{n\geq 1}\gamma(t_{n-1},t_n), K:=\max_{n\geq 1}K(t_{n-1},t_n)
\end{equation}
are finite. For any given $\delta>0$, let 
\begin{equation}\label{1205-2}
	A_n^{\delta}:=\{1\leq i\leq n: \eta(t_{i-1},t_i)\geq \delta\} \ \text{ and } \ n^{\delta}:=\# A_n^{\delta} \text{ be the number of elements in } A_n^{\delta},
\end{equation}
and
\begin{equation}\label{1205-4}
	\bar{\gamma}^{\delta}_{n}:=\frac{1}{n^{\delta}}\sum_{i\in A_{n}^{\delta}}\gamma(t_{i-1},t_i).
\end{equation}

\begin{theorem}
	\label{Theorem 1214}
  \begin{itemize}
    \item[(i)] Suppose Assumption \ref{A1} holds. If there exist a well-controlled partition $\{t_n\}_{n\geq 0}\subset \mathbb{R}$ and a subsequence $\mathcal{A}:=\{t_{n_k}\}_{k\geq 1}$ and $\delta>0, \varpi\in (0,1), \gamma^*\in (0,1)$ such that 
		\begin{equation}\label{1205-3}
			\gamma^*<1-\frac{2K}{R}, \ \ \varpi>\frac{(\gamma-1)^+R+2K}{(\gamma-1)^+R+(1-\gamma^*)R}
		\end{equation}
		and
		\begin{equation}\label{0929-3}
			\liminf_{k\to \infty}\frac{n_k^{\delta}}{n_k}>\varpi, \ \ \ \limsup_{k\to \infty}\bar{\gamma}^{\delta}_{n_k}< \gamma^*,
		\end{equation}
		then for any $t\in \mathbb{R}$, there exist constants $\beta>0, r\in (0,1)$ depending only on $(\delta,K,R,\gamma,\gamma^*)$ and constant $C_t$ such that for any $t\geq t_{n_k}$, $\mu_1,\mu_2\in \mathcal{P}(\mathbb{X})$
		\begin{equation}\label{0929-1}
			\rho_{\beta}\big(P^*(t,t_{n_k})\mu_1, P^*(t,t_{n_k})\mu_2\big)\leq C_t r^{n_k}\rho_{\beta}(\mu_1, \mu_2).
		\end{equation}
		\item[(ii)] Suppose all conditions in (i) hold. If in addition we assume that 
		\begin{equation}\label{0918-4}
			\ell_{x_0}:=\sup_{j\geq i\geq 1}P(t_{n_i},t_{n_j})V(x_0)<\infty, \text{ for some } x_0\in \mathbb{X},
		\end{equation}
		then there exists a unique entrance measure $\mu_{\cdot}$ of $P$ in $\mathcal{M}_{\mathcal{A}}$ such that
		\begin{equation}
			\label{1217-2}
			\rho_{\beta}\bigl(P(t,t_{n_k},x,\cdot), \mu_t\bigr)
			\leq C_t\big(1+ V(x)+ V(x_0)+\ell_{x_0}\big)r^{n_k}.
		\end{equation}
  \end{itemize}
\end{theorem}

Theorem \ref{Theorem 1214} includes many varieties of different scenarios, not only the case of geometric convergence, but also cases of subgeometric and supergeometric convergence. In general, we have the following result.

\begin{theorem}\label{Corollary 1101}
	Assume all conditions in Theorem \ref{Theorem 1214} (i) hold. If for any $t\in \mathbb{R}$, there exists a nondecreasing function $\phi: \mathbb{R}^+\to \mathbb{R}^+$ such that
	\begin{equation*}
		\liminf_{k\to \infty}\frac{n_k}{\phi(t-t_{n_k})}>0,
	\end{equation*}
	then there exist $C_t>0,\lambda>0$ such that for any $t\geq t_{n_k}$,
	\begin{equation}
		\label{1027-2}
		\rho_{\beta}\big(P^*(t,t_{n_k})\mu_1, P^*(t,t_{n_k})\mu_2\big)\leq C_t e^{-\lambda \phi(t-t_{n_k})}\rho_{\beta}(\mu_1, \mu_2), \text{ for all } \mu_1,\mu_2\in \mathcal{P}(\mathbb{X}),
	\end{equation}
	and if in addition \eqref{0918-4} holds, then there exists a unique entrance measure $\mu_{\cdot}$ of $P$ in $\mathcal{M}_{\mathcal{A}}$ such that
	\begin{equation}
		\label{1027-1}
		\rho_{\beta}\bigl(P(t,t_{n_k},x,\cdot), \mu_t\bigr)
		\leq C_t\big(1+ V(x)+ V(x_0)+\ell_{x_0}\big)e^{-\lambda \phi(t-t_{n_k})}.
	\end{equation}
	In particular, if $\phi(t)=t^{\alpha}:=\phi_{\alpha}(t), t\in \mathbb{R}^+$ for some $\alpha>0$, then
	\begin{equation}\label{1108-1}
		\rho_{\beta}\bigl(P(t,t_{n_k},x,\cdot), \mu_t\bigr)
		\leq C_t\big(1+ V(x)+ V(x_0)+\ell_{x_0}\big)e^{-\lambda |t-t_{n_k}|^{\alpha}}.
	\end{equation}
\end{theorem}
Here we use the weighted total variation distance $\rho_{\beta}$ in the convergence to the entrance measure (\eqref{1217-2} and \eqref{1108-1}). We will in Section 3 show that this convergence implies the convergence in both the total variation distance and the Wasserstein distance $\mathcal{W}_1$.

As an application, our results apply to the classical Benzi-Parisi-Sutera-Vulpiani's (BPSV) model: 
 \begin{equation}
 \label{2021aa}
    	dX_t=\big(X_t-X_t^3+f(t)\big)dt+dW_t,
\end{equation}
as long as $f$ is a continuous and bounded function and give a rigorous proof of the existence and uniqueness of entrance measure. 

The SDE \eqref{2021aa} when $f=C\cos t$ was proposed in  Benzi-Parisi-Sutera-Vulpiani {\cite{benzi1982stochastic} as a simplified mathematical model for the physical 
transition between the two climates of ice age and interglacial period and the existence of periodic measure and its ergodicity were studied in \cite{feng2019existence,Feng-Zhao-Zhong2021}. A reasonable intuitive physical explanation of transition at the same frequency   
as the period of the periodic forcing 
was given without a rigorous mathematical analysis in \cite{benzi1982stochastic}. In literature, the transition was observed through the peak in the power spectrum of paleoclimate variations in the last 700,000 years 
at a periodicity of around 100,000 years. 
This is in complementary with smaller peaks at periods of 20,000 and 40,000 years. The major peak represents dramatic climate 
change to a temperature of 10K in Kelvin scale.    This phenomenon was suggested to be related to variations 
in the earth's orbital parameter which also has a similar periodic pattern of changes (Milankovitch \cite{Milankovitch1930}, Hays-Imbrie-Shackleton \cite{HIS76}). But it was pointed out that such studies were able to 
reproduce smaller peaks, but failed to explain the 100,000-year cycle of major peak of sharp change and the introduction of an additive noise in SDE (\ref{2021aa}) offers a model in consistence with the correct physical intuition (Benzi-Parisi-Sutera-Vulpiani {\cite{benzi1982stochastic}).

The periodic model catches large period 
macroscopic climate change, but does not include random oscillation and seasonal change of the weather in the 
microscopic scale. When the latter is taken into considerations, a random quasi-periodic model \eqref{2021aa} when $f$ is quasi-periodic e.g. $f(t)=C_1\cos (\alpha _1t)+\cdots +C_n\cos (\alpha_n t),$ where $\alpha_1, \cdots, \alpha _n$ are rationally independent, should 
give a novel description of periodicities both in climate and weather simultaneously. It should give the model including the major peak of 100,000 years and the smaller peaks of 20,000 years and 40,000 years simultaneously. It is noted that the smaller peaks were not given in BPSV's original model. The concept of quasi-periodic measure gives a statistical formulation for the climate dynamics and enriches BPSV's one priod model.  Under the assumption that the drift is weakly dissipative and locally Lipschitz, and the diffusion is nondegenerate and Lipschitz in the $x$-variable, we prove the existence and uniqueness of the quasi-periodic measure. We obtain the invariant measure of lifted Markovian semigroup and its ergodicity.

We would like to stress that our results go beyond periodic, quasi-periodic or almost-periodic forcing and apply to even more general situations. We also reduce the requirement of uniform condition of homogeneous Markov chain significantly to allow the Lyapunov function fail to contract for many or even majority of time durations, even the local Doeblin condition is not satisfied due to the degeneracy of noise. Two interesting examples that we would like to draw reader's attention are
$$dX_t=(X_t-\sin^+(\sqrt{|t|})X_t^3)dt+\sin^+(\sqrt{|t|})dW_t,$$ 
discussed in Section 5 and an heterogeneous Ornstein-Uhlenbeck process where the convergence to the entrance measure is subgeometric discussed in Section 6.

Our tool is also new and innovative (Theorem \ref{Theorem 1214}), though some ideas come from Harris recurrence approach and Hairer-Mattingly's refinement. Harris's recurrence says if from any position, the process visits any measurable set of positive measure at a finite time, then there is a unique stationary measure. This is equivalent to that there exists a ``small set'' such that if the Markov chain visits the ``small set'' infinitely often and the ``small set'' should be sufficiently large and possess a uniform minorization condition or local Doeblin condition (Baxendale \cite{baxendale2011te}). Meyn-Tweedie \cite{meyn1992stability} observed the usefulness of Lyapunov function in the proof of the recurrence of small set thus extended the applicability of Harris small set method and obtained geometric convergence to invariant probability measure. However, for time-inhomogeneous systems, ``visit infinitely often'' property may be difficult to establish directly, the Lyapunov function may not even be contract at most of the time, as the behaviour of Markovian semigroup may be heterogeneous and cannot be copied from one time duration to another even in law.

Hairer-Mattingly \cite{hairer2011yet} provided the contraction of the one-step Markovian semigroup acting on probability measures. Using this one-step contraction, they obtained the same results of convergence to invariant measures, their ergodicity and mixing rate for time-homogeneous Markov chains. In this paper, we note that this ``within one-step'' argument makes it possible for us to use Harris's condition to time-inhomogeneous Markov processes essentially by allowing the semigroup behaving differently in different time intervals. In fact we observe that at certain intervals the growth of the distance of $P^*\mu_1$ and $P^*\mu_2$ can be controlled by the distance of $\mu_1$ and $\mu_2$ even when the contraction condition  on Lyapunov function and/or the local Doeblin condition may fail, as long as we prepare to allow the growth rate to be bigger than 1. This can be due to the lack of dissipativity and/or the nondegeneracy of noise. Theorem \ref{Theorem 1214} is a very general theorem for the existence of entrance measure and speed of convergence and can apply to many problems.

The local Doeblin condition on a small set is difficult to prove. For this we need lower bound of the densities of transition probabilities when the time duration is strictly away from 0. Though there are many results on the upper bounds of fundamental solutions of parabolic partial differential equations, there are only few results on the lower bound especially as the drift of our SDE is only assumed to be locally Lipschitz and of polynomial growth. Our approach to obtain the lower bound is inspired by the lower bound estimate in the case of global Lipschitz coefficients (Delarue-Menozzi \cite{delarue2010density}, Menozzi-Pesce-Zhang \cite{menozzi2021density}) and some localization and approximation by a sequence of Lipschitz functions. The key is to obtain lower bound in the approximation that is uniform for each step of the approximation. The proof is completed by a combination of stochastic and analytic arguments and a two-point comparison of the solution of Fokker-Planck equation (Bogachev-Krylov-R\"ockner-Shaposhnikov \cite{bogachev2015fokker}). This lower bound result is new and of independent interests. 

In the final section of this paper, we will discuss quasi-periodic measures and their ergodic theory. Ergodic theory was extended to periodic measures in Feng-Zhao \cite{feng2020random}. Extensions to quasi-periodic measures began in Feng-Qu-Zhao \cite{feng2021random} for SDEs under strong dissipative conditions with convergence in total variation distance. In this paper, ergodic theory for SDEs with locally Lipschitz and weakly dissipative conditions is established and the results are applicable to wide range of SDEs in the $\rho_{\beta}$ distance so in both the Wasserstein distance and the total variation distance. Note in both periodic and quasi-periodic cases, lifting to a cylinder is key to establish the ergodic theory. On the cylinder, an 
invariant measure can be constructed. But there are added difficulties in the quasi-periodic case, where SDEs with the unfolding multi-time parent coefficients associated with the quasi-periodic coefficients are not well-posed.  A further time reparameterization of the parent coefficients proposed in \cite{feng2021random} (see (\ref{Equation K_r_1,r_2})) plays some crucial role both in the construction of quasi-periodic measures and the lifting. The reparameterization makes it possible to construct unfolding multi-parameters Markovian semigroups and entrance measures from the SDEs with the multi-parameters coefficients. It is observed that following a Birkhoff's ergodic argument, 
the conditions needed for the parent coefficients in the analysis of the existence of the entrance measures of the multi-parameters SDE \eqref{Equation K_r_1,r_2} are naturally satisfied due to our assumptions on the coefficients for SDE \eqref{SDE}.  It is noted that in the periodic case,  though the reparameterization is naturally there, it was not necessary to do so.

\section{Long time contraction for time-inhomogeneous Markov processes and entrance measures}
Throughout this section, we consider a time-inhomogeneous Markovian transition function $P(t,s,x,\cdot)$ on a metric space $(\mathbb{X},\mathcal{B}(\mathbb{X}))$ which satisfies
\begin{eqnarray*}
  P(t,s,x,\Gamma)&=&\int_{\mathbb{X}}P(r,s,x,dy)P(t,r,y,\Gamma)\\
  &=&\left( P(r,s)\circ P(t,r) \right)1_{\Gamma}(x)\\
  &=&\left( P^*(t,r)\circ P^*(r,s) \right)\delta_x(\Gamma),
\end{eqnarray*}
for all $x\in  \mathbb{X}, \ \Gamma\in \mathcal{B}(\mathbb{X}), \ s\leq r\leq t$. The above can make perfect sense by recalling the following standard notation of semigroup. Let $P(t,s)$ be a semigroup acting as a linear operator on functions $f: \mathbb{X}\to \mathbb{R}$:
\begin{equation}\label{0110-5}
	P(t,s)f(x)=\int_{\mathbb{X}}P(t,s,x,dy)f(y), \ x\in \mathbb{X},
\end{equation}
and $P^*(t,s)$ be its adjoint operator acting on signed measures $\mu$ on $(\mathbb{X},\mathcal{B}(\mathbb{X}))$:
\begin{equation*}
	\label{Define measure transition P^*}
	P^*(t,s)\mu(\Gamma)=\int_{\mathbb{X}}P(t,s,x,\Gamma)\mu(dx), \ \Gamma\in\mathcal{B}(\mathbb{X}),
\end{equation*}
whenever the integrals exist. Then
\begin{equation*}
	P(r,s)\circ P(t,r)=P(t,s),\ \ P^*(t,r)\circ P^*(r,s)=P^*(t,s), \ s\leq r\leq t.
\end{equation*}
Also, for any signed measure $\mu$ on $(\mathbb{X},\mathcal{B}(\mathbb{X}))$ and any function $f: \mathbb{X}\to \mathbb{R}$ we write $\langle f, \mu \rangle:=\int f d\mu$, whenever the integral exists. Note that
\begin{eqnarray*}
  \langle P(t,s)f, \mu \rangle&=&\int_{\mathbb{X}}P(t,s)f(x)\mu(dx)\\
	&=&\int_{\mathbb{X}}f(x)P^*(t,s)\mu(dx)\\
	&=&\langle f, P^*(t,s)\mu \rangle,
\end{eqnarray*}
which means $P^*(t,s)$ is the dual operator of $P(t,s)$. 

Denote by $\mathcal{P}(\mathbb{X})$ the collection of probability measures on $(\mathbb{X},\mathcal{B}(\mathbb{X}))$. Note that invariant measure is an essential topic in the study of time-homogeneous Markovian transition semigroup. But in time-inhomogeneous case, ``invariant measure'' is replaced by the following so called ``entrance measure'' (Dynkin \cite{dynkin1978}).  
\begin{definition}\label{Def 1}
 	We say a measure-valued map $\mu: \mathbb{R}\rightarrow \mathcal{P}(\mathbb{X})$ is an entrance measure of a time-inhomogeneous Markovian semigroup $P(\cdot,\cdot)$ (or $P$) if 
	\begin{equation*}
	P^*(t,s)\mu_s=\mu_t
	\end{equation*} for all $t\geq s, s\in \mathbb{R}$. Moreover, if the corresponding Markovian transition function $P(t,s,x,\cdot)$ is the transition kernel of an SDE, we also say that this SDE has an entrance measure $\mu$.
\end{definition}

In the spirit of the proof of Theorem 1.3 in \cite{hairer2011yet}, for a discrete Markovian transition kernel $P(x,\cdot)$, we have the following lemma.

\begin{lemma}
	\label{lemma 1124-1}
	If there exist a function $V: \mathbb{X}\to [0,\infty)$, a probability measure $\nu\in \mathcal{P}(\mathbb{X})$ and nonnegative constants $\gamma, K, \eta, R$ such that
	\begin{equation*}
		PV(x)\leq \gamma V(x)+K, \ \text{ for all } x\in \mathbb{X} \ \text{ and } \ \inf_{V(x)\leq R}P(x,\cdot)\geq \eta\nu(\cdot),
	\end{equation*}
	then for any $\beta\geq 0$ and $\mu_1, \mu_2\in \mathcal{P}(\mathbb{X})$, we have
	\begin{equation*}
		\rho_{\beta}(P^*\mu_1, P^*\mu_2)\leq \zeta \rho_{\beta}(\mu_1, \mu_2),
	\end{equation*}
	where
	\begin{equation*}
		\zeta=\max \bigg\{1-\eta+\beta K, \ \frac{2+\beta(\gamma R+2K)}{2+\beta R}\bigg\}.
	\end{equation*}
\end{lemma}
This lemma extends the result in \cite{hairer2011yet} by relaxing conditions on $\gamma,\eta,K$ to include that  $\eta=0, \gamma>1$ and $P^*$ to be expanding or contracting instead of contracting only considered in \cite{hairer2011yet}. The point for this flexibility is to allow the semigroup to expand and local Doeblin condition to fail if we prepare to let the constant $\zeta$ not limited to the case $\zeta<1$ (contraction case) only. But the formula for the growth constant $\zeta$ is the same as that of the contraction rate in \cite{hairer2011yet}. The proof of Lemma \ref{lemma 1124-1} is essentially the same as that of Theorem 1.3 in \cite{hairer2011yet}, so we omit it.

Under Assumption \ref{A1}, Lemma \ref{lemma 1124-1} shows that for any partition $t=t_0>t_1>\cdots>t_n=s$, $\beta\geq 0$ and $\mu_1,\mu_2\in \mathcal{P}(\mathbb{X})$,
\begin{equation}
	\label{1217-1}
	\rho_{\beta}\big(P^*(t,s)\mu_1, P^*(t,s)\mu_2\big)\leq \biggl(\prod_{i=1}^n\zeta_{\beta}(t_{i-1},t_{i}) \biggr) \rho_{\beta}(\mu_1, \mu_2),
\end{equation}
where 
\begin{equation}\label{0926-1}
	\zeta_{\beta}(t_{i-1},t_{i})=\max \bigg\{1-\eta(t_{i-1},t_{i})+\beta K(t_{i-1},t_{i}), \ \frac{2+\beta\big(\gamma(t_{i-1},t_{i}) R+2K(t_{i-1},t_{i})\big)}{2+\beta R}\bigg\}.
\end{equation}
Note that 
$$\lim_{\beta\to 0}\bigl(1-\eta(t_{i-1},t_{i})+\beta K(t_{i-1},t_{i})\bigr)=1-\eta(t_{i-1},t_{i})$$
and
$$\lim_{\beta\to 0}\frac{2+\beta\big(\gamma(t_{i-1},t_{i}) R+2K(t_{i-1},t_{i})\big)}{2+\beta R}=1,$$
so if $\eta(t_{i-1},t_{i})>0$, then $\zeta_{\beta}(t_{i-1},t_{i})$ equals the second term in the maximum of \eqref{0926-1} for sufficiently small $\beta$. This observation essentially leads to the following theorem with some careful argument.

Now we give the proof of Theorem \ref{Theorem 1214}.
\begin{proof}[Proof of Theorem \ref{Theorem 1214}]
	(i). By the definition of $\zeta_{\beta}(t_{i-1},t_{i})$ in \eqref{0926-1}, it is easy to check that when $\beta\leq \beta_1:=\min\big\{\frac{\delta}{2K}, \frac{2\delta}{R(2-\delta)}\big\}$, then for all $i\in A^{\delta}:=\cup_{n=1}^{\infty}A_n^{\delta}$,
	\begin{equation*}
		\frac{2+\beta\big(\gamma(t_{i-1},t_{i}) R+2K(t_{i-1},t_{i})\big)}{2+\beta R}\geq 1-\frac{\delta}{2}\geq 1-\eta(t_{i-1},t_{i})+\beta K(t_{i-1},t_{i}),
	\end{equation*}
	so
	\begin{equation*}
		\zeta_{\beta}(t_{i-1},t_{i})=\frac{2+\beta\big(\gamma(t_{i-1},t_{i}) R+2K(t_{i-1},t_{i})\big)}{2+\beta R}\leq \frac{2+\beta\big(\gamma(t_{i-1},t_{i}) R+2K\big)}{2+\beta R}.
	\end{equation*}
	On the other hand, for all $i\notin A^{\delta}$,
	\begin{equation*}
		\zeta_{\beta}(t_{i-1},t_{i})\leq \max\bigg\{1+\beta K, \frac{2+\beta(\gamma R+2K)}{2+\beta R}\bigg\}\leq 1+\frac{(\gamma-1)^+R+2K}{2}\beta.
	\end{equation*}
	Now consider $t\in \mathbb{R}$, there exists a smallest $i_0(t)$ such that $t\geq t_{i_0(t)}$. For all $n> i_0(t)$, define
	\begin{equation*}
		A_n^{\delta}(t):=\{i_0(t)<i\leq n: i\in A_n^{\delta}\}, \ n^{\delta}(t):=\# A_n^{\delta}(t), \text{ and } \bar{\gamma}^{\delta}_{n}(t):=\frac{1}{n^{\delta}(t)}\sum_{i\in A_n^{\delta}(t)}\gamma(t_{i-1},t_i).
	\end{equation*}
	It is easy to see that
	\begin{equation*}
		\liminf_{k\to \infty}\frac{n_k^{\delta}(t)}{n_k}=\liminf_{k\to \infty}\frac{n_k^{\delta}}{n_k}> \varpi \text{ and } \limsup_{k\to \infty}\bar{\gamma}^{\delta}_{n_k}(t)=\limsup_{k\to \infty}\bar{\gamma}^{\delta}_{n_k}<\gamma^*.
	\end{equation*}
	Then there exists $k_0(t)$ such that for all $k\geq k_0(t)$, we have $n_k^{\delta}(t)\geq \varpi n_k$ and $\bar{\gamma}^{\delta}_{n_k}(t)\leq \gamma^*$. In this case, applying \eqref{1217-1} on $t\geq t_{i_0(t)}>t_{i_0(t)+1}>\cdots>t_{n_k}$ gives for any $\beta\leq \beta_1$ and $\mu_1,\mu_2\in \mathcal{P}(\mathbb{X})$
	\begin{equation}
		\label{0918-1}
		\begin{split}
			&\quad \ \rho_{\beta}\big(P^*(t,t_{n_k})\mu_1, P^*(t,t_{n_k})\mu_2\big)\\
			&\leq \zeta_{\beta}(t,t_{i_0(t)})\bigg(1+\frac{(\gamma-1)^+R+2K}{2}\beta\bigg)^{n_k-i_0(t)-n_k^\delta(t)}\\
      &\quad \ \times\prod_{i\in A^{\delta}_{n_k}(t)} \frac{2+\beta(\gamma(t_{i-1},t_i)R+2K)}{2+\beta R}\rho_{\beta}(\mu_1, \mu_2)\\
			&\leq \zeta_{\beta}(t,t_{i_0(t)})\bigg(1+\frac{(\gamma-1)^+R+2K}{2}\beta\bigg)^{(1-\varpi)n_k-i_0(t)}\\
      &\quad \ \times\biggl(\frac{2+\beta(\bar{\gamma}_{n_k}^{\delta}(t)R+2K)}{2+\beta R}\biggr)^{n_k^{\delta}(t)}\rho_{\beta}(\mu_1, \mu_2),
		\end{split}
	\end{equation}
	where we have used the inequality of arithmetic and geometric means in the second inequality in \eqref{0918-1}. Note that $\varpi>\frac{(\gamma-1)^+R+2K}{(\gamma-1)^+R+(1-\gamma^*)R}$ is equivalent to
	\begin{equation*}
		\frac{(\gamma-1)^+R+2K}{2}(1-\varpi)<\frac{(1-\gamma^*)R-2K}{2}\varpi.
	\end{equation*}
	Hence there exists $\beta_2\in (0,\beta_1)$ such that 
	\begin{equation}\label{1111-2}
		\frac{(\gamma-1)^+R+2K}{2}(1-\varpi)<\frac{(1-\gamma^*)R-2K}{2+\beta_2R}\varpi.
	\end{equation}
	Note also that $(1-\gamma^*)R>2K$, then for any $\beta\leq \beta_2$,
	\begin{equation}\label{1111-1}
		\begin{split}
      \frac{2+\beta(\bar{\gamma}_{n_k}^{\delta}(t)R+2K)}{2+\beta R}&\leq \frac{2+\beta(\gamma^*R+2K)}{2+\beta R}\\
      &=1-\frac{(1-\gamma^*)R-2K}{2+\beta R}\beta\\
      &\leq 1-\frac{(1-\gamma^*)R-2K}{2+\beta_2 R}\beta.
    \end{split}
	\end{equation}
	then it follows from \eqref{0918-1} and \eqref{1111-1} that for any $\beta\leq \beta_2$,
	\begin{equation}\label{1111-3}
		\begin{split}
			&\quad \ \rho_{\beta}\big(P^*(t,t_{n_k})\mu_1, P^*(t,t_{n_k})\mu_2\big)\\
      &\leq \zeta_{\beta}(t,t_{i_0(t)})(1+c_1\beta)^{(1-\varpi)n_k-i_0(t)} (1-c_2\beta)^{\varpi n_k}\rho_{\beta}(\mu_1, \mu_2),
		\end{split}
	\end{equation}
	where
	\begin{equation*}
		c_1=\frac{(\gamma-1)^+R+2K}{2}, \ \ c_2=\frac{(1-\gamma^*)R-2K}{2+\beta_2R}.
	\end{equation*}
	Now let us consider function $\phi(x)=(1+c_1x)^{1-\varpi}(1-c_2 x)^{\varpi}$. It is easy to check from \eqref{1111-2} that $\phi(0)=1$ and $\phi'(0)=(1-\varpi)c_1-\varpi c_2<0$. 
	
	Thus there exists $\beta\in (0,\beta_2)$ such that $0<\phi(\beta):=r<1$.
	Then \eqref{0929-1} follows from \eqref{1111-3} by setting 
	\begin{equation}\label{0929-2}
		\begin{split}
			C_t:&=\max\Big\{\zeta_{\beta}(t,t_{i_0(t)})(1+c_1\beta)^{-i_0(t)},\zeta_{\beta}(t,t_{i_0(t)})(1+c_1\beta)^{n_{k_0(t)}-i_0(t)}r^{-n_{k_0(t)}}\Big\}\\
			&=\zeta_{\beta}(t,t_{i_0(t)})(1+c_1\beta)^{n_{k_0(t)}-i_0(t)}r^{-n_{k_0(t)}}.
		\end{split}
	\end{equation}

	(ii). If in addition $\ell_{x_0}:=\sup_{j\geq i\geq 1}P(t_{n_i},t_{n_j})V(x_0)<\infty$, we have for $t\geq t_{n_i}\geq t_{n_j}$
	\begin{equation*}
		\begin{split}
			\rho_{\beta} \big(P(t,t_{n_i},x_0,\cdot), P(t,t_{n_j},x_0,\cdot)\big)
			&\leq \rho_{\beta} \big(P^*(t,t_{n_i})\delta_{x_0}, P^*(t,t_{n_i})P(t_{n_i},t_{n_j},x_0,\cdot)\big)\\
			&\leq C_t r^{n_i}\rho_{\beta} \big(\delta_{x_0}, P(t_{n_i},t_{n_j},x_0,\cdot)\big)\\
			&\leq \big(2+\beta V(x_0)+\beta \ell_{x_0}\big)C_t r^{n_i} \to 0 \text{ as } i,j\to \infty.
		\end{split}
	\end{equation*}
	Hence there exists $\mu_t^{x_0}\in \mathcal{P}(\mathbb{X})$ such that
	\begin{equation}\label{1206-1}
		\rho_{\beta} \big(P(t,t_{n_i},x_0,\cdot), \mu_t^{x_0}\big)
		\leq \big(2+\beta V(x_0)+\beta \ell_{x_0}\big)C_tr^{n_i}.
	\end{equation}
	Moreover, for all $t\geq s$,
	\begin{equation*}
		P^*(t,s)\mu_s^{x_0}=\lim_{i\to \infty}P^*(t,s)P(s,t_{n_i},x_0,\cdot)=\lim_{i\to \infty}P(t,t_{n_i},x_0,\cdot)=\mu_t^{x_0},
	\end{equation*}
	and
	\begin{equation*}
		\begin{split}
			\sup_{i\in \mathbb{N}^+}\int_{\mathbb{X}}V(y)\mu_{t_{n_i}}^{x_0}(dy)&\leq\sup_{i\in \mathbb{N}_+}\liminf_{j\to \infty}\int_{\mathbb{X}}V(y)P(t_{n_i},t_{n_j},x_0,dy)\\
			&=\sup_{i\in \mathbb{N}_+}\liminf_{j\to \infty}P(t_{n_i},t_{n_j})V(x_0)\leq \ell_{x_0}.
		\end{split}
	\end{equation*}
	Therefore $\mu_{\cdot}^{x_0}$ is an entrance measure of Markovian semigroup $P$ in $\mathcal{M}_{\mathcal{A}}$ and for any $x\in \mathbb{X}$,
	\begin{equation*}
		\begin{split}
			\rho_{\beta} \big(P(t,t_{n_i},x,\cdot), \mu_t^{x_0}\big)
			&\leq \rho_{\beta} \big(P(t,t_{n_i},x,\cdot), P(t,t_{n_i},x_0,\cdot)\big)+\rho_{\beta} \big(P(t,t_{n_i},x_0,\cdot), \mu_t^{x_0}\big)\\
			&\leq \big(4+\beta V(x)+2\beta V(x_0)+\beta\ell_{x_0}\big)C_tr^{n_i},
		\end{split}
	\end{equation*}
	which gives \eqref{1217-2}.

	It remains to show the uniqueness of entrance measure. Suppose that there are two entrance measures $\mu^1,\mu^2$ in $\mathcal{M}_{\mathcal{A}}$. Denote $M_i:=\sup_{t\in \mathcal{A}}\int_{\mathbb{X}}V(x)\mu^i_{t}(dx)<\infty, \ i=1,2$. Then it follows that 
	\begin{equation*}
		\begin{split}
			\rho_{\beta}(\mu^1_t,\mu^2_t)&\leq \liminf_{j\to \infty} \rho_{\beta}\bigl(P^*(t,t_{n_j})\mu^1_{t_{n_j}}, P^*(t,t_{n_j})\mu^2_{t_{n_j}}\bigr)\\
			&\leq \liminf_{j\to \infty}C_tr^{n_j}\rho_{\beta} \big(\mu^1_{t_{n_j}}, \mu^2_{t_{n_j}}\big)\\
			&\leq \liminf_{j\to \infty} (2+\beta M_1+\beta M_2)C_tr^{n_j}=0.
		\end{split}
	\end{equation*}
	Then the uniqueness follows.
\end{proof}

\begin{remark}\label{remark 1114}
	From the proof of Theorem \ref{Theorem 1214}, it is easy to see that $\gamma, K$ in \eqref{1205-3} can be replaced by
	\begin{equation*}
		\gamma:=\limsup_{n\to \infty}\gamma(t_{n-1},t_n), K:=\limsup_{n\to \infty}K(t_{n-1},t_n).
	\end{equation*}
	Moreover, condition \eqref{0929-3} can be replaced by one of the following conditions:
	\begin{itemize}
		\item $\liminf_{n\to \infty}\frac{n^{\delta}}{n}>\varpi$ and $\liminf_{n\to \infty}\bar{\gamma}^{\delta}_{n}< \gamma^*$;
		\item $\limsup_{n\to \infty}\frac{n^{\delta}}{n}>\varpi$ and $\limsup_{n\to \infty}\bar{\gamma}^{\delta}_{n}< \gamma^*$.
	\end{itemize}
	This is because in both cases, there exists a subsequence $\{t_{n_k}\}_{k\geq 1}$ such that \eqref{0929-3} holds.
\end{remark}

Next we give the proof of Theorem \ref{Corollary 1101}.

\begin{proof}[Proof of Theorem \ref{Corollary 1101}]
	Note that there exist $k_0>0, a>0$ such that for all $k> k_0$, we have
	\begin{equation*}
		n_k\geq a\phi(t-t_{n_k}).
	\end{equation*}
	Hence \eqref{1027-2} and \eqref{1027-1} follows from \eqref{0929-1} and \eqref{1217-2} by letting $\lambda=-a\ln r>0$ for $k> k_0$. In the case that $k\leq k_0$, \eqref{1027-2} and \eqref{1027-1} still hold when we replace $C_t$ by $C_te^{\lambda \phi(t-t_{n_{k_0}})}$.
\end{proof}

\begin{remark}\label{remark 1108}
	If $\limsup_{n\geq 0}|t_{n}-t_{n+1}|<\infty$, then $\phi$ can be choosen as $\phi_1$ and we obtain geometric convergence. In fact, if we choose $b>\limsup_{n\geq 0}|t_{n}-t_{n+1}|$, it is easy to see that
	\begin{equation*}
		\liminf_{n\to \infty}\frac{n}{t-t_n}>\frac{1}{b}>0.
	\end{equation*}
 \end{remark}

This theorem is applicable in a wide variety of situations of great generalities. It allows us to break the time horizon into many intervals that may be of unequal lengths.
Moreover, two measures may contract in one interval and expand in another and they may contract or expand at different rates in different intervals. However, 
the system can still contract over the infinite time horizon to a unique entrance measure. 
See Example 
\ref{2022a} and Remark \ref{2022b} for some interesting 
examples. 

As a special case, SDEs satisfying the following Assumption \ref{A2} form an important class of stochastic systems arising in real world problems e.g. the stochastic resonance model of BPSV. Under this condition, we can obtain that $\bar{\zeta}_{\beta}(t,s)$ tends to 0 exponentially as $|t-s|\to \infty$ for some proper $\beta>0$. We thus can prove the geometric convergence of Markovian transition function $p(t,s,x,\cdot)$ as $s\to -\infty$ to the entrance measure $\mu_t, t\in \mathbb{R}$ in $\mathcal{M}$ and the uniqueness of entrance measure in $\mathcal{M}$. 

\begin{condition}
	\label{A2}
	Functions $\gamma(t,s), K(t,s),\eta(t,s)$ and constant $R$ in Assumption \ref{A1} satisfy the following condition:
	\begin{description}
		\item[(i)] There exists a non-decreasing function $h:\mathbb{R}^+\to \mathbb{R}^+$ such that 
		\begin{equation*}
			\max\{\gamma(t,s), K(t,s)\}\leq h(|t-s|).
		\end{equation*}

		\item[(ii)] There exist positive constants $\Delta, \gamma_{\Delta}, \eta_{\Delta}$ such that for all $t\in \mathbb{R}$,
		\begin{equation*}
			\gamma(t+\Delta, t)\leq \gamma_{\Delta}<1, \ \ R>\frac{2h(\Delta)}{1-\gamma_{\Delta}} \ \text{ and } \ \eta(t+\Delta,t)\geq \eta_{\Delta}.
		\end{equation*}
	\end{description}
\end{condition}

\begin{corollary}
	\label{Theorem of geometric ergodic of time-inhoogeneous}
	Suppose Assumptions \ref{A1} and \ref{A2} hold. 
	Then there exist $\beta=\frac{\eta_{\Delta}}{2h(\Delta)}$ and constants $C>0,\lambda>0$ depending only on $(\Delta,\gamma_{\Delta}, h(\Delta),\eta_{\Delta}, R)$ such that for any $\mu_1, \mu_2\in \mathcal{P}(\mathbb{X})$ and for all $s\leq t$
	\begin{equation}
		\label{1218-1}
		\rho_{\beta}\bigl(P^*(t,s)\mu_1, P^*(t,s)\mu_2\bigr)\leq Ce^{-\lambda (t-s)}\rho_{\beta}(\mu_1, \mu_2).
	\end{equation}
	Moreover, there exists a unique entrance measure $\mu$ of $P$ in $\mathcal{M}$ such that for any $t\geq s$, $x\in \mathbb{X}$
	\begin{equation}
		\label{1218-3}
		\rho_{\beta}\bigl(P(t,s,x,\cdot), \mu_t\bigr)\leq C(1+V(x))e^{-\lambda(t-s)}.
	\end{equation}
\end{corollary}

\begin{proof}
	We first prove \eqref{1218-1}. Note that Lemma \ref{lemma 1124-1} yields for $\beta=\frac{\eta_{\Delta}}{2h(\Delta)}$, there exists
	$$\zeta=\max \bigg\{1-\eta_{\Delta}+\beta h(\Delta), \ \frac{2+\beta(\gamma_{\Delta} R+2h(\Delta))}{2+\beta R}\bigg\}<1$$
	such that for all $t\in \mathbb{R}$ and $\mu_1, \mu_2\in \mathcal{P}(\mathbb{X})$
	\begin{equation}
		\label{Ineq contraction of P^*}
		\rho_{\beta} \big(P^*(t+\Delta, t)\mu_1, P^*(t+\Delta, t)\mu_2\big) \leq  \zeta\rho_{\beta}(\mu_1, \mu_2).
	\end{equation}
	For $m\geq 0$, by iterating $m$-times the inequality \eqref{Ineq contraction of P^*}, we know that for all $t\in \mathbb{R}$,
	\begin{equation*}
		\rho_{\beta} \big(P^*(t+m\Delta, t)\mu_1, P^*(t+m\Delta, t)\mu_2\big) \leq  \zeta^m\rho_{\beta}(\mu_1, \mu_2).
	\end{equation*}
	On the other hand, for all $0\leq t-s\leq \Delta$, we have
	\begin{equation*}
		\begin{split}
			\zeta(t,s)&:=\max \bigg\{1-\eta(t,s)+\beta K(t,s), \ \frac{2+\beta(\gamma(t,s) R+2K(t,s))}{2+\beta R}\bigg\}\\
			&\leq \max \bigg\{1+\beta h(\Delta), \ \frac{2+\beta h(\Delta)(2+R)}{2+\beta R}\bigg\}=:\zeta_0
		\end{split}
	\end{equation*}
	Hence for any $t\geq s$, let $m$ be the integer such that $m\Delta\leq t-s<(m+1)\Delta$, we conclude that
	\begin{equation*}
		\begin{split}
			&\rho_{\beta}\bigl(P^*(t,s)\mu_1, P^*(t,s)\mu_2\bigr)\\
			&= \rho_{\beta}\bigl(P^*(t,s+m\Delta)\circ P^*(s+m\Delta,s)\mu_1, P^*(t,s+m\Delta)\circ P^*(s+m\Delta,s)\mu_2\bigr)\\
			&\leq \zeta_0\rho_{\beta}\bigl(P^*(s+m\Delta,s)\mu_1, P^*(s+m\Delta,s)\mu_2\bigr)\\
			&\leq \zeta_0\zeta^m\rho_{\beta}\bigl(\mu_1, \mu_2\bigr).
		\end{split}
	\end{equation*}
	Then \eqref{1218-1} follows by setting $C=\frac{\zeta_0}{\zeta}$ and $\lambda=-\frac{\ln \zeta}{\Delta}$.

	To check assumption \eqref{0918-4}, note that from Assumption \ref{A2} {\bf (i)} and semigroup relation 
	\begin{equation*}
		P(s+m\Delta,s)=P(s+m\Delta,s+(m-1)\Delta)\circ P(s+(m-1)\Delta,s), \ \text{ for all } m\geq 1,
	\end{equation*}
	we have
	\begin{equation}\label{1114-2}
		\begin{split}
			P(s+m\Delta,s)V(x)&\leq \gamma_{\Delta}P(s+(m-1)\Delta,s)V(x)+h(\Delta)\\
			&\leq \gamma_{\Delta}^2P(s+(m-2)\Delta,s)V(x)+h(\Delta)(1+\gamma_{\Delta})\\
			&\cdots\\
			&\leq \gamma_{\Delta}^mV(x)+h(\Delta)(1+\gamma_{\Delta}+\cdots+\gamma_{\Delta}^{m-1})\\
			&\leq V(x)+\frac{h(\Delta)}{1-\gamma_{\Delta}}.
		\end{split}
	\end{equation}
	Then for any $t\geq s$ with $m\Delta\leq t-s<(m+1)\Delta$,
	\begin{equation}\label{0918-5}
		P(t,s)V(x)=P(t,s+m\Delta)P(s+m\Delta)V(x)\leq h(\Delta)V(x)+\frac{h^2(\Delta)}{1-\gamma_{\Delta}}+h(\Delta).
	\end{equation}
	Now comparing condition \eqref{0918-4} with \eqref{0918-5} and using the similar method as in the proof of Theorem \ref{Theorem 1214}, we know that there exists a unique entrance measure $\mu$ of $P$ in $\mathcal{M}$ such that \eqref{1218-3} holds.
\end{proof}

	Let $\mu$ be a signed measure and $\mu_1,\mu_2$ be two probability measures on $\mathbb{X}$.
	Recall that the total variation norm of $\mu$ is given by 
	$$\|\mu\|_{TV}:=\mu^+(\mathbb{X})+\mu^-(\mathbb{X}),$$ 
	where $\mu=\mu^+-\mu^-$ is the Jordan decomposition of $\mu$, and the 1-Wasserstein distance of $\mu_1$ and $\mu_2$ is given by 
	\begin{equation*}
		\mathcal{W}_{1}(\mu_1,\mu_2):=\sup_{\pi\in \mathcal{C}(\mu_1,\mu_2)}\int_{\mathbb{X}}|x-y|\pi(dx,dy),
	\end{equation*}
	where $\mathcal{C}(\mu_1,\mu_2)$ is the set of all couplings of $\mu_1$ and $\mu_2$ with marginal distributions $\mu_1$ and $\mu_2$.

	For any $\beta>0$, it is easy to see that $\|\mu_1-\mu_2\|_{TV}\leq \rho_{\beta}(\mu_1,\mu_2)$. 
	
	On the other hand, the dual representation of $\mathcal{W}_1$ gives
	\begin{equation*}
		\begin{split}
			\mathcal{W}_{1}(\mu_1,\mu_2)&=\sup_{|f|_{Lip}\leq 1}\bigg|\int_{\mathbb{X}}f(x)\mu_1(dx)-\int_{\mathbb{X}}f(x)\mu_2(dx)\bigg|\\
			&=\sup_{|f|_{Lip}\leq 1}\bigg|\int_{\mathbb{X}}\big(f(x)-f(0)\big)(\mu_1-\mu_2)(dx)\bigg|\\
			&\leq \int_{\mathbb{X}}|x||\mu_1-\mu_2|(dx),
		\end{split}
	\end{equation*}
	where $|f|_{Lip}$ denotes the minimal Lipschitz constant for the Lipschitz continuous function $f$.
	If we assume $V$ satisfies that $\liminf_{|x|\to \infty}\frac{V(x)}{|x|}>0$, then
	there exists a constant $c_{\beta,V}>0$ such that $c_{\beta,V}|x|\leq 1+\beta V(x)$. Hence we have $c_{\beta,V}\mathcal{W}_{1}(\mu_1,\mu_2)\leq \rho_{\beta}(\mu_1,\mu_2)$ and the following remark is easy to understand.

\begin{remark}\label{New remark}
	If for any $\beta>0$, there exists $c_{\beta,V}>0$ such that $c_{\beta,V}|x|\leq 1+\beta V(x)$, then all the convergences to the entrance measures in the distance $\rho_{\beta}$ in Theorems \ref{Theorem 1214}, \ref{Corollary 1101} and Corollary \ref{Theorem of geometric ergodic of time-inhoogeneous} and also those appearing in the following sections imply the convergences in the total variation distance and the Wasserstein distance.

	In particular, if $V(x)=|x|^2$, we have $c_{\beta,V}=2\sqrt{\beta}$.
\end{remark}

\section{Lower bound estimates for the density of Markovian transition probability}

 In light of Theorem \ref{Theorem 1214}, it is important to check Assumption \ref{A1} for problems concerned. For this, we need lower bound estimate for the density of Markovian transition probability given in Theorem \ref{Theorem of lower bound density unsmooth coefficient}. We will provide a full proof of this result in this section. However, the proof is quite long and not directly used in the rest of this paper. The reader can skip the proof and read the other sections first in order not to interrupt the flow of the reading. Note Theorem \ref{Theorem of lower bound density unsmooth coefficient} will be used in other sections and the proof can be read independently.

 Consider the following stochastic differential equation on $\mathbb{R}^d$:
 \begin{equation}
 \label{SDE}
 \begin{cases}
 dX(t)=b(t, X(t))dt+\sigma(t, X(t))dW_t,  \quad t\geq s,\\
 X_s=x,
 \end{cases}
 \end{equation}
 where $b: \mathbb{R}\times \mathbb{R}^d\rightarrow \mathbb{R}^d, \ \sigma: \mathbb{R}\times \mathbb{R}^d\rightarrow \mathbb{R}^{d\times d},$ are continuous functions, $W_t, t\in \mathbb{R}$ is a two-sided $\mathbb{R}^d$-valued Brownian motion on a probability space $(\Omega,\mathcal{F},\mathbf{P})$ with $W_0=0$ and $W_t-W_s$ being a standard Gaussian distribution $\mathcal{N}(0,(t-s)I_{d\times d})$, where $I_{d\times d}$ is the $d\times d$ unit matrix.

 In this paper, we impose the following assumptions on the coefficients $b, \sigma$:
 \begin{condition}
  \label{A3}
  \begin{description}
    \item [(i)] The diffusion coefficient matrix $\sigma$ is continuous and there exists a positive constant $\Gamma_1\geq 1$ such that
		\begin{equation}
			\label{Ineq of non-degenerate diffusion}
			\langle \sigma\sigma^{\top}(t,x)\xi, \xi \rangle\leq \Gamma_1|\xi|^2, \ x, \xi\in \mathbb{R}^d,\ t\in \mathbb{R},
		\end{equation}
		and for all $x,y\in \mathbb{R}^d$,
		\begin{equation}
			\label{my1127}
			\|\sigma(t,x)-\sigma(t,y)\|\leq \Gamma_1 |x-y|,
		\end{equation}
		where $\|A\|:=\sqrt{Tr[AA^{\top}]}$ for all $A\in \mathbb{R}^{d\times d}$.
    \item [(ii)] (Locally Lipschitz, polynomial growth and coercivity). The drift vector $b(t,x)$ is continuous and locally Lipschitz continuous in $x$ uniformly in $t$ and there exist $\kappa\geq 1, \Gamma_2>0$ such that 
	\begin{equation}
		\label{Ineq polynomial growth}
		|b(t,x)|\leq \Gamma_2(1+|x|^{\kappa}).
	\end{equation}
	Moreover, there exist continuous functions $\alpha: \mathbb{R}\to \mathbb{R}$ and $\Lambda: \mathbb{R}\to \mathbb{R}^+$ such that for all $x\in \mathbb{R}^d$ and $t\in \mathbb{R}$,
	\begin{equation}
		\label{Ineq weakly coercivity}
		\langle x, b(t,x)\rangle \leq \alpha_t |x|^2+\Lambda_t,
	\end{equation}
	and there exists a non-decreasing function $g: \mathbb{R}^+\to \mathbb{R}^+$ such that for all $s\leq t$
	\begin{equation}
		\label{1228-1}
		\int_{s}^{t}(\alpha^+_r+\Lambda_r)dr\leq g(t-s),
	\end{equation}
	where $\alpha^+:=\alpha\vee 0$.
    \end{description}
  \end{condition}

  \begin{remark}
	  The requirement \eqref{1228-1} is standard to be satisfied in many problems.
  \end{remark}
  In order to give a lower bound density of the solution to \eqref{SDE}, the following nondegenerate diffusion coefficient is needed.
  \begin{condition}\label{A3 new}
	The diffusion coefficient matrix $\sigma$ satisfies the following nondegenerate condition: for all $t\in \mathbb{R}$,
	\begin{equation}
		\label{Ineq of non-degenerate diffusion 2}
		\langle \sigma\sigma^{\top}(t,x)\xi, \xi \rangle\geq \Gamma_1^{-1}|\xi|^2, \ x, \xi\in \mathbb{R}^d,
	\end{equation}
	where $\Gamma_1$ is the same as in \eqref{Ineq of non-degenerate diffusion}.
  \end{condition}
  \begin{remark}
	  Assumption \ref{A3 new} is only used to estimate the lower bound density of the solution to \eqref{SDE} (see Theorem \ref{Theorem of lower bound density unsmooth coefficient}). In fact, the condition \eqref{Ineq of non-degenerate diffusion 2} is only needed to hold for $t$ in a ``relatively small'' subset of $\mathbb{R}$ in our main Theorem \ref{Theorem existence and uniqueness of entrance measure} (see Assumption \ref{A6}).
  \end{remark}
  Note that \eqref{Ineq of non-degenerate diffusion} shows that all the eigenvalues of $\sigma\sigma^{\top}(t,x)$ will lie in $[0, \Gamma_1]$. Hence
  $$\|\sigma(t,x)\|:=\sqrt{Tr[\sigma\sigma^{\top}(t,x)]}\in \left[0, \sqrt{d\Gamma_1}\right].$$

  Under Assumption \ref{A3}, it can be proved that SDE \eqref{SDE} has a unique solution $\{X_t^{s,x}\}_{t\geq s}$ (c.f. Theorem \ref{Theorem of uniformly bounded solution}). Moreover, the unique solution will be a Markov process, which induces a (time-inhomogeneous) Markovian transition probability. If, moreover, Assumption \ref{A3 new} holds, according to Theorem 4.5 in \cite{feng2021random} up to a slight modification in its proof, solution $X_t^{s,x}$ admits a density $p(t,s,x,y)$, i.e. for all $ A\in \mathcal{B}(\mathbb{R}^d)$,
  \begin{equation*}
  	\mathbf{P}[X_t^{s,x}\in A]=\int_A p(t,s,x,y)dy.
  \end{equation*}

  In this section, we will show that the transition density $p(t,s,x,y)$ has a uniform lower bound on any compact set. This is the key in the analysis to use Theorem \ref{Theorem 1214}, in particular in verification of Assumption \ref{A1} {\bf (ii)} for some $\bar{\eta}: \mathbb{R}^+\to (0,1)$ such that $\eta(t,s)\geq \bar{\eta}(|t-s|)$. It is already well-known that the density $p$ is smooth and has upper and lower bound if the coefficients $b, \sigma$ are smooth and globally Lipschitz continuous. Now we need to prove some similar results on some deterministic regularized flow associated with the drift $b$ satisfying Assumption \ref{A3}. First we construct $b_N$ as follows:
  \begin{equation}
	\label{Smooth analysis of b 1}
	b_N(t,x)=
	  \begin{cases}
		  b(t,x), & |x|\leq N,\\
		  b\Big(t,\frac{N}{|x|}x\Big), & |x|>N.
	  \end{cases}
  \end{equation}
  Here $N\in \mathbb{N}$. Since $b(t,\cdot)$ is locally Lipschitz continuous uniformly in $t$, it is easy to check that $b_N(t,\cdot)$ is globally Lipschitz continuous with some Lipschitz constant $\ell_N$. On the other hand, for any $|x|>N$,
  \begin{equation}
	\label{0715-3}
	  \begin{split}
		\langle x, b_N(t,x)\rangle&=\frac{|x|}{N}\left\langle \frac{N}{|x|}x, b\left(t,\frac{N}{|x|}x\right)\right\rangle\\
		&\leq \alpha_t N|x|+\Lambda_t \frac{|x|}{N}\\
		&\leq (\alpha^+_t+\Lambda_t)|x|^2.
	  \end{split}
  \end{equation}
  Hence for all $N\in \mathbb{N}$, 
  \begin{equation*}
	\langle x, b_N(t,x)\rangle\leq (\alpha^+_t+\Lambda_t)|x|^2+\Lambda_t.
  \end{equation*}
  Moreover, $b_N$ satisfies \eqref{Ineq polynomial growth} with the same $\Gamma_2, \kappa$.
  Now let
  \begin{equation}
	\rho(x)=
	  \begin{cases}
		c\exp \left\{ \frac{1}{|x|^2-1} \right\}, & |x|< 1,\\
		0, & |x|\geq 1,
	  \end{cases}
  \end{equation}
  where constant $c$ is chosen so that $\int_{\mathbb{R}^d}\rho(x)dx=1$. For $\epsilon\in (0,1]$, define
  \begin{equation}
	\label{Smooth analysis of b 3}
	  \rho_{\epsilon}(x):=\epsilon^{-d}\rho(\epsilon^{-1}x), \ b_N^{\epsilon}(t,x):=\big(b_N(t,\cdot)*\rho_{\epsilon}\big)(x)=\int_{\mathbb{R}^d}b_N(t,y)\rho_{\epsilon}(x-y)dy,
  \end{equation}
  where $*$ stands for the usual spatial convolution. Then $b_N^{\epsilon}$ is smooth and 
  \begin{equation}
	\label{Smooth analysis of b 4}
	  |b_N^{\epsilon}(t,x)-b_N(t,x)|\leq \int_{\mathbb{R}^d}|b_N(t,y)-b_N(t,x)|\rho_{\epsilon}(x-y)dy\leq \ell_N\epsilon.
  \end{equation}
  Hence
  \begin{eqnarray*}
    \langle x, b_N^{\epsilon}(t,x)\rangle &=&\langle x, b_N(t,x)\rangle+\langle x, b_N^{\epsilon}(t,x)-b_N(t,x)\rangle\\
	&\leq& (\alpha^+_t+\Lambda_t)|x|^2+\Lambda_t+\ell_N\epsilon |x|\\
	&\leq& (\alpha^+_t+\Lambda_t+\ell_N^2\epsilon^2)|x|^2+\Lambda_t+1.
  \end{eqnarray*}
  Choose $\epsilon_N=\frac{1}{\ell_N}\wedge 1$, then for all $\epsilon\in (0,\epsilon_N]$ we have
  \begin{equation}
	\label{Smooth analysis of b 5}
	\langle x, b_N^{\epsilon}(t,x)\rangle\leq (\alpha^+_t+\Lambda_t+1)|x|^2+\Lambda_t+1.
  \end{equation}
  Moreover, for any $\epsilon\in (0,1]$,
  \begin{equation}
	\label{Smooth analysis of b 6}
	|b_N^{\epsilon}(t,x)|\leq \int_{\mathbb{R}^d}|b_N(t,y)|\rho_{\epsilon}(x-y)dy\leq 2^{\kappa}\Gamma_2(1+|x|^{\kappa}).
  \end{equation}
  Then $b_N^{\epsilon}$ is smooth, Lipschitz continuous with the Lipschitz constant $\ell_N$, and satisfies \eqref{Ineq weakly coercivity} and \eqref{Ineq polynomial growth} in Assumption \ref{A3} (ii) with 
  \begin{equation}
	\label{1127-1}
	\alpha^{\prime}_t=\alpha^+_t+\Lambda_t+1,\ \ \Lambda^{\prime}_t=\Lambda_t+1,\ \ \Gamma_2^{\prime}=2^{\kappa}\Gamma_2.
  \end{equation} 
  Similarly define 
  \begin{equation}
	\label{0804-1}
	  \sigma^{\epsilon}(t,x):=\sigma(t,\cdot)*\rho_{\epsilon}(x)=\int_{\mathbb{R}^d}\sigma(t,y)\rho_{\epsilon}(x-y)dy.
  \end{equation}
  Then it is easy to check that $\sigma^{\epsilon}$ is smooth and satisfies Assumption \ref{A3} (i) and Assumption \ref{A3 new} with the same $\Gamma_1$. Moreover, we also have
  \begin{equation}
	  \label{0804-4}
	  \|\sigma^{\epsilon}(t,x)-\sigma(t,x)\|\leq \int_{\mathbb{R}^d}\|\sigma(t,y)-\sigma(t,x)\|\rho_{\epsilon}(x-y)dy\leq \Gamma_1\epsilon.
  \end{equation}

  \subsection{Lower bound of density for SDEs with smooth approximating coefficients}
 In this subsection we investigate the approximating smooth system, thus always assume the following condition:
 
 \begin{hypothesis}
  Functions $b(t,x),\sigma(t,x)$ are smooth in $x$ and satisfy Assumptions \ref{A3}, \ref{A3 new} with parameters $(\Gamma_1, \alpha'_t, \Lambda'_t, \Gamma'_2,\kappa)$. Moreover, $b(t,\cdot)$ is globally Lipschitz continuous uniformly in $t$.
 \end{hypothesis}
Condition {\bf (H)} is not the condition required for the result of this paper. It is a condition satisfied by the systems in the approximating procedure as a mean to take the limit later to SDEs with Assumptions \ref{A3} and \ref{A3 new}.

 Under {\bf (H)}, it is well-known that for each $(s,x)\in \mathbb{R}\times \mathbb{R}^d$, SDE \eqref{SDE} has a unique strong solution $X_t^{s,x}$.
 Moreover, the strong solution is a Markov process and admits a smooth density $p(t,s,x,y)$ which satisfies the backward Kolmogorov equation:
 \begin{equation*}
	\label{Eq of backward Kolmogorov equation}
 	\partial_s p(t,s,x,y)+\mathcal{L}_{s,x}p(t,s,x,y)=0, \ p(t,s,\cdot,y)\to \delta_y(\cdot) \text{ weakly as } s\uparrow t,
 \end{equation*}
and the forward Kolmogorov equation (Fokker-Planck equation):
\begin{equation}
	\label{Eq of forward Kolmogorov equation}
	\partial_t p(t,s,x,y)-\mathcal{L}^*_{t,y}p(t,s,x,y)=0, \ p(t,s,x,\cdot)\to \delta_x(\cdot) \text{ weakly as } t\downarrow s,
\end{equation}
where setting $a=\sigma\sigma^{\top}$,
\begin{equation*}
	\mathcal{L}_{s,x}f(x)=\frac{1}{2}Tr\left(a(s,x)\triangledown_x^2f(x)\right)+\langle b(s,x), \triangledown_xf(x)\rangle,
\end{equation*}
and $\mathcal{L}^*_{t,y}$ being the dual operator of $\mathcal{L}_{t,y}$ by
\begin{equation}
	\label{0804-2}
	\mathcal{L}^*_{t,y}f(y)=\frac{1}{2}\sum_{i,j=1}^d\partial_{y_i}\partial_{y_j}(a_{ij}(t,y)f(y))- {\rm div} (b(t,\cdot)f)(y).
\end{equation}

Actually under {\bf (H)} there has been the following two-sided Gaussian estimates on a compact set in times (see \cite{delarue2010density,menozzi2021density} for more details)
\begin{equation*}
	C^{-1}\varrho_{\lambda}(t-s,\theta_{t,s}(x)-y)\leq p(t,s,x,y)\leq C\varrho_{\lambda^{-1}}(t-s,\theta_{t,s}(x)-y),
\end{equation*}
for some $C,\lambda\geq 1$ where
\begin{equation*}
	\varrho_{\lambda}(t,x):=t^{-\frac{d}{2}}\exp(-\lambda |x|^2/t), \ t>0,
\end{equation*}
and $\theta$ stands for the deterministic flow associated with the drift, i.e.
\begin{equation}
	\label{ODE}
	\frac{d}{dt}\theta_{t,\tau}(\xi)=b(t,\theta_{t,\tau}(\xi)), \ \theta_{\tau,\tau}(\xi)=\xi.
\end{equation}
But these constants $C, \lambda$ depend on the global Lipschitz constant of drift $b$, which fails in the local Lipschitz case. In this section, we will give the lower bound of $p(t,s,x,y)$ which is independent of the global Lipschitz constant of drift $b$.

One key idea here is to consider the following SDE with fixed $(\tau,\xi)\in \mathbb{R}\times \mathbb{R}^d$ as freezing parameters which will be chosen later on
\begin{equation}
	\label{Freezing parameter SDE}
	d\tilde{X}_{t,s}^{\tau,\xi}=b(t,\theta_{t,\tau}(\xi))dt+\sigma(t,\tilde{X}_{t,s}^{\tau,\xi})dW_t,\ \tilde{X}_{s,s}^{\tau,\xi}=x,\ t\geq s, 
\end{equation}
where $\theta_{\cdot,\tau}(\xi)$ is the unique solution of ODE \eqref{ODE}. Note that $(\theta_{t,\tau}(\xi))_{t\geq \tau}$ stands for the forward flow and $(\theta_{t,\tau}(\xi))_{t\leq \tau}$ the backward flow of ODE \eqref{ODE}. Then it is clear that, for any choice of freezing couple $(\tau,\xi)$, $\tilde{X}_{t,s}^{\tau,\xi}$ has a density $\tilde{p}^{\tau,\xi}(t,s,x,y)$.
Moreover, $\tilde{p}^{\tau,\xi}(t,s,x,y)$ satisfies the backward Kolmogorov equation and the forward Kolmogorov equation, where the infinitesimal operator $\mathcal{L}^{\tau,\xi}_{\cdot,\cdot}$ and its dual $(\mathcal{L}^{\tau,\xi}_{\cdot,\cdot})^*$ correspond to the drift given by $b(\cdot, \theta_{\cdot,\tau}(\xi))$.

Note that we can rewrite \eqref{Eq of forward Kolmogorov equation} by the following:
\begin{equation*}
	\partial_t p(t,s,x,y)-(\mathcal{L}^{\tau,\xi}_{t,y})^*p(t,s,x,y)=(\tilde{\mathcal{L}}^{\tau,\xi}_{t,y})^*p(t,s,x,y),
\end{equation*}
where 
\begin{equation*}
	\tilde{\mathcal{L}}^{\tau,\xi}_{t,y}=\mathcal{L}_{t,y}-\mathcal{L}^{\tau,\xi}_{t,y}=\langle b(t,y)-b(t,\theta_{t,\tau}(\xi)), \triangledown_y\rangle.
\end{equation*}
Therefore $p(t,s,x,y)$ can be represented by 
\begin{equation}
	\label{Freezing parameter Duhamel formula}
	\begin{split}
	p(t,s,x,y)&=\tilde{p}^{\tau,\xi}(t,s,x,y)+\int_s^t\int_{\mathbb{R}^d}\tilde{p}^{\tau,\xi}(t,r,z,y)(\tilde{\mathcal{L}}^{\tau,\xi}_{r,z})^*p(r,s,x,z)dzdr\\
	&=\tilde{p}^{\tau,\xi}(t,s,x,y)+\int_s^t\int_{\mathbb{R}^d}p(r,s,x,z)\tilde{\mathcal{L}}^{\tau,\xi}_{r,z}\tilde{p}^{\tau,\xi}(t,r,z,y)dzdr.
	\end{split}
\end{equation}
Equation \eqref{Freezing parameter Duhamel formula} is also known as the Duhamel type representation formula.

Now for notational convenience, we write
\begin{equation*}
	H^{\tau,\xi}(t,s,x,y):=\tilde{\mathcal{L}}^{\tau,\xi}_{s,x}\tilde{p}^{\tau,\xi}(t,s,x,y),
\end{equation*}
and
\begin{equation*}
	p\otimes H^{\tau,\xi}(t,s,x,y)=\int_s^t\int_{\mathbb{R}^d}p(r,s,x,z)H^{\tau,\xi}(t,r,z,y)dzdr.
\end{equation*}
The idea to estimate $p(t,s,x,y)$ via \eqref{Freezing parameter Duhamel formula} consists of iterating this representation formula. An induction shows that, for any integer $n\geq 2$,
\begin{equation}
	\label{Iterating representation formula}
	p(t,s,x,y)=\tilde{p}^{\tau,\xi}(t,s,x,y)+\sum_{j=1}^{n-1}(\tilde{p}^{\tau,\xi}\otimes (H^{\tau,\xi})^{\otimes j})(t,s,x,y)+p\otimes (H^{\tau,\xi})^{\otimes n}(t,s,x,y).
\end{equation}

The key for the lower bound of $p(t,s,x,y)$ is to find the lower bound of the first term and the upper bounds of the last two terms on the right hand side of \eqref{Iterating representation formula}. 

First we give the following two-sided estimates for $\tilde{p}^{\tau,\xi}(t,s,x,y)$.
\begin{lemma}
	\label{Lemma of Gaussian density estimate}
	Under {\bf (H)}, for any $t-s\leq T$, there exist constants $\tilde{\lambda}_0, \tilde{\lambda}_1, \tilde{C}_0, \tilde{C}_1>0$ depending only on $(d,\Gamma_1, \Gamma_2, \kappa,T)$ and $\bar{\theta}^{\tau,\xi}_{t,s}$, where 
	$\bar{\theta}^{\tau,\xi}_{t,s}:=\sup_{r\in[s,t]}|\theta_{r,\tau}(\xi)|,$ such that for all $x,y\in \mathbb{R}^d$,
	\begin{equation}
		\label{* in Lemma of Gaussian density estimate}
		\tilde{C}_0^{-1}\varrho_{\tilde{\lambda}_0^{-1}}(t-s, \vartheta_{t,s}^{\tau,\xi}+x-y)\leq \tilde{p}^{\tau,\xi}(t,s,x,y) \leq \tilde{C}_0\varrho_{\tilde{\lambda}_0}(t-s, \vartheta_{t,s}^{\tau,\xi}+x-y),
	\end{equation}
	and
	\begin{equation}
		\label{** in Lemma of Gaussian density estimate}
		|\triangledown_x\tilde{p}^{\tau,\xi}(t,s,x,y)|\leq \tilde{C}_1(t-s)^{-\frac{1}{2}}\varrho_{\tilde{\lambda}_1}(t-s, \vartheta_{t,s}^{\tau,\xi}+x-y),
	\end{equation}
	where $\vartheta_{t,s}^{\tau,\xi}=\int_s^tb(r,\theta_{r,\tau}(\xi))dr$.
\end{lemma}
\begin{proof}
	Recall that \eqref{* in Lemma of Gaussian density estimate} and \eqref{** in Lemma of Gaussian density estimate} were established in \cite{menozzi2021density} for $\sigma$ being $\alpha$-H\"{o}der continuous (Assumption $(\mathbf{H}_{\alpha}^{\sigma})$) and $b$ being $\beta$-H\"{o}der continuous (Assumption $(\mathbf{H}_{\beta}^{b})$).
	Note that these assumptions for $\alpha=\beta=1$ are automatically satisfied in our case.
	Then our results follow from (1.15) and (1.16) in Theorem 1.2 in \cite{menozzi2021density} respectively.
\end{proof}
\begin{remark}
	Let $\tilde{\lambda}=\tilde{\lambda}_0\wedge \tilde{\lambda}_1$. It is obvious that \eqref{* in Lemma of Gaussian density estimate} and \eqref{** in Lemma of Gaussian density estimate} still hold if we replace $\tilde{\lambda}_0, \tilde{\lambda}_1$ by $\tilde{\lambda}$.
\end{remark}

The next lemma gives the upper bound of the iterated convolution $(H^{\tau,\xi})^{\otimes n}$. First for any fixed $T>0$, we fix the following notation:
\begin{equation*}
	\gamma(\kappa,\lambda):=\sup_{x\in \mathbb{R}^d}|x|^{\kappa}e^{-\lambda|x|^2}, \ \ \lambda_n=\frac{\tilde{\lambda}}{2^n},
\end{equation*}
\begin{equation*}
	C_1=2^{4\kappa-1}\Gamma_2\tilde{C}_1\left( 1+\gamma(\kappa,\lambda_1)+2^{\kappa^2}\Gamma_2^{\kappa}T^{\kappa} \right)\left( 1+|\bar{\theta}^{\tau,\xi}_{t,s}|^{\kappa^2} \right),
\end{equation*}
and for $n\geq 2$
\begin{align*}
  C_n&=2^{3n\kappa}\left( 1+\gamma((n-1)\kappa,\lambda_n)+2^{(n-1)\kappa^2}\Gamma_2^{(n-1)\kappa}T^{(n-1)\kappa} \right)\\
  &\quad \ \times\left( 1+|\bar{\theta}^{\tau,\xi}_{t,s}|^{(n-1)\kappa^2} \right)(\lambda_{n-1} \pi^{-1})^{-d/2}C_1 C_{n-1}.
\end{align*}
\begin{lemma}
	\label{Lemma of estimate H}
	Under {\bf (H)}, for any $t-s\leq T$ we have
	\begin{equation}
		\label{* in Lemma of estimate H}
		|(H^{\tau,\xi})^{\otimes n}(t,s,x,y)|\leq C_n(1+|x|^{n\kappa}\wedge |y|^{n\kappa})(t-s)^{(n-2)/2}\varrho_{\lambda_n}(t-s,\vartheta_{t,s}^{\tau,\xi}+x-y).
	\end{equation}
\end{lemma}

\begin{proof}
	By the definition of $H$, the polynomial growth condition \eqref{Ineq polynomial growth} with parameters $(\Gamma', \kappa)$ and Lemma \ref{Lemma of Gaussian density estimate}, we have
	\begin{equation}
		\label{1 in Lemma of estimate H}
		\begin{split}
			|H^{\tau,\xi}(t,s,x,y)|&\leq |b(s,x)-b(s,\theta_{s,\tau}(\xi))|\cdot |\triangledown_x\tilde{p}^{\tau,\xi}(t,s,x,y)|\\
		&\leq \Gamma'_2\tilde{C}_1\big(2+|x|^{\kappa}+|\theta_{s,\tau}(\xi)|^{\kappa}\big)(t-s)^{-\frac{1}{2}}\varrho_{\tilde{\lambda}}(t-s,\vartheta_{t,s}^{\tau,\xi}+x-y).
		\end{split}
	\end{equation}
	Here $\Gamma'_2=2^{\kappa}\Gamma_2$ as defined in \eqref{1127-1}.
	Since $(a+b)^{\kappa}\leq 2^{\kappa-1}(a^{\kappa}+b^{\kappa})$ for all $a,b>0$ and
	\begin{equation*}
		|\vartheta_{t,s}^{\tau,\xi}|\leq \Gamma'_2 T(1+|\bar{\theta}^{\tau,\xi}_{t,s}|^{\kappa}),
	\end{equation*}
	we know that
	\begin{equation}
		\label{2 in Lemma of estimate H}
		\begin{split}
			&\quad \ |H^{\tau,\xi}(t,s,x,y)|\\
      &\leq \Gamma'_2\tilde{C}_1\left( 2+2^{2\kappa-2}\left( 2^{\kappa-1}(\Gamma'_2)^{\kappa}T^{\kappa}\left( 1+|\bar{\theta}^{\tau,\xi}_{t,s}|^{\kappa^2} \right)+|y|^{\kappa} \right)+|\bar{\theta}^{\tau,\xi}_{t,s}|^{\kappa} \right)(t-s)^{-\frac{1}{2}}\\
      &\quad \ \times\varrho_{\tilde{\lambda}}(t-s,\vartheta_{t,s}^{\tau,\xi}+x-y)+2^{\kappa-1}\Gamma'_2\tilde{C}_1|\vartheta_{t,s}^{\tau,\xi}+x-y|^{\kappa}(t-s)^{-\frac{1}{2}}\varrho_{\tilde{\lambda}}(t-s,\vartheta_{t,s}^{\tau,\xi}+x-y)\\
		&\leq 2^{3\kappa-1}\Gamma'_2\tilde{C}_1\left( 1+\gamma(\kappa,\lambda_1)+(\Gamma'_2)^{\kappa}T^{\kappa} \right)\left( 1+|\bar{\theta}^{\tau,\xi}_{t,s}|^{\kappa^2} \right)\\
		&\quad \ \times (1+|y|^{\kappa})(t-s)^{-\frac{1}{2}}\varrho_{\lambda_1}(t-s,\vartheta_{t,s}^{\tau,\xi}+x-y).
		\end{split}
	\end{equation}
	Then it follows from \eqref{1 in Lemma of estimate H} and \eqref{2 in Lemma of estimate H} that
	\begin{equation}
		\label{Estimates of H 1}
		|H^{\tau,\xi}(t,s,x,y)|\leq C_1(1+|x|^{\kappa}\wedge |y|^{\kappa})(t-s)^{-1/2}\varrho_{\lambda_1}(t-s,\vartheta_{t,s}^{\tau,\xi}+x-y).
	\end{equation}
	For $n=2$, estimate \eqref{Estimates of H 1} and the definition of ``convolution'' $\otimes$ yield
	\begin{equation*}
		|(H^{\tau,\xi})^{\otimes 2}(t,s,x,y)|\leq C_1^2(1+|x|^{\kappa})(1+|y|^{\kappa})\int_s^t(r-s)^{-1/2}(t-r)^{-1/2}(\varrho^{\tau,\xi}_{\lambda_1})^{\otimes}(r,t,s,x,y)dr,
	\end{equation*}
	where
	\begin{equation}
		\label{0707_1}
		\begin{split}
			(\varrho^{\tau,\xi}_{\lambda_1})^{\otimes}(r,t,s,x,y)&=\int_{\mathbb{R}^d}\varrho_{\lambda_1}(r-s,\vartheta_{r,s}^{\tau,\xi}+x-z)\varrho_{\lambda_1}(t-r,\vartheta_{t,r}^{\tau,\xi}+z-y)dz\\
		&=(\lambda_1 \pi^{-1})^{-d/2}\varrho_{\lambda_1}(t-s,\vartheta_{t,s}^{\tau,\xi}+x-y),
		\end{split}
	\end{equation}
	where the second equality is due to the Chapman-Kolmogorov property for Gaussian semigroup. Note that
	\begin{align*}
		\int_s^t(r-s)^{-1/2}(t-r)^{-1/2}dr=2\int_0^{(t-s)/2}r^{-1/2}(t-s-r)^{-1/2}dr\leq 4.
	\end{align*}
	Thus we have
	\begin{equation}
		\label{3 in Lemma of estimate H}
		|(H^{\tau,\xi})^{\otimes 2}(t,s,x,y)|\leq 4(\lambda_1 \pi^{-1})^{-d/2}C_1^2 (1+|x|^{\kappa})(1+|y|^{\kappa})\varrho_{\lambda_1}(t-s,\vartheta_{t,s}^{\tau,\xi}+x-y).
	\end{equation} 
	Now for general $\kappa,\lambda>0$, we have 
	\begin{equation}
		\label{4 in Lemma of estimate H}
		\begin{split}
			&\quad \ (1+|x|^{\kappa})\varrho_{\lambda}(t-s,\vartheta_{t,s}^{\tau,\xi}+x-y)\\
			&\leq \left( 1+2^{2\kappa-2}\left( 2^{\kappa-1}(\Gamma'_2)^{\kappa}T^{\kappa}\left( 1+|\bar{\theta}^{\tau,\xi}_{t,s}|^{\kappa^2} \right)+|y|^{\kappa} \right) \right)\varrho_{\lambda}(t-s,\vartheta_{t,s}^{\tau,\xi}+x-y)\\
			&\quad \ +2^{\kappa-1}|\vartheta_{t,s}^{\tau,\xi}+x-y|^{\kappa}\varrho_{\lambda}(t-s,\vartheta_{t,s}^{\tau,\xi}+x-y)\\
			&\leq 2^{3\kappa-3}\left( 1+\gamma(\kappa,\lambda/2)+(\Gamma'_2)^{\kappa}T^{\kappa} \right)\left( 1+|\bar{\theta}^{\tau,\xi}_{t,s}|^{\kappa^2} \right)(1+|y|^{\kappa})\varrho_{\lambda/2}(t-s,\vartheta_{t,s}^{\tau,\xi}+x-y),
		\end{split}
	\end{equation}
	and symmetrically
	\begin{equation}
		\label{5 in Lemma of estimate H}
		\begin{split}
			(1+|y|^{\kappa})\varrho_{\lambda}(t-s,\vartheta_{t,s}^{\tau,\xi}+x-y)
			&\leq 2^{3\kappa-3}\left( 1+\gamma(\kappa,\lambda/2)+(\Gamma'_2)^{\kappa}T^{\kappa} \right)\left( 1+|\bar{\theta}^{\tau,\xi}_{t,s}|^{\kappa^2} \right)\\
			&\quad \ \times(1+|x|^{\kappa})\varrho_{\lambda/2}(t-s,\vartheta_{t,s}^{\tau,\xi}+x-y).
		\end{split}
	\end{equation}
	Then it follows from \eqref{3 in Lemma of estimate H}, \eqref{4 in Lemma of estimate H} and \eqref{5 in Lemma of estimate H} that
	\begin{equation*}
		|(H^{\tau,\xi})^{\otimes 2}(t,s,x,y)|\leq C_2 (1+|x|^{2\kappa}\wedge |y|^{2\kappa})\varrho_{\lambda_2}(t-s,\vartheta_{t,s}^{\tau,\xi}+x-y).
	\end{equation*}
	The general estimate \eqref{* in Lemma of estimate H} can be proved similarly as above by induction.
  Note that $(H^{\tau,\xi})^{\otimes n}=H^{\tau,\xi}\otimes (H^{\tau,\xi})^{\otimes (n-1)}$, then the factor $(t-s)^{(n-2)/2}$ appears in \eqref{* in Lemma of estimate H} since for all $n\geq 2$,
  \begin{align*}
    \int_s^t(r-s)^{-1/2}(t-r)^{(n-3)/2}dr&\leq (t-s)^{(n-2)/2}\int_s^t(r-s)^{-1/2}(t-r)^{-1/2}dr\\
    &\leq 4(t-s)^{(n-2)/2}.
  \end{align*}
\end{proof}

Note that $p(t,s,x,y)$ satisfies the forward Kolmogorov equation \eqref{Eq of forward Kolmogorov equation}. According to Theorem 8.2.1 in \cite{bogachev2015fokker}, we have the following two-point lower bound relation of the density.
\begin{lemma}
	\label{Lemma of lower bound density at a point}
	Under {\bf (H)}, for any $T>0$, there exists a constant $\mathcal{K}$ depending on $(d,\Gamma_1,\Gamma_2,\kappa,T)$ such that for any $x,y,y_0\in \mathbb{R}^d$, $s<r<t$ with $t-s\leq T$, we have
	\begin{equation*}
		\label{* in Lemma of lower bound density at a point}
		p(t,s,x,y)\geq p(r,s,x,y_0)\exp\left\{-\mathcal{K}\left( 1+\frac{t-r}{r-s}\left( 1+|y-y_0|^{2\kappa} \right)+\frac{1}{t-r}|y-y_0|^2 \right) \right\}.
	\end{equation*}
\end{lemma}
\begin{proof}
	Note that operator $\mathcal{L}^*_{t,y}$ in \eqref{0804-2} can be writen as the following divergence form:
	\begin{equation*}
		\mathcal{L}^*_{t,y}f(y)=\frac{1}{2}\sum_{i,j=1}^d\partial_{y_i}(a_{ij}(t,y)\partial_{y_j}f(y))-{\rm div}(b^*(t,\cdot)f)(y),
	\end{equation*}
	where 
	\begin{equation}
		\label{0804-3}
		b^*_i(t,y)=b_i(t,y)-\frac{1}{2}\sum_{j=1}^d\partial_{y_j}a_{ij}(t,y).
	\end{equation}
	Note that $a_{ij}(t,y)$ is smooth in $x$ for all $1\leq i,j\leq d$, from \eqref{Ineq of non-degenerate diffusion} and \eqref{my1127},
	$$|a_{ij}(t,x)-a_{ij}(t,y)|\leq (\|\sigma(t,x)\|+\|\sigma(t,y)\|)\|\sigma(t,x)-\sigma(t,y)\|\leq 2d^{1/2}\Gamma_1^{3/2}|x-y|,$$
	then $|\partial_{y_j}a_{ij}(t,y)|\leq 2d^{1/2}\Gamma_1^{3/2}$. It turns out, by polynomial growth condition \eqref{Ineq polynomial growth} with parameters $(2^{\kappa}\Gamma_2, \kappa)$ and \eqref{0804-3}, that
	\begin{equation*}
		|b^*(t,y)|\leq 2^{\kappa}\Gamma_2(1+|y|^{\kappa})+d^2\Gamma_1^{3/2}.
	\end{equation*}
	Now fix $s\in \mathbb{R}, x,y_0\in \mathbb{R}^d$. Denote $q^{s,x,y_0}(v,y):=p(s+v,s,x,y+y_0)$, $0\leq v\leq T$. It is easy to check that 
	$$\partial_vq^{s,x,y_0}(v,y)-\mathcal{L}^*_{v,y}q^{s,x,y_0}(v,y)=0.$$
  Since $t-s\leq T$, then it follows from \eqref{Ineq of non-degenerate diffusion} and Theorem 8.2.1 in \cite{bogachev2015fokker} that there exists a constant $\mathcal{K}_0=\mathcal{K}_0(d,\Gamma_1,T)$ such that for all $s<r<t,\ y\in \mathbb{R}^d$,
  \begin{eqnarray*}
    &&q^{s,x,y_0}(t-s,y-y_0)\\
    &\geq& q^{s,x,y_0}(r-s,0)\\
    &&\times \exp\left\{-\mathcal{K}_0\left( 1+\frac{t-r}{r-s}\left( 2^{\kappa}\Gamma_2(1+|y-y_0|^{\kappa})+d^2\Gamma_1^{3/2} \right)^2+\frac{1}{t-r}|y-y_0|^2 \right) \right\}\\
		&\geq& q^{s,x,y_0}(r-s,0)\exp\left\{-\mathcal{K}\left( 1+\frac{t-r}{r-s}\left( 1+|y-y_0|^{2\kappa} \right)+\frac{1}{t-r}|y-y_0|^2 \right) \right\},
  \end{eqnarray*}
	where $\mathcal{K}=(1+2^{\kappa+2}\Gamma_2+2d^2\Gamma_1^{3/2})\mathcal{K}_0$.
\end{proof}

 Now for any $T>0$, denote
 $$\bar{C}_j=2j^{-1}(\lambda_j\pi^{-1})^{-d/2}\tilde{C}_0C_j,\ j\leq d,\ \ \bar{C}_{d+1}:=2C_{d+1},$$
 $$\bar{\theta}:=\sqrt{2}e^{g(T)+T}\sqrt{g(T)+T}, \ \ \hat{C}=\sum_{j=1}^{d+1}2^{j\kappa-1}\bar{C}_jT^{(j-1)/2}\left( 2+\bar{\theta}^{j\kappa} \right),$$
 and
 $$\eta_1=2^{(d/2)-1}\tilde{C}_0^{-1}e^{-\mathcal{K}}, \ \ \eta_2=2^{2\kappa-1}\mathcal{K}(1+\bar{\theta}^{2\kappa})(1+8T\tilde{C}_0^2\hat{C}^2), \ \ \eta_3=4\mathcal{K}(1+\bar{\theta}^2).$$
 Then we have the following lower bound density result. 
 \begin{theorem}
	\label{Theorem of lower bound density smooth coefficient}
	 Under {\bf (H)}, for all $t-s\leq T$ we have
	 \begin{equation}
    \label{* in Theorem of lower bound density smooth coefficient}
    \begin{split}
      p(t,s,x,y)&\geq \eta_1(t-s)^{-d/2}\\
      &\quad \ \times \exp\left\{-\eta_2\big(1+|x|^{2(d+1)\kappa}\big)(1+|x-y|^{2\kappa})-\eta_3(t-s)^{-1}(1+|x-y|^2)\right\}.
    \end{split}
	 \end{equation}
	 In particular, $\eta_1, \eta_2, \eta_3$ depend only on $(d,\Gamma_1,\Gamma_2,\kappa,g(T),T)$.
 \end{theorem}

 \begin{proof}
	 We first estimate $(\tilde{p}^{\tau,\xi}\otimes (H^{\tau,\xi})^{\otimes j})$ term in \eqref{Iterating representation formula}. By Lemma \ref{Lemma of Gaussian density estimate}, Lemma \ref{Lemma of estimate H} and \eqref{0707_1}, we have
	 \begin{equation}
		\label{1 in Theorem of lower bound density smooth coefficient}
		\begin{split}
			|\tilde{p}^{\tau,\xi}\otimes (H^{\tau,\xi})^{\otimes j}|(t,s,x,y)
      &\leq \tilde{C}_0C_j(1+|y|^{j\kappa})\int_s^t(t-r)^{(j-2)/2}\\
      &\quad \ \times\int_{\mathbb{R}^d}\varrho_{\lambda_j}(r-s,\vartheta_{r,s}^{\tau,\xi}+x-z)\varrho_{\lambda_j}(t-r,\vartheta_{t,r}^{\tau,\xi}+z-y)dzdr \\
			&\leq \bar{C}_j(1+|y|^{j\kappa})(t-s)^{j/2}\varrho_{\lambda_j}(t-s,\vartheta_{t,s}^{\tau,\xi}+x-y),
		\end{split}
	 \end{equation}
	 where
	 $$\bar{C}_j=2j^{-1}(\lambda_j\pi^{-1})^{-d/2}\tilde{C}_0C_j.$$
	 For the last term in \eqref{Iterating representation formula}, taking $n=d+1$ we have
	 \begin{equation}
		\label{2 in Theorem of lower bound density smooth coefficient}
		\begin{split}
			&\quad \ |p\otimes(H^{\tau,\xi})^{\otimes (d+1)}|(t,s,x,y)\\
      &\leq C_{d+1} \big(1+|y|^{(d+1)\kappa}\big)\int_s^t(t-r)^{(d-1)/2}\times \int_{\mathbb{R}^d}p(r,s,x,z)\varrho_{\lambda_{d+1}}(t-r,\vartheta_{t,r}^{\tau,\xi}+z-y)dzdr\\
			&= C_{d+1} \big(1+|y|^{(d+1)\kappa}\big)\int_s^t(t-r)^{(d-1)/2}\mathbf{E}\big[\varrho_{\lambda_{d+1}}(t-r,\vartheta_{t,r}^{\tau,\xi}+X_r^{s,x}-y)\big]dr\\
			&\leq C_{d+1} \big(1+|y|^{(d+1)\kappa}\big)\int_s^t(t-r)^{-1/2}dr\\
			&= 2C_{d+1} \big(1+|y|^{(d+1)\kappa}\big)\sqrt{t-s}.
		\end{split}
	 \end{equation}
	 Here in the second inequality we used the fact that $(t-r)^{d/2}$ cancels $(t-r)^{-d/2}$ in $\varrho_{\lambda_{d+1}}(t-r,\cdot)$.
	 Then it follows from \eqref{Iterating representation formula}, \eqref{* in Lemma of Gaussian density estimate}, \eqref{1 in Theorem of lower bound density smooth coefficient} and \eqref{2 in Theorem of lower bound density smooth coefficient} that
	 \begin{equation}
		\label{3 in Theorem of lower bound density smooth coefficient}
		 \begin{split}
			 p(t,s,x,y)&\geq \tilde{C}_0^{-1}\varrho_{\tilde{\lambda}^{-1}}(t-s,\vartheta_{t,s}^{\tau,\xi}+x-y)\\
			 &\quad \ -\sum_{j=1}^d\bar{C}_j(1+|y|^{j\kappa})(t-s)^{j/2}\varrho_{\lambda_j}(t-s,\vartheta_{t,s}^{\tau,\xi}+x-y)\\
			 &\quad \ -2C_{d+1} \big(1+|y|^{(d+1)\kappa}\big)\sqrt{t-s}.
		 \end{split}
	 \end{equation}
	 Note that inequality \eqref{3 in Theorem of lower bound density smooth coefficient} holds for all freezing point $(\tau,\xi)$  (constants $\bar{C}_j, C_{d+1}$ also depend on the choice of $(\tau,\xi)$ with $\bar{\theta}_{t,s}^{\tau,\xi}$). Now we choose $(\tau,\xi)=(s,0)$ in \eqref{3 in Theorem of lower bound density smooth coefficient}.  First note that
	 \begin{equation*}
		 |\theta_{t,s}(x)|^2
		 \leq e^{2\int_{s}^{t}\alpha'_rdr}|x|^2+\int_{s}^{t}e^{2\int_{u}^{t}\alpha'_rdr}2\Lambda'_udu\leq e^{2(g(T)+T)}|x|^2+2e^{2(g(T)+T)}(g(T)+T),
	 \end{equation*}
	 hence
	 \begin{equation*}
		\bar{\theta}_{t,s}^{s,0}=\sup_{r\in [s,t]}|\theta_{r,s}(0)|\leq \sqrt{2}e^{g(T)+T}\sqrt{g(T)+T}.
	 \end{equation*}
	 Here $\alpha'_r$ and $\Lambda'_r$ were defined in \eqref{1127-1}.
	 In the following, we will replace $\bar{\theta}_{t,s}^{\tau,\xi}$ by $\sqrt{2}e^{g(T)+T}\sqrt{g(T)+T}$ in the construction $C_n$ for all $n\in \mathbb{N}$, still denoted by $C_n$. Then $C_n$ depends only on $(n,d,\Gamma,\kappa,g(T),T)$ and
	 \begin{equation*}
		\label{4 in Theorem of lower bound density smooth coefficient}
		 \begin{split}
			 p(t,s,x,\theta_{t,s}(0)+x)&\geq \left( \tilde{C}_0^{-1}-\left( \sum_{j=1}^{d+1}\bar{C}_jT^{(j-1)/2}(1+|\theta_{t,s}(0)+x|^{j\kappa})\right) \sqrt{t-s}\right)(t-s)^{-d/2}\\
			 &\geq \left( \tilde{C}_0^{-1}-\hat{C}(1+|x|^{(d+1)\kappa}) \sqrt{t-s}\right)(t-s)^{-d/2}.
		 \end{split}
	 \end{equation*}
	 Let $\delta_0:=\big(2\tilde{C}_0\hat{C}(1+|x|^{(d+1)\kappa})\big)^{-2}$, then for any $x\in \mathbb{R}^d$, $t-s\leq \delta_0$ we have
	 \begin{equation}
		\label{5 in Theorem of lower bound density smooth coefficient}
		p(t,s,x,\theta_{t,s}(0)+x)\geq \frac{1}{2}\tilde{C}_0^{-1}(t-s)^{-d/2}.
	 \end{equation}
	 Hence it follows from Lemma \ref{Lemma of lower bound density at a point} and \eqref{5 in Theorem of lower bound density smooth coefficient} that for any $x\in \mathbb{R}^d$, $\delta:=\delta_0\wedge \frac{t-s}{2}$ and $t-s\leq T$
	 \begin{align*}
		 p(t,s,x,y)&\geq p(s+\delta,s,x,\theta_{s+\delta,s}(0)+x)\\
     &\quad \ \times \exp\bigg\{-\mathcal{K}\bigg( 1+\frac{t-s-\delta}{\delta}\left( 1+|y-x-\theta_{s+\delta,s}(0)|^{2\kappa} \right)\\
     &\qquad \qquad \qquad \qquad \qquad \qquad \qquad +\frac{|y-x-\theta_{s+\delta,s}(0)|^{2}}{t-s-\delta}\bigg) \bigg\}\\
		 &\geq 2^{(d/2)-1}\tilde{C}_0^{-1}e^{-\mathcal{K}}(t-s)^{-d/2}\\
		 &\quad \ \times \exp\bigg\{-2^{2\kappa-1}\mathcal{K}(1+2^{\kappa}e^{2\kappa(g(T)+T)}(g(T)+T)^{\kappa})\frac{t-s-\delta}{\delta}\left( 1+|x-y|^{2\kappa} \right)\\
		 &\ \qquad \qquad \qquad \qquad  -4\mathcal{K}(1+2e^{2(g(T)+T)}(g(T)+T))(t-s)^{-1}(1+|x-y|^2) \bigg\}.
	 \end{align*}
	 where we used $\delta^{-d/2}\geq \left( \frac{t-s}{2} \right)^{-d/2}$ in the second inequality. Note also that
	 \begin{equation*}
		\frac{t-s-\delta}{\delta}=\big(2\tilde{C}_0\hat{C}(1+|x|^{(d+1)\kappa})\big)^2(t-s-\delta)\vee 1\leq (1+8T\tilde{C}_0^2\hat{C}^2)(1+|x|^{2(d+1)\kappa}).
	 \end{equation*}
	 Hence we obtained \eqref{* in Theorem of lower bound density smooth coefficient}.
 \end{proof}

 \subsection{Lower bound of density for SDEs without smooth assumption on coefficients}
  Now when the coefficients $b,\sigma$ are smooth, we have obtained the lower bound \eqref{* in Theorem of lower bound density smooth coefficient} of the transition density $p(t,s,x,y)$ by Duhamel type representation formula \eqref{Freezing parameter Duhamel formula} in the previous subsection. In this subsection, we will study the lower bound of $p(t,s,x,y)$ under Assumption \ref{A3} along with the smooth approximation outlined from \eqref{Smooth analysis of b 1}-\eqref{0804-4}.

  Now we denote by $X_t^{s,x}, (X_N)_t^{s,x}, (X_N^{\epsilon})_t^{s,x}$, the unique strong solution of SDE \eqref{SDE} with coefficients $(b,\sigma), (b_N,\sigma), (b_N^{\epsilon},\sigma^{\epsilon})$, respectively. Similarly denote by $P(t,s,x,\Gamma), P_N(t,s,x,\Gamma)$ and $P_N^{\epsilon}(t,s,x,\Gamma)$ the Markovian transition functions of $X_t^{s,x}, (X_N)_t^{s,x}$ and $(X_N^{\epsilon})_t^{s,x}$ respectively. It is well-known that they all have densities,  denoted by $p(t,s,x,y), p_N(t,s,x,y)$ and $p_N^{\epsilon}(t,s,x,y)$, respectively.

  \begin{lemma}
	\label{Lemma of density in the convergence case}
	  Assume that $\mathbb{R}^d$-valued random variables $X_n, X$ have densities $q_n(x), q(x)$. If $X_n\rightarrow X$ $a.s.$ or in the sense of $\mathbb{L}^p$ for some $p\geq 1$, then we have 
	  $$q(x)\geq \liminf_{n\to \infty}q_n(x), \ \text{ for }\  a.e.\  x\in \mathbb{R}^d.$$
  \end{lemma}
  \begin{proof} 
	  If $X_n\xrightarrow{\mathbb{L}^p} X$ for some $p\geq 1$, then there exists subsequence $n_k$ such that $X_{n_k}\xrightarrow{a.s.} X$. Note that $\liminf_{n_k\to \infty}q_{n_k}(x)\geq \liminf_{n\to \infty}q_n(x)$, hence it is enough to prove this lemma with the almost sure convergence assumption only.

	  Let $\underline{q}(x)=\liminf_{n\to \infty}q_n(x)$. Then for any $f\in \mathcal{B}_{b,+}(\mathbb{R}^d)$, a nonnegative bounded Borel measurable function, applying Fatou's lemma we have
	  \begin{equation*}
		  \int_{\mathbb{R}^d}f(x)\underline{q}(x)dx\leq \liminf_{n\to \infty}\int_{\mathbb{R}^d}f(x)q_n(x)dx.
	  \end{equation*}
	  Since $X_n\xrightarrow{a.s.} X$, by dominated convergence theorem we obtain
	  \begin{equation*}
		\lim_{n\to \infty}\int_{\mathbb{R}^d}f(x)q_n(x)dx=\lim_{n\to \infty}\mathbf{E}\left[ f(X_n) \right]=\mathbf{E}\left[ f(X) \right]=\int_{\mathbb{R}^d}f(x)q(x)dx.
	  \end{equation*}
	  Hence for all $f\in \mathcal{B}_{b,+}(\mathbb{R}^d)$
	  \begin{equation*}
		\int_{\mathbb{R}^d}f(x)q(x)dx\geq \int_{\mathbb{R}^d}f(x)\underline{q}(x)dx,
	  \end{equation*}
	  which implies that $q(x)\geq \underline{q}(x)$ for $a.e.\ x\in \mathbb{R}^d$.
  \end{proof}

  Notice that according to Theorem \ref{Theorem of lower bound density smooth coefficient}, we have already had the lower bound of $p_N^{\epsilon}$. Then along with Lemma \ref{Lemma of density in the convergence case}, we need some convergent relations between $X_t^{s,x},(X_N)_t^{s,x}$ and $(X_N^{\epsilon})_t^{s,x}$ to obtain the lower bound of $p$.
  \begin{theorem}
  	\label{Theorem of uniformly bounded solution}
	  Suppose Assumption \ref{A3} holds. Then for any initial condition $(s,x)\in \mathbb{R}\times \mathbb{R}^d$, there exists a unique solution $(X_t^{s,x})_{t\geq s}$ of SDE \eqref{SDE} such that
	  \begin{equation}
		\label{* in Theorem of uniformly bounded solution}
		\begin{split}
			\mathbf{E}\left[ |X_{t}^{s,x}|^{2}\right]&\leq e^{2\int_{s}^{t}\alpha_rdr}|x|^{2}+\int_{s}^{t}e^{2\int_{u}^{t}\alpha_rdr}(2\Lambda_u+d\Gamma_1)du\\
			&\leq e^{2g(t-s)}|x|^{2}+e^{2g(t-s)}(2g(t-s)+d\Gamma_1|t-s|),
		\end{split}
	\end{equation}
	and there exists a constant $C(d,\Gamma_1,g(t-s),|t-s|)$ depending only on $(d,\Gamma_1,g(t-s),|t-s|)$ such that
	\begin{equation}
		\label{*' in Theorem of uniformly bounded solution}
		\begin{split}
			\mathbf{E}\bigg[\sup_{s\leq r\leq t} |X_{r}^{s,x}|^{2}\bigg]&\leq C(d,\Gamma_1,g(t-s),|t-s|)(1+|x|^2).
		\end{split}
	\end{equation}
	  Moreover, we have the following relations: for any $t\geq s, x\in \mathbb{R}^d$,
	  \begin{align}
		\label{** in Theorem of uniformly bounded solution}
		(X_N^{\epsilon})_t^{s,x}\xrightarrow{\mathbb{L}^2} (X_N)_t^{s,x} \text{ as } \epsilon\downarrow 0 \ \text{ and } \ 
		(X_N)_t^{s,x} \xrightarrow{a.s.} X_t^{s,x} \text{ as } N\to \infty.
	  \end{align}
  \end{theorem}

  \begin{proof}
	  Denote by $\mathcal{L}$ the infinitesimal generator operator of SDE \eqref{SDE}, i.e. for all $V\in C^{1,2}(\mathbb{R}\times \mathbb{R}^d)$,
	  $$\mathcal{L}V(t,x)=\frac{\partial V(t,x)}{\partial t}+\sum_{i=1}^db_i(t,x)\frac{\partial V(t,x)}{\partial x_i}+\frac{1}{2}\sum_{i,j=1}^d(\sigma\sigma^{\top})_{i,j}(t,x)\frac{\partial^2 V(t,x)}{\partial x_i\partial x_j}.$$
	  Let $V(t,x)=|x|^2$, we have
	  \begin{align*}
		\mathcal{L}V(t,x)=2\langle x,b(t,x)\rangle+\|\sigma(t,x)\|^2\leq 2\alpha_t |x|^2+2\Lambda_t+d\Gamma_1.
	  \end{align*}
	  Now let $\widehat{V}(t,x)=e^{-2\int_{s}^{t}\alpha_rdr}|x|^2$, then by a similar calculation, we know that $$\mathcal{L}\widehat{V}(t,x)\leq e^{-2\int_{s}^{t}\alpha_rdr}(2\Lambda_t+d\Gamma_1).$$ 
	  Denote 
	  \begin{equation}\label{1017-1}
		T_N^{s,x}:=\inf\{t\geq s: |X_t^{s,x}|>N\} \ \text{ with } \inf \emptyset=+\infty,
	  \end{equation}
	  then it follows from Dynkin's formula that 
	  \begin{equation*}
		  \begin{split}
			\mathbf{E}\bigg[\widehat{V}\Big(t\wedge T_N^{s,x},X_{t\wedge T_N^{s,x}}^{s,x}\Big)-\widehat{V}(s,X_s^{s,x})\bigg]&=\mathbf{E}\int_s^{t\wedge T_N^{s,x}}\mathcal{L}\widehat{V}(u,X_u^{s,x})du\\
			&\leq \int_{s}^{t}e^{-2\int_{s}^{u}\alpha_rdr}(2\Lambda_u+d\Gamma_1)du.
		  \end{split}
	  \end{equation*}
	  Note that $\alpha\in C(\mathbb{R}), \Lambda\in \mathbb{L}^1_{loc}(\mathbb{R})$, then
	  \begin{align}
		\label{Dynkin formula with stopping time}
		\mathbf{E}\left[ e^{-2\int_s^{t\wedge T_N^{s,x}}\alpha_rdr}\Big|X_{t\wedge T_N^{s,x}}^{s,x}\Big|^{2} \right]\leq |x|^2+\int_{s}^{t}e^{-2\int_{s}^{u}\alpha_rdr}(2\Lambda_u+d\Gamma_1)du<\infty.
	\end{align}
	Since $e^{-2\int_s^{t\wedge T_N^{s,x}}\alpha_rdr}\geq e^{-2g(t-s)}$, then \eqref{Dynkin formula with stopping time} leads to the estimate
	\begin{equation}
		\label{Estimate of stopping time T_N}
		\begin{split}
		\mathbf{P}\{T_N^{s,x}\leq t\}\leq \frac{1}{N^{2}}e^{2g(t-s)}\bigg( |x|^2+\int_{s}^{t}e^{-2\int_{s}^{u}\alpha_rdr}(2\Lambda_u+d\Gamma_1)du \bigg) \to 0 \text{ as } N\to \infty.
		\end{split}
	\end{equation}
	Then by monotone convergence theorem
	$$1=\lim_{N\to \infty}\mathbf{E}\big[1_{\{T_N^{s,x}> t\}}\big]=\mathbf{E}\bigl[\lim_{N\to \infty}1_{\{T_N^{s,x}> t\}}\bigr]$$
	which shows that
	\begin{equation*}
		\lim_{N\to \infty} T_N^{s,x}> t, \ \mathbf{P}-a.s., \text{ for all } t\geq s. 
	\end{equation*}
	Thus $T_N^{s,x}\uparrow \infty, \ \mathbf{P}-a.s.$ as $N\to \infty$, which gives the unique regular solution of SDE \eqref{SDE}. Note that we can also deduce from \eqref{Dynkin formula with stopping time} that
	\begin{equation}
		\label{0706-2}
		\mathbf{E}\left[ e^{-2\int_{s}^{t}\alpha_rdr}|X_{t}^{s,x}|^{2}1_{\{T_N^{s,x}>t\}} \right]\leq |x|^{2}+\int_{s}^{t}e^{-2\int_{s}^{u}\alpha_rdr}(2\Lambda_u+d\Gamma_1)du.
	\end{equation}
	Letting $N\to \infty$ in \eqref{0706-2} together with \eqref{Estimate of stopping time T_N}, we obtain \eqref{* in Theorem of uniformly bounded solution}.

	In order to prove \eqref{*' in Theorem of uniformly bounded solution}, we apply It$\hat{\rm o}$'s formula to $|X_t^{s,x}|^2$ on $[s,t]$
	\begin{equation*}
		\begin{split}
			|X_t^{s,x}|^2=&|x|^2+\int_{s}^{t}\big(2\langle X_u^{s,x}, b(u,X_u^{s,x})\rangle+\|\sigma(u,X_u^{s,x})\|^2\big)du+2\int_{s}^{t}\langle X_u^{s,x}, \sigma(u,X_u^{s,x})dW_u\rangle\\
			\leq& |x|^2+\int_{s}^{t}\big(2\alpha_u|X_u^{s,x}|^2+2\Lambda_u+d\Gamma_1\big)du+2\int_{s}^{t}\langle X_u^{s,x}, \sigma(u,X_u^{s,x})dW_u\rangle.
		\end{split}
	\end{equation*}
	Hence
	\begin{equation}
		\label{1127-2}
		\sup_{s\leq r\leq t}|X_r^{s,x}|^2\leq |x|^2+\int_{s}^{t}\big(2\alpha^+_u|X_u^{s,x}|^2+2\Lambda_u+d\Gamma_1\big)du+2\sup_{s\leq r\leq t}\bigg|\int_{s}^{r}\langle X_u^{s,x}, \sigma(u,X_u^{s,x})dW_u\rangle\bigg|.
	\end{equation}
	Then \eqref{*' in Theorem of uniformly bounded solution} follows from \eqref{* in Theorem of uniformly bounded solution}, \eqref{1127-2} and Burkholder-Davis-Gundy (BDG) inequality.

	Now we are in position to prove \eqref{** in Theorem of uniformly bounded solution}. For any fixed $N$, since $b_N$ has Lipschitz constant $\ell_N$, by \eqref{Smooth analysis of b 4} and \eqref{0804-4}, we have
	\begin{align*}
		\mathbf{E}\left[ \big| (X_N^{\epsilon})_t^{s,x}-(X_N)_t^{s,x} \big|^2 \right]&\leq 2\mathbf{E}\bigg[ \bigg|\int_s^t\left( b_N^{\epsilon}(r,(X_N^{\epsilon})_r^{s,x})- b_N(r,(X_N)_r^{s,x}) \right) dr\bigg|^2 \bigg]\\
		&\quad \ +2\mathbf{E}\bigg[ \bigg|\int_s^t\left( \sigma^{\epsilon}(r,(X_N^{\epsilon})_r^{s,x})- \sigma(r,(X_N)_r^{s,x}) \right) dW_r\bigg|^2 \bigg]\\
		&\leq 4\left( \ell_N^2 (t-s)+\Gamma_1^2 \right)(t-s)\epsilon^2\\
		&\quad \ +4(\ell_N^2(t-s)+\Gamma_1^2)\int_s^t\mathbf{E}\left[ \big| (X_N^{\epsilon})_r^{s,x}-(X_N)_r^{s,x} \big|^2 \right] dr.
	\end{align*}
	Then by the Gronwall inequality,
	\begin{equation*}
		\mathbf{E}\left[ \big| (X_N^{\epsilon})_t^{s,x}-(X_N)_t^{s,x} \big|^2 \right]\leq 4\left( \ell_N^2 (t-s)+\Gamma_1^2\right)(t-s) e^{4(\ell_N^2(t-s)+\Gamma_1^2) (t-s)}\epsilon^2.
	\end{equation*}
	Hence $(X_N^{\epsilon})_t^{s,x}\xrightarrow{\mathbb{L}^2}(X_N)_t^{s,x}$ as $\epsilon\downarrow 0$.

	Note that $b_N(t,x)=b(t,x)$ for all $|x|\leq N$, it is easy to check that $X_{t\wedge T_N^{s,x}}^{s,x}=(X_N)_{t\wedge T_N^{s,x}}^{s,x}$. Since $T_N^{s,x}\uparrow \infty, \ \mathbf{P}-a.s.$ as $N\to \infty$, then we know that $(X_N)_t^{s,x}\xrightarrow{a.s.}X_t^{s,x}$ as $N\to \infty$.
  \end{proof}

  Then we have the following lower bound estimate for density $p(t,s,x,y)$.

  \begin{theorem}
	\label{Theorem of lower bound density unsmooth coefficient}
	  Under Assumptions \ref{A3} and \ref{A3 new}, for any $T>0$, there exist positive constants $\eta_1,\eta_2$ and $\eta_3$ depending only on $(d,\Gamma_1,\Gamma_2,\kappa,g(T),T)$ such that for all $|t-s|\leq T$, we have
	  \begin{equation}
    \label{* in Theorem of lower bound density unsmooth coefficient}
    \begin{split}
      p(t,s,x,y)&\geq \eta_1(t-s)^{-d/2}\\
      &\quad \ \times\exp\left\{-\eta_2\big(1+|x|^{2(d+1)\kappa}\big)(1+|x-y|^{2\kappa})-\eta_3(t-s)^{-1}(1+|x-y|^2)\right\}. 
    \end{split}
	 \end{equation}
  \end{theorem}
  \begin{proof}
	  Since for any $\epsilon\leq \epsilon_N:=\frac{1}{\ell_N}\wedge 1$, $b_N^{\epsilon}$ satisfies Assumption \ref{A3} (ii) with parameters $\Gamma_1,\kappa$ and $(\alpha'_t,\Lambda'_t,\Gamma'_2)$ defined as in \eqref{1127-1}. Then by Theorem \ref{Theorem of lower bound density smooth coefficient}, there exist positive constants $\eta_1,\eta_2,\eta_3$ depending only on $(d,\Gamma_1,\Gamma_2,\kappa,g(T),T)$ such that $p_N^{\epsilon}$ satisfies inequality \eqref{* in Theorem of lower bound density unsmooth coefficient}. Note the lower bound in \eqref{* in Theorem of lower bound density unsmooth coefficient} is independent of $\epsilon$ and $N$, so the desired inequality eventually follows from \eqref{** in Theorem of uniformly bounded solution} and Lemma \ref{Lemma of density in the convergence case}. Here we take limit as $\epsilon\to 0$ followed by taking limit as $N\to \infty$.
  \end{proof}

\section{Geometric versus subgeometric or supergeometric convergence to entrance measures}\label{Sec 4}
In this section, we study the well-posedness of entrance measure for the nonautonomous SDE \eqref{SDE}. 
\subsection{Convergence to entrance measures with nondegenerate noises}
Now let us consider the following assumption.
\begin{condition}
	\label{A1 1203}
	Assume that $\alpha_{\cdot}$ in \eqref{Ineq weakly coercivity} satisfies the following condition:
	  \begin{equation*}
		  \limsup_{T\to \infty}\frac{1}{T}\int_{-T}^{0}\alpha_rdr<0.
	  \end{equation*}
  \end{condition}

  \begin{remark}
	  \label{Re1 1203}
	  It is easy to see that Assumption \ref{A1 1203} is equivalent to the following condition: for any $t\in \mathbb{R}$,
	  \begin{equation*}
		\limsup_{T\to \infty}\frac{1}{T+t}\int_{-T}^{t}\alpha_rdr=\limsup_{T\to \infty}\frac{1}{T}\int_{-T}^{0}\alpha_rdr<0.
	  \end{equation*}
	  It means that the drift term $b$ has weakly dissipative property in average.
  \end{remark}

  Let $V(x)=|x|^2$ and
  \begin{equation}\label{1206-3}
	  m_t:=\int_{-\infty}^{t}e^{2\int_{u}^{t}\alpha_rdr}(2\Lambda_u+d\Gamma_1)du.
  \end{equation}
  It follows from \eqref{* in Theorem of uniformly bounded solution} that
  \begin{equation}\label{1204-1}
	P(t,s)V(x)\leq e^{2\int_{s}^{t}\alpha_rdr}V(x)+\int_{s}^{t}e^{2\int_{u}^{t}\alpha_rdr}(2\Lambda_u+d\Gamma_1)du\leq \gamma(t,s)V(x)+m_t,
  \end{equation}
  where $\gamma(t,s):=e^{2\int_{s}^{t}\alpha_rdr}$.
  Under Assumptions \ref{A3} and \ref{A1 1203}, it can be proved that $m_t<\infty$ for all $t\in \mathbb{R}$. In fact, Assumption \ref{A1 1203} shows that there exist $a>0, T_t>0$ such that for all $s\leq -T_t$, we have $\int_{s}^{t}\alpha_rdr\leq -a(t-s)$. Hence
  \begin{equation*}
	\begin{split}
		m_t&\leq\int_{-T_t}^{t}e^{2\int_{u}^{t}\alpha_rdr}(2\Lambda_u+d\Gamma_1)du+\sum_{n=1}^{\infty}\int_{-T_t-n}^{-T_t-(n-1)}e^{-2a(t-u)}(2\Lambda_u+d\Gamma_1)du\\
		&\leq \int_{-T_t}^{t}e^{2\int_{u}^{t}\alpha_rdr}(2\Lambda_u+d\Gamma_1)du+\sum_{n=1}^{\infty}(2g(1)+d\Gamma_1)e^{-a(t+T_t+n-1)}\\
		&= \int_{-T_t}^{t}e^{2\int_{u}^{t}\alpha_rdr}(2\Lambda_u+d\Gamma_1)du+\frac{(2g(1)+d\Gamma_1)e^{-a(t+T_t)}}{1-e^{-a}}<\infty.
	\end{split}
\end{equation*}

Set 
\begin{equation}\label{1212-3}
	\alpha(\Delta):=\max\bigg\{\limsup_{|t-s|\leq \Delta, s\to -\infty}\int_{s}^{t}\alpha_rdr, 0\bigg\}
\end{equation}
and
\begin{equation}\label{1212-6}
	\mathcal{M}_{m}:=\bigg\{\mu: \mathbb{R}\to \mathcal{P}(\mathbb{R}^d)\bigg| \int_{\mathbb{R}^d}V(x)\mu_t(dx)\leq m_t \text{ for all } t\in \mathbb{R}\bigg\}.
\end{equation}

In the following, we say a measure-valued function $\{\mu_t\}_{t\in \mathbb{R}}$ is continuous if it is continuous in the sense of the weak convergence in $\mathcal{P}(\mathbb{R}^d)$ with respect to $t$. Then we have the following Theorem.

\begin{theorem}\label{Theorem 1205}
	Suppose Assumptions \ref{A3}, \ref{A3 new} and \ref{A1 1203} hold. Assume that $\liminf_{t\to -\infty}\alpha_t>-\infty$,
	\begin{itemize}
		\item [(1)] if 
		\begin{equation}\label{0110-3}
			\limsup_{t\to -\infty}m_t<\infty,
		\end{equation}
		there exists a unique continuous entrance measure $\mu_{\cdot}$ of SDE (\ref{SDE}) in $\mathcal{M}_{m}$. 
		\item [(2)] if 
		\begin{equation}\label{0110-4}
			\limsup_{\Delta\to \infty}\frac{\alpha(\Delta)}{\ln\Delta}< \frac{1}{2},
		\end{equation}
		then for any $t\in \mathbb{R}, \mu_1,\mu_2\in \mathcal{P}(\mathbb{R}^d)$, there exist $\beta>0, C_t>0,\lambda>0$ such that for all $t\geq s$,
		\begin{equation}\label{1205 main 1}
			\rho_{\beta}\bigl(P^*(t,s)\mu_1,P^*(t,s)\mu_2\bigr)\leq C_te^{-\lambda (t-s)}\rho_{\beta}(\mu_1,\mu_2).
		\end{equation}
		If in addition that $\liminf_{t\to -\infty}m_t<\infty$, there exists a unique continuous entrance measure $\mu_{\cdot}$ of SDE (\ref{SDE}) in $\mathcal{M}_{m}$ such that
		\begin{equation}\label{1205 main 2}
			\rho_{\beta}\bigl(P(t,s,x,\cdot),\mu_t\bigr)\leq C_t(1+m_s+|x|^2)e^{-\lambda (t-s)}.
		\end{equation}
	\end{itemize}
\end{theorem}

To prove the continuity of $\mu_{\cdot}$, we need the following lemma.

  \begin{lemma}
	\label{Lemma of strong feller and stochastically continuous}
	  Under Assumptions \ref{A3}, we have the following properties:
	  \begin{description}
		  \item[(i)] The two-parameter Markovian semigroup $P(t,s)$ is a Feller semigroup, i.e. for any $t>s$, $f\in C_b(\mathbb{R}^d)$, $P(t,s)f\in C_b(\mathbb{R}^d)$;
		  \item[(ii)] Let $X_t^{s,x}$ be the unique solution of SDE \eqref{SDE} with initial condition $(s,x)$, then
		  \begin{equation}
			\label{* in Lemma of strong feller and stochastically continuous}
			  \lim_{t\to t_0}\mathbf{E}\left[ \big|X_t^{s,x}-X_{t_0}^{s,x}\big|^2 \right]=0.
		  \end{equation}
		  In particular, $X_t^{s,x}\xrightarrow{\mathbf{P}} X_{t_0}^{s,x} \text{ as } t\to t_0$, i.e. for any $\delta>0$,
		  \begin{equation}
			\label{** in Lemma of strong feller and stochastically continuous}
			\lim_{t\to t_0} \mathbf{P}\big\{|X_t^{s,x}-X_{t_0}^{s,x}|>\delta\big\}=0,
		  \end{equation}
		  and $P(t,s)$ is  stochastically continuous.
	  \end{description}
  \end{lemma}
  \begin{proof}
	$\mathbf{(i)}$: For any $N\in \mathbb{N}$, since $b_N$ and $\sigma$ are globally Lipschitz continuous, it is well-known that $P_N(t,s)$ is a Feller semigroup.

	Next we will show that for any $f\in C_b(\mathbb{R}^d)$, $P_N(t,s)f(x)\to P(t,s)f(x)$ locally uniformly in $x$ as $N\to \infty$, i.e. for any $M>0$, 
	\begin{equation*}
		\lim_{N\to \infty}P_N(t,s)f(x)=P(t,s)f(x) \ \text{ uniformly on } B_M:=\{x\in \mathbb{R}^d: |x|\leq M\}.
	\end{equation*}
	Note that $X_{t\wedge T_N^{s,x}}^{s,x}=(X_N)_{t\wedge T_N^{s,x}}^{s,x}$ for $T_N^{s,x}$ being defined in \eqref{1017-1}, then 
	\begin{align*}
		|P_N(t,s)f(x)-P(t,s)f(x)|&=\big|\mathbf{E}[f((X_N)_t^{s,x})]-\mathbf{E}[f(X_t^{s,x})]\big|\\
		&=\Big|\mathbf{E}\left[ \left\{ f((X_N)_t^{s,x})-f(X_t^{s,x}) \right\}1_{\{T_N^{s,x}\leq t\}} \right]\Big|\\
		&\leq 2|f|_{\infty}\mathbf{P}\{T_N^{s,x}\leq t\}.
	\end{align*}
	It follows from \eqref{Estimate of stopping time T_N} that
	\begin{equation*}
		\begin{split}
			\sup_{x\in B_M}|P_N(t,s)f(x)-P(t,s)f(x)|\leq& \frac{2|f|_{\infty}}{N^{2}}e^{2g(t-s)}\left( M^2+\int_{s}^{t}e^{-2\int_{s}^{u}\alpha_rdr}(2\Lambda_u+d\Gamma_1)du \right)
		\end{split}
	\end{equation*}
	tends to 0 as $N\to \infty$.
	Hence for any $M>0$, the continuity of $P_N(t,s)f$ and uniform convergence on $B_M$ implies the continuity of $P(t,s)f$ on $B_M$. But for any $x\in \mathbb{R}^d$, one can choose $M$ large enough such that $x\in B_M$, thus $P(t,s)f(x)$ is continuous in $x$. Its boundedness follows from that of $f$ easily.

	$\mathbf{(ii)}$: Without loss of generality, we assume that $t_0\geq t\geq s$. Note that
	\begin{equation*}
		X_{t_0}^{s,x}=X_{t}^{s,x}+\int_{t}^{t_0}b(r,X_r^{s,x})dr+\int_{t}^{t_0}\sigma(r,X_r^{s,x})dW_r.
	\end{equation*}
	Hence
	\begin{equation*}
		\sup_{s\leq t\leq t_0}\bigg|\int_{t}^{t_0}b(r,X_r^{s,x})dr\bigg|^2\leq 4\sup_{s\leq t\leq t_0}|X_t^{s,x}|^2+2\sup_{s\leq t\leq t_0}\bigg|\int_{t}^{t_0}\sigma(r,X_r^{s,x})dW_r\bigg|^2.
	\end{equation*}
	Then \eqref{*' in Theorem of uniformly bounded solution} and Burkholder-Davis-Gundy inequality yield that
	\begin{equation*}
		\mathbf{E}\bigg[ \sup_{s\leq t\leq t_0}\bigg|\int_{t}^{t_0}b(r,X_r^{s,x})dr\bigg|^2 \bigg]\leq C(d,\Gamma_1,g(t_0-s),|t_0-s|)(1+|x|^2).
	\end{equation*}
	Since $\big|\int_{t}^{t_0}b(r,X_r^{s,x})dr\big|^2 \to 0$ as $t\to t_0$, $\mathbf{P}-a.s.$, it follows from the dominated convergence theorem that
	\begin{equation*}
		\lim_{t\to t_0}\mathbf{E}\bigg[ \bigg|\int_{t}^{t_0}b(r,X_r^{s,x})dr\bigg|^2 \bigg]=0.
	\end{equation*}
	Notice that 
	\begin{equation*}
		\mathbf{E}\bigg[ \bigg|\int_{t}^{t_0}\sigma(r,X_r^{s,x})dW_r\bigg|^2 \bigg]=\int_{t}^{t_0}\mathbf{E}\left[ \|\sigma(r,X_r^{s,x})\|^2 \right]dr\leq d\Gamma_1|t_0-t|,
	\end{equation*} 
	therefore
	\begin{equation*}
		\lim_{t\to t_0}\mathbf{E}\left[ \big|X_t^{s,x}-X_{t_0}^{s,x}\big|^2 \right]\leq 2\liminf_{t\to t_0}\biggl(\mathbf{E}\bigg[ \bigg|\int_{t}^{t_0}b(r,X_r^{s,x})dr\bigg|^2 \bigg]+\mathbf{E}\bigg[ \bigg|\int_{t}^{t_0}\sigma(r,X_r^{s,x})dW_r\bigg|^2 \bigg]\biggr)=0.
	\end{equation*}
	Let $t_0=s$ in \eqref{** in Lemma of strong feller and stochastically continuous}, then we know that $P(t,s)$ is stochastically continuous.	
  \end{proof}

  Now we give the proof of Theorem \ref{Theorem 1205}.

\begin{proof}[Proof of Theorem \ref{Theorem 1205}]
	Note that Assumption \ref{A1 1203} shows that there exist $a>0, T>0$ such that $\int_{t}^{0}\alpha_rdr\leq -a |t|$ for all $t\leq -T$.

	(1). There exists $M>0,K>0$ such that $\inf_{t\leq 0}\alpha_t\geq -M$ and $\sup_{t\leq 0}m_t\leq K$. 

	For any $\Delta>0$ and $\varsigma >0$, let 
	\begin{equation}\label{1204-3}
		B_n^{\varsigma,\Delta}:=\bigg\{1\leq i\leq n\bigg| \int_{-i\Delta}^{-(i-1)\Delta}\alpha_rdr\leq -\varsigma\Delta\bigg\}, \ n^{\varsigma,\Delta}:=\# B_n^{\varsigma,\Delta}.
	\end{equation}
	Note that for $n\geq \frac{T}{\Delta}$, $\int_{-n\Delta}^{0}\alpha_rdr\leq -an\Delta$. On the other hand,
	\begin{equation*}
		\int_{-n\Delta}^{0}\alpha_rdr\geq -\varsigma\Delta(n-n^{\varsigma,\Delta})-M\Delta n^{\varsigma,\Delta}=-\varsigma n \Delta-(M-\varsigma)n^{\varsigma,\Delta}\Delta.
	\end{equation*}
	Then choose $\varsigma_0=\frac{1}{2}a \wedge M$, we have $\frac{n^{\varsigma_0,\Delta}}{n}\geq \frac{a}{2M}>0$. Hence $B^{\varsigma_0,\Delta}=\cup_{n\geq 0}B_n^{\varsigma_0,\Delta}$ has infinite elements.

	Set $\Delta=1$. Now let us consider $\{t_n\}_{n\geq 0}$ as follows: 
	\begin{equation*}
		\begin{split}
			&t_0=0, \ t_1=\sup\bigg\{-(i-1)\bigg|i\geq 2; \ i\in B^{\varsigma_0,1}; \ \int_{-(i-1)}^{0}\alpha_rdr\leq 0\bigg\},\\
			&t_{2n}=t_{2n-1}-1, \ t_{2n+1}=\sup\bigg\{-(i-1)\bigg|i\geq |t_{2n}|+2; \ i\in B^{\varsigma_0,1}; \ \int_{-(i-1)}^{t_{2n}}\alpha_rdr\leq 0\bigg\}, n\geq 1.
		\end{split}
	\end{equation*}
	Then by \eqref{1204-1}, we have $\gamma(t_n,t_{n+1})\leq 1:=\gamma$ and $K(t_n,t_{n+1})\leq m_{t_n}\leq K$. Hence $\{t_n\}_{n\geq 0}$ is a well-controlled partition. 

	Note that $t_{2n}-t_{2n-1}=1$, \eqref{* in Theorem of lower bound density unsmooth coefficient} in Theorem \ref{Theorem of lower bound density unsmooth coefficient} implies that for any $R>0$,
	  \begin{equation}
		\label{1204-2}
		  \begin{split}
			&\inf_{V(x)\leq R, |y|\leq 1}p(t_{2n-1},t_{2n},x,y)\\
			&\geq \eta_1\exp\{-\eta_2(1+R^{(d+1)\kappa})(1+(\sqrt{R}+1)^{2\kappa})-\eta_3(1+(\sqrt{R}+1)^2)\}\\
			&=:\bar{\eta}(R,1)>0,
		  \end{split}
	  \end{equation}
	  where $\eta_1,\eta_2$ and $\eta_3$ are positive numbers depending only on $(d,\Gamma_1,\Gamma_2,\kappa,g(1))$. Hence for any $\Gamma\in \mathcal{B}(\mathbb{R}^d)$ we have
	  \begin{equation*}
		  \begin{split}
			\inf_{V(x)\leq R}P(t_{2n-1},t_{2n},x,\Gamma)&\geq \inf_{V(x)\leq R}\int_{\Gamma\cap B_1}p(t_{2n-1},t_{2n},x,y)dy\\
			&\geq \bar{\eta}(R,1){\rm Leb} (\Gamma\cap B_1),
		  \end{split}
	  \end{equation*}
	  where $B_1$ is the unit ball in $\mathbb{R}^d$ and Leb$(\cdot)$ is the Lebesgue measure on $\mathbb{R}^d$. Note that $|B_1|^{-1}{\rm Leb}(\cdot\cap B_1)$ is a probability measure on $\mathbb{R}^d$, then Assumption \ref{A1} {\bf (ii)} holds by choosing $\nu:=|B_1|^{-1}{\rm Leb}(\cdot\cap B_1)$ and
	  \begin{equation*}
		\eta(t,s):=
		\begin{cases}
			|B_1|\bar{\eta}(R,1), & [t,s]=[t_{2n-1},t_{2n}] \text{ for some } n,\\
			0, &\text{otherwise}.
		\end{cases}
	  \end{equation*}
	  Now choose $R>\frac{4K}{1-e^{-2\varsigma_0}}$ and $\delta=|B_1|\bar{\eta}(R,1)$. Let $A_n^{\delta}$, $n^{\delta}$ and $\bar{\gamma}_n^{\delta}$ be defined as in \eqref{1205-2} and \eqref{1205-4}. Then
	  \begin{equation*}
		  \bar{\gamma}_n^{\delta}=e^{-2\varsigma_0}=:\gamma^*<1 \text{ and } \liminf_{n\to \infty}\frac{n^{\delta}}{n}=\frac{1}{2}>\frac{(\gamma-1)^+R+2K}{(\gamma-1)^+R+(1-\gamma^*)R}.
	  \end{equation*}
	  Note also that for all $j>i\geq 0$, $\gamma(t_i,t_j)\leq 1$ and $K(t_i,t_j)\leq m_{t_i}\leq K$, we have
	  \begin{equation*}
		  \sup_{j>i\geq 0}P(t_i,t_j)V(x)\leq V(x)+K.
	  \end{equation*}
	  Hence all conditions in Theorem \ref{Theorem 1214} hold and there exists a unique entrance measure $\mu_{\cdot}$ of SDE \eqref{SDE} in $\mathcal{M}_{\mathcal{A}}$ where $\mathcal{A}=\{t_n\}_{n\geq 0}$. Moreover, similar to \eqref{1206-1}, for all $t_n\leq t$, there exists $0<r<1$ such that
	  \begin{equation*}
		  \rho_{\beta}\bigl(P(t,t_n,0,\cdot),\mu_t\bigr)\leq (2+\beta K)C_tr^n,
	  \end{equation*}
	  where $C_t$ is defined as in \eqref{0929-2}. 
	  Since $\mathcal{M}_{m}\subset \mathcal{M}_{\mathcal{A}}$, it remains to show that $\mu_{\cdot}$ in $\mathcal{M}_{m}$. In fact,
	  \begin{equation*}
		  \int_{\mathbb{R}^d}V(x)\mu_t(dx)=\lim_{n\to \infty}\int_{\mathbb{R}^d}V(x)P(t,t_n,0,dx)=\lim_{n\to \infty}P(t,t_n)V(0)\leq m_t.
	  \end{equation*}
	  
	  Finally, we show that $\mu_{\cdot}$ is continuous. Note that by \eqref{* in Theorem of uniformly bounded solution}, for any $D>0$ and $|t-s|\leq D$, 
	  \begin{equation*}
		  \gamma(t,s)\leq e^{2g(D)}=:\gamma_D, \ K(t,s)\leq e^{2g(D)}\bigl(2g(D)+d\Gamma_1 D\bigr)=:K_D.
	  \end{equation*}
	  Then Lemma \ref{lemma 1124-1} shows that $\zeta_{\beta}(t,s)\leq 1+\gamma_D+\beta K_D$. Hence by the construction \eqref{0929-2} of $C_t$, it can be shown that for any $T>0$,
	  \begin{equation}\label{1206-2}
		  \lim_{n\to \infty} \sup_{|t|\leq T}\rho_{\beta}\bigl(P(t,t_n,0,\cdot),\mu_t\bigr)=0.
	  \end{equation}
	  For any $t'\to t$, there exists $T>0$  such that $t',t\in [-T,T]$. Then for any $f\in C_{b,Lip}(\mathbb{R}^d)$ with Lipschitz constant $L_f$ and $t_n\leq -T$, it follows from \eqref{1206-2} that
	  \begin{align*}
		|\langle f,\mu_{t'}\rangle-\langle f,\mu_{t}\rangle|&\leq |\langle f,\mu_{t'}-P(t',t_n,0,\cdot)\rangle|+|\langle f,P(t,t_n,0,\cdot)-\mu_{t}\rangle|\\
		&\quad \ +|\langle f,P(t',t_n,0,\cdot)-P(t,t_n,0,\cdot)\rangle|\\
		&\leq |f|_{\infty}\left(\|P(t',t_n,0,\cdot)-\mu_{t'}\|_{TV}+\|P(t,t_n,0,\cdot)-\mu_{t}\|_{TV} \right)\\
		&\quad \ +\mathbf{E}\big[ \big|f\big(X_{t'}^{t_n,0}\big)-f\big(X_{t}^{t_n,0}\big)\big| \big]\\
		&\leq 2 |f|_{\infty} \sup_{|t|\leq T}\rho_{\beta}\bigl(P(t,t_n,0,\cdot),\mu_t\bigr)+L_f\Big( \mathbf{E}\Big[ \big|f\big(X_{t'}^{t_n,0}\big)-f\big(X_{t}^{t_n,0}\big)\big|^2 \Big] \Big)^{1/2}.
	  \end{align*}
	  Then by \eqref{* in Lemma of strong feller and stochastically continuous} in Lemma \ref{Lemma of strong feller and stochastically continuous}, we know that
	  \begin{align*}
		\limsup_{t'\to t}|\langle f,\mu_{t'}\rangle-\langle f,\mu_{t}\rangle|\leq 2 |f|_{\infty}\sup_{|t|\leq T}\rho_{\beta}\bigl(P(t,t_n,0,\cdot),\mu_t\bigr),
	\end{align*}
	which means that $\mu_{t'}\xrightarrow{\mathcal{W}} \mu_{t}$ as $t'\to t$ by taking limit $n\to \infty$.

	  (2). We still have $\inf_{t\leq 0}\alpha_t\geq -M$ for some $M>0$. For any fixed $\Delta>0$, let $B^{\varsigma_0,\Delta}_n$ and $n^{\varsigma_0,\Delta}$ be defined as in \eqref{1204-3} where $\varsigma_0=\frac{1}{2}a \wedge M$, we know that $\liminf_{n\to \infty}\frac{n^{\varsigma_0,\Delta}}{n}\geq \frac{a}{2M}>0$. 

	  Let $B^{\varsigma_0,\Delta}=\cup_{n\geq 0}B^{\varsigma_0,\Delta}_n$. For any $i\in B^{\varsigma_0,\Delta}$, we consider the following partition $\{t^i_j\}_{0\leq j\leq N}$ of $[-i\Delta,-(i-1)\Delta]$:
	  \begin{equation*}
		  t^i_0=-(i-1)\Delta, \ t^i_j:=\sup\bigg\{s\leq t^i_{j-1}\bigg| \int_{s}^{t^i_{j-1}}\alpha_rdr=-\frac{\varsigma_0\Delta}{N}\bigg\} \text{ for } 1\leq j\leq N-1, \ t^i_N=-i\Delta.
	  \end{equation*}
	  Since $\int_{-i\Delta}^{-(i-1)\Delta}\alpha_rdr\leq -\varsigma_0\Delta$, then $t^i_0>t^i_1>\cdots>t^i_N$. Moreover, for any $0\leq j\leq N-1$,
	  \begin{equation*}
		-\frac{\varsigma_0\Delta}{N}\geq \int_{t^i_{j+1}}^{t^i_j}\alpha_rdr\geq -M(t^i_j-t^i_{j+1}).
	  \end{equation*}
	  Hence $\Delta\geq t^i_j-t^i_{j+1}\geq \frac{\varsigma_0\Delta}{MN}$. Similarly, by \eqref{* in Theorem of lower bound density unsmooth coefficient} in Theorem \ref{Theorem of lower bound density unsmooth coefficient} we know that for any $R>0$,
	  \begin{equation}
		  \begin{split}
			&\inf_{V(x)\leq R, |y|\leq 1}p(t^i_j,t^i_{j+1},x,y)\\
			&\geq \eta_1\Delta^{-d/2}\exp\{-\eta_2(1+R^{(d+1)\kappa})(1+(\sqrt{R}+1)^{2\kappa})-\eta_3MN\varsigma_0^{-1}\Delta^{-1}(1+(\sqrt{R}+1)^2)\}\\
			&=:\bar{\eta}(R;\Delta,N)>0.
		  \end{split}
	  \end{equation}
	  Hence for all $i\in B^{\varsigma_0,\Delta}$, $0\leq j\leq N-1$,
	  \begin{equation*}
		\inf_{V(x)\leq R}P(t^i_{j},t^i_{j+1},x,\cdot)\geq |B_1|\bar{\eta}(R;\Delta,N)\nu,
	  \end{equation*}
	  where $\nu$ is the normalized Lebesgue measure of unit ball in $\mathbb{R}^d$.

	  Let $\{t_n\}_{n\geq 0}=\{-n\Delta\}_{n\geq 0}\cup(\cup_{i\in B^{\varsigma_0,\Delta}}\{t^i_j\}_{0\leq j\leq N})$ be a decreasing sequence. Now we construct $\eta(t,s)$ as follows:
	  \begin{equation*}
		\eta(t,s):=
		\begin{cases}
			|B_1|\bar{\eta}(R,\Delta,N), & [t,s]=[t^i_{j},t^i_{j+1}] \text{ for some } i\in B^{\varsigma_0,\Delta}, \ 0\leq j\leq N-1,\\
			0, &\text{otherwise}.
		\end{cases}
	  \end{equation*}
	  We choose $\delta=|B_1|\bar{\eta}(R,\Delta,N)$ where $R,\Delta,N$ will be chosen later. Let $A_n^{\delta}$, $n^{\delta}$ and $\bar{\gamma}_n^{\delta}$ be defined as in \eqref{1205-2} and \eqref{1205-4}. Since $\liminf_{n\to \infty}\frac{n^{\varsigma_0,\Delta}}{n}\geq \frac{a}{2M}$, it can be proved that $\liminf_{n\to \infty}\frac{n^{\delta}}{n}\geq \frac{N}{N+c}$ with $c=\frac{2M}{a}-1$. Note that $|t_n-t_{n+1}|\leq \Delta$ for all $n\geq 0$, \eqref{* in Theorem of uniformly bounded solution} and \eqref{1204-1} yield
	  \begin{equation*}
		  \limsup_{n\to \infty}\gamma(t_n,t_{n+1})\leq e^{2\alpha(\Delta)}=:\gamma, \ \text{ and  } \ \max_{n\geq 0}K(t_n,t_{n+1})\leq e^{2g(\Delta)}\bigl(2g(\Delta)+d\Gamma_1\Delta\bigr)=:K.
	  \end{equation*}
	  By the construction of $\{t_n\}_{n\geq 0}$, we have $\limsup_{n\to \infty}\bar{\gamma}_n^{\delta}\leq e^{-\frac{2\varsigma_0\Delta}{N}}:=\gamma^*<1$. To verify \eqref{1205-3}, it is sufficient to show that 
	  \begin{equation}\label{1205-5}
		  \frac{N}{N+c}>\frac{(e^{2\alpha(\Delta)}-1)R+2K}{\big(e^{2\alpha(\Delta)}-e^{-\frac{2\varsigma_0\Delta}{N}}\big)R}
	  \end{equation}
	  for suitable $R,\Delta,N$.
	  We first show that
	  \begin{equation}\label{1205-6}
		\frac{N}{N+c}>\frac{e^{2\alpha(\Delta)}-1}{e^{2\alpha(\Delta)}-e^{-\frac{2\varsigma_0\Delta}{N}}}
	\end{equation}
	for some $N,\Delta$. Note that \eqref{1205-6} is equivalent to
	\begin{equation*}
		c(e^{2\alpha(\Delta)}-1)<N\bigl(1-e^{-\frac{2\varsigma_0\Delta}{N}}\bigr).
	\end{equation*}
	Since $\lim_{N\to \infty}N\bigl(1-e^{-\frac{2\varsigma_0\Delta}{N}}\bigr)=2\varsigma_0\Delta$, we choose $N$ big enough such that $N\bigl(1-e^{-\frac{2\varsigma_0\Delta}{N}}\bigr)>\varsigma_0\Delta$. So it is sufficient to show that $ce^{2\alpha(\Delta)}<\varsigma_0\Delta$ which reads
	\begin{equation}\label{1205-7}
		\frac{\alpha(\Delta)}{\ln\Delta}<\frac{1}{2}\Bigl(\frac{\ln\varsigma-\ln c}{\ln\Delta}+1\Bigr).
	\end{equation}
	Since the left hand side of \eqref{1205-7} less than $\frac{1}{2}$ as $\Delta\to \infty$ while the right hand side tends to $\frac{1}{2}$, then we choose $\Delta$ large enough such that \eqref{1205-7} holds. Hence we already choose $N,\Delta$ such that \eqref{1205-6} holds. Note that the right hand side of \eqref{1205-5} tends to that of \eqref{1205-6} when $R\to \infty$, we can also choose $R$ sufficient large such that \eqref{1205-5} holds.

	Now we have shown that all conditions except \eqref{0918-4} in Theorem \ref{Theorem 1214} hold. Since $|t_n-t_{n+1}|\leq \Delta$ for all $n\geq 0$, then by Theorem \ref{Corollary 1101} and Remark \ref{remark 1108}, there exist $\beta>0, C_t>0,\lambda>0$ such that \eqref{1205 main 1} holds for all $s\in \{t_n\}_{n\geq 0}$. For $s\in (t_{n+1},t_n)$, since $|t_n-t_{n+1}|\Delta$, it follows from \eqref{* in Theorem of uniformly bounded solution} that $\gamma(t_n,s)\leq e^{2g(\Delta)}$ and $K(t_n,s)\leq K$. Hence there exists $\zeta_0=e^{2g(\Delta)}+\beta K$ such that 
	\begin{equation*}
		\rho_{\beta}\bigl(P^*(t_n,s)\mu_1,P^*(t_n,s)\mu_2\bigr)\leq \zeta_0  \rho_{\beta}(\mu_1,\mu_2).
	\end{equation*}
	Therefore \eqref{1205 main 1} holds for all $s\leq t$ with a modified $C_t$.

	If in addition that $\liminf_{t\to -\infty}m_t<\infty$, there exists a sequence $A:=\{s_n\}_{n\geq 0}$ such that $s_n\downarrow -\infty$ and for all $s_n>s_m$
	\begin{equation}\label{1205-8}
		\gamma(s_n,s_m)=e^{2\int_{s_m}^{s_n}\alpha_rdr}\leq 1, \ K(s_n,s_m)\leq \sup_{n\geq 0}m_{s_n}=:K<\infty.
	\end{equation}
	In fact, we first choose $\{s_n\}_{n\geq 0}$ such that $\sup_{n\geq 0}m_{s_n}<\infty$. By assumption \ref{A1 1203} and Remark \ref{Re1 1203}, we can choose the subsequence which still denote $\{s_n\}_{n\geq 0}$ such that \eqref{1205-8} holds. Hence
	\begin{equation*}
		\sup_{n>m}P(s_n,s_m)V(x)\leq V(x)+K.
	\end{equation*}
	Then similar to the proof of Theorem \ref{Theorem 1214}, there exists a unique entrance measure $\mu_{\cdot}$ in $\mathcal{M}_{A}$ such that $\{P(t,s_n,0,\cdot)\}_{n}$ converges to $\mu_t$ under $\rho_{\beta}$ distance. Note that $\mathcal{M}_{m}\subset \mathcal{M}_{A}$, it remains to show that the entrance measure $\mu_{\cdot}$ is in $\mathcal{M}_{m}$ such that \eqref{1205 main 2} holds. In fact, according to \eqref{1204-1}, for any $t\in \mathbb{R}$,
	\begin{equation*}
		\int_{\mathbb{R}^d}|x|^2\mu_t(dx)=\lim_{n\to \infty}\int_{\mathbb{R}^d}|x|^2P(t,s_n,0,dx)=\lim_{n\to \infty}P(t,s_n)V(0)\leq m_t.
	\end{equation*}
	Then \eqref{1205 main 2} follows by
	\begin{equation*}
		\begin{split}
			\rho_{\beta}\bigl(P(t,s,x,\cdot),\mu_t\bigr)&=\rho_{\beta}\bigl(P^*(t,s)\delta_x,P^*(t,s)\mu_s\bigr)\\
			&\leq C_te^{-\lambda(t-s)}\rho_{\beta}(\delta_x,\mu_s)\\
			&\leq C_t(2+\beta m_s+\beta V(x))e^{-\lambda(t-s)}\rho_{\beta}(\delta_x,\mu_s).
		\end{split}
	\end{equation*}
	Similar to that in $(1)$, we can also show that \eqref{1206-2} holds while replacing $t_n$ by $s_n$. Hence the continuity of $\mu_{\cdot}$ follows by the same method.
\end{proof}

Roughly speaking, if $b$ is weakly dissipative in average (Assumption \ref{A1 1203}), then the transition kernel will geometrically converge to the unique entrance measure. It is nature to generalize Assumption \ref{A1 1203} by the following condition:
\begin{equation}\label{1212-1}
	\limsup_{T\to \infty}\frac{1}{\phi(T)}\int_{-T}^{0}\alpha_rdr<0,
\end{equation}
where $\phi:\mathbb{R}^+\to \mathbb{R}^+$ is a nondecreasing function such that $\phi(T)\uparrow \infty$ as $T\to \infty$.


  Without loss of generality, we can assume $\phi$ is a strict increasing function such that $\phi(0)=0$ and $\phi\in C^2([0,\infty))$. Then we construct $\tilde{\phi}\in C^2(\mathbb{R})$ as follows:
  \begin{equation*}
	\tilde{\phi}(t)=
	  \begin{cases}
		  \phi(t), & t\geq 0,\\
		  -\phi(-t), & t<0.
	  \end{cases}
  \end{equation*}
  Without ambiguity, we still denote $\tilde{\phi}$ as $\phi$. It is easy to see that $\phi$ is invertible and $\phi^{-1}\in C^2(\mathbb{R})$.

  Now let $Y_t=X_{\phi^{-1}(t)}$, then for any $s\leq t, x\in \mathbb{R}^d$, we know that $Y_t^{s,x}=X_{\phi^{-1}(t)}^{\phi^{-1}(s),x}$ and hence
  \begin{equation}
	  \begin{split}
		Y_t^{s,x}&=x+\int_{\phi^{-1}(s)}^{\phi^{-1}(t)}b\big(r,X_r^{\phi^{-1}(s),x}\big)dr+\int_{\phi^{-1}(s)}^{\phi^{-1}(t)}\sigma\big(r,X_r^{\phi^{-1}(s),x}\big)dW_r\\
		&=x+\int_{s}^{t}b\big(\phi^{-1}(r),Y_r^{s,x}\big)(\phi^{-1})'(r)dr+\int_{\phi^{-1}(s)}^{\phi^{-1}(t)}\sigma\big(r,Y_{\phi(r)}^{s,x}\big)dW_r,
	  \end{split}
  \end{equation}
  where $(\phi^{-1})'$ is the derivative of $\phi^{-1}$. Let $\mathcal{F}_t$ be the natural filtration generated by $\{W_s, s\leq t\}$ and $\mathcal{F}^{\phi^{-1}}_t:=\mathcal{F}_{\phi^{-1}(t)}$. Theorem 8.5.7 (time-change formula for It$\hat{\rm o}$ integrals) in \cite{Oksendal2003} yields that there exist a standard two-sided $\mathcal{F}^{\phi^{-1}}$ Brownian motion $\tilde{W}$ such that 
  \begin{equation*}
	\int_{\phi^{-1}(s)}^{\phi^{-1}(t)}\sigma\big(r,Y_{\phi(r)}^{s,x}\big)dW_r=\int_{s}^{t}\sigma\big(\phi^{-1}(r),Y_{r}^{s,x}\big)\sqrt{(\phi^{-1})'(r)}d\tilde{W}_r.
  \end{equation*}
  Then $Y$ satisfies the following SDE:
  \begin{equation}\label{SDE 1212}
	  dY_t=b^{\phi^{-1}}(t,Y_t)dt+\sigma^{\phi^{-1}}(t,Y_t)d\tilde{W}_t,
  \end{equation}
  where
  \begin{equation}\label{1212-2}
	b^{\phi^{-1}}(t,x):=(\phi^{-1})'(t)b\big(\phi^{-1}(t),x\big), \ \ \sigma^{\phi^{-1}}(t,x):=\sqrt{(\phi^{-1})'(t)}\sigma\big(\phi^{-1}(t),x\big).
  \end{equation}
  Under Assumption \ref{A3} and \eqref{1212-1}, it is easy to see that
  \begin{equation*}
	  \langle x, b^{\phi^{-1}}(t,x)\rangle \leq (\phi^{-1})'(t)\alpha_{\phi^{-1}(t)}|x|^2+(\phi^{-1})'(t)\Lambda_{\phi^{-1}(t)}=:\alpha^{\phi^{-1}}_t|x|^2+\Lambda^{\phi^{-1}}_t
  \end{equation*}
  and
  \begin{equation*}
	  \limsup_{T\to \infty}\frac{1}{T}\int_{-T}^{0}\alpha^{\phi^{-1}}_rdr=\limsup_{T\to \infty}\frac{1}{T}\int_{-\phi^{-1}(T)}^{0}\alpha_rdr=\limsup_{T\to \infty}\frac{1}{\phi(T)}\int_{-T}^{0}\alpha_rdr<0.
  \end{equation*}
  Let $P_Y(t,s,x,\cdot)$ be the transition kernel of SDE \eqref{SDE 1212}. It is easy to see that $P_Y(t,s,x,\cdot)=P\bigl(\phi^{-1}(t), \phi^{-1}(s), x, \cdot\bigr)$. Hence $\{\mu_t\}_{t\in \mathbb{R}}$ is an entrance measure of SDE \eqref{SDE} if and only if $\{\mu_{\phi^{-1}(t)}\}_{t\in \mathbb{R}}$ is an entrance measure of SDE \eqref{SDE 1212}.
  In order to give the wellposeness of entrance measure $\mu_{\cdot}$ to SDE \eqref{SDE} and the convergence rate of $\rho_{\beta}\bigl(P(t,s,x,\cdot),\mu_t\bigr)$, it is natural to apply Theorem \ref{Theorem 1205} to SDE \eqref{SDE 1212}, which gives the following Theorem.
  \begin{theorem}\label{corollary 1226}
	Suppose that $b^{\phi^{-1}}, \sigma^{\phi^{-1}}$ and $\alpha^{\phi^{-1}}$ satisfy Assumptions \ref{A3}, \ref{A3 new} and \ref{A1 1203}. Let $m^{\phi^{-1}}$ and $\alpha^{\phi^{-1}}(\Delta)$ be defined as in \eqref{1206-3} and \eqref{1212-3} with respect to $\alpha^{\phi^{-1}}, \Lambda^{\phi^{-1}}$ and $\Gamma^{\phi^{-1}}_1$. Assume that $\liminf_{t\to -\infty}\alpha^{\phi^{-1}}_t>-\infty$,
	\begin{itemize}
		\item [(1)] if 
		\begin{equation}\label{0110-1}
			\limsup_{t\to -\infty}m^{\phi^{-1}}_t<\infty,
		\end{equation}
		there exists a unique continuous entrance measure $\mu_{\cdot}$ of SDE (\ref{SDE}) in $\mathcal{M}_{m}$. 
		\item [(2)] if 
		\begin{equation}\label{0110-2}
			\limsup_{\Delta\to \infty}\frac{\alpha^{\phi^{-1}}(\Delta)}{\ln\Delta}< \frac{1}{2},
		\end{equation}
		then for any $t\in \mathbb{R}, \mu_1,\mu_2\in \mathcal{P}(\mathbb{R}^d)$, there exist $\beta>0, C_t>0,\lambda>0$ such that for all $t\geq s$,
		\begin{equation*}
			\rho_{\beta}\bigl(P^*(t,s)\mu_1,P^*(t,s)\mu_2\bigr)\leq C_te^{-\lambda \phi(t-s)}\rho_{\beta}(\mu_1,\mu_2).
		\end{equation*}
		If in addition that $\liminf_{t\to -\infty}m_t<\infty$, there exists a unique continuous entrance measure $\mu_{\cdot}$ of SDE (\ref{SDE}) in $\mathcal{M}_{m}$ such that
		\begin{equation*}
			\rho_{\beta}\bigl(P(t,s,x,\cdot),\mu_t\bigr)\leq C_t(1+m_s+|x|^2)e^{-\lambda \phi(t-s)}.
		\end{equation*}
	\end{itemize}
  \end{theorem}
  Now we give some SDEs where their transition kernels geometrically or subgeometrically or supergeometrically converge to their entrance measures.

  \begin{example}\label{New eg}
	Consider the following one-dimensional SDE:
	\begin{equation}\label{sde 1108}
	  dX_t=\big(X_t-X_t^3+f(t)\big)dt+dW_t,
  \end{equation}
  where $f$ is a bounded and continuous function. It is easy to see that the drift term $b(t,x)=x-x^3+f(t)$ has polynomial growth and satisfies
  \begin{equation*}
	  xb(t,x)\leq -2|x|^2+|f|_{\infty}^2+4.
  \end{equation*}
  It is easy to check that \eqref{0110-3} and \eqref{0110-4} hold true.
  Hence all conditions in Theorem \ref{Theorem 1205} hold and SDE \eqref{sde 1108} has a unique entrance measure with geometric convergence rate. In particular,
  \begin{itemize}
	  \item when $f(t)=\cos t$, the entrance measure is a periodic measure and was obtained in \cite{feng2020random,feng2019existence};
	  \item when $f(t)=\cos t+\cos (\sqrt{2}t)$, this result is new and it will be proved in Section \ref{zhao22a}
	   that the entrance measure is a quasi-periodic measure;
	  \item when $f(t)=\sin\Bigl(\frac{1}{2+\cos t+\cos (\sqrt{2}t)}\Bigr)$ is an almost-periodic function (see \cite{Cheban-Liu2020}), the results in \cite{Cheban-Liu2020} cannot apply to this equation, while our result on the existence and uniqueness of its entrance measure is new.
  \end{itemize}
\end{example}

  \begin{example}
	  We consider the following 1-dimensional SDEs with parameter $\epsilon>-1$:
	  \begin{equation}\label{SDE 1212-1}
		  dX^{\epsilon}_t=-|t|^{\epsilon}X^{\epsilon}_tdt+|t|^{\frac{\epsilon}{2}}dW_t.
	  \end{equation}
	  Now let $\phi_{\epsilon}$ be a odd function such that $\phi_{\epsilon}(t):=t^{1+\epsilon}$ for all $t\geq 0$. Then by \eqref{SDE 1212}, we know that $Y_t:=X_{\phi_{\epsilon}^{-1}(t)}$ satisfies the following equation:
	  \begin{equation*}
		  dY_t=-\frac{1}{1+\epsilon}Y_tdt+\sqrt{\frac{1}{1+\epsilon}}d\tilde{W}_t.
	  \end{equation*}
	  Hence $Y$ satisfies a time-homogeneous SDE with strong dissipative drift term! Then $Y$ has a unique invariant measure $\mu$ which is a Gaussian measure $\mathcal{N}\big(0,\frac{1}{2}\big)$. Let $P_{\epsilon}(t,s,x,\cdot)$ be the transition kernel of SDE \eqref{SDE 1212-1}. By Theorem \ref{corollary 1226}, $\mu$ is also the unique invariant measure (entrance measure with constant measure $\mu$) of $X^{\epsilon}$ and there exists $\beta>0, C_t>0,\lambda>0$ such that
	  \begin{equation*}
		\rho_{\beta}\bigl(P_{\epsilon}(t,s,x,\cdot),\mu\bigr)\leq C_te^{-\lambda |t-s|^{1+\epsilon}}.
	  \end{equation*}

	  On the other hand, it can be seen that 
	  $$X_t^{\epsilon,s,x}=e^{-\int_{s}^{t}|r|^{\epsilon}dr}x+\int_{s}^{t}e^{-\int_{u}^{t}|r|^{\epsilon}dr}|u|^{\frac{\epsilon}{2}}dW_u, \ t\geq s, x\in \mathbb{R}$$
	  is the solution to SDE \eqref{SDE 1212-1} with initial condition $(s,x)$. Since $2\sqrt{\beta}|y|\leq 1+\beta |y|^2$, then
	\begin{equation}\label{1226-1}
		\begin{split}
			\rho_{\beta}\big(P_{\epsilon}(t,s,x,\cdot),\mu\big)&=\sup_{|f|\leq 1+\beta V}\bigg|\mathbf{E}[f(X_t^{\epsilon,s,x})]-\int_{\mathbb{R}}f(y)\mu(dy)\bigg|\\
			&\geq 2\sqrt{\beta}\big|\mathbf{E}[X_t^{\epsilon,s,x}]\big|\\
			&=2\sqrt{\beta}|x|e^{-\int_{s}^{t}|r|^{\epsilon}dr}.
		\end{split}
	\end{equation}
	For any fixed $t\in \mathbb{R}$, we have
	\begin{equation*}
		\lim_{s\to -\infty}\frac{\int_{s}^{t}|r|^{\epsilon}dr}{|t-s|^{1+\epsilon}}=\lim_{s\to -\infty}\frac{\int_{s}^{0}|r|^{\epsilon}dr}{|t-s|^{1+\epsilon}}=\lim_{s\to -\infty}\frac{\frac{1}{1+\epsilon}|s|^{1+\epsilon}}{|t-s|^{1+\epsilon}}=\frac{1}{1+\epsilon}.
	\end{equation*}
	Hence there exist $c_t>0,\lambda'>\frac{1}{1+\epsilon}$ such that for all $t\geq s$,
	\begin{equation*}
		\rho_{\beta}\bigl(P_{\epsilon}(t,s,x,\cdot),\mu\bigr)\geq c_t|x|e^{-\lambda'|t-s|^{1+\epsilon}}.
	\end{equation*}
	Hence,
	\begin{itemize}
		\item if $-1<\epsilon<0$, $P_{\epsilon}(t,s,x,\cdot)$ subgeometrically converges to the invariant measure $\mu$;
		\item if $\epsilon=0$, $P_{\epsilon}(t,s,x,\cdot)$ geometrically converges to the invariant measure $\mu$;
		\item if $\epsilon>0$, $P_{\epsilon}(t,s,x,\cdot)$ supergeometrically converges to the invariant measure $\mu$.
	\end{itemize}
  \end{example}

  \subsection{Convergence to entrance measures with possibly degenerate noises}

  Now we give the following assumption on $\alpha_{\cdot}$ which is slightly stronger than Assumption \ref{A1 1203}.
  \begin{condition}
	\label{New A2}
	  There exist $\Delta>0, \bar{\alpha}<0$ such that for all $t\in \mathbb{R}$, $\int_{t}^{t+\Delta}\alpha_rdr\leq \bar{\alpha}.$
  \end{condition}
  In some sense, this is an uniform assumption for a common length of intervals with weakly dissipative assumption. Though SDEs with this uniform assumption are special cases that Theorem \ref{Theorem 1214} can deal with, they contain many physically interesting equations, e.g. the stochastic resonance model of BPSV on transition of climates. In these cases, very neat results will be given.

  Under Assumption \ref{New A2}, the following (possibly degenerate) condition of $\sigma$ will give the wellposeness of the entrance measure.

  \begin{condition}\label{A6}
	 There exists a decreasing sequence $\{t_n\}_{n=0}^{\infty}\subset \mathbb{R}$ such that $t_{i-1}-t_i\geq \Delta$ for all $i\geq 1$ and the nondegenerate condition \eqref{Ineq of non-degenerate diffusion 2} holds for all $t\in \cup_{n=0}^{\infty}[t_n-\Delta, t_n]$, where $\Delta$ is the same as in Assumption \ref{New A2}.
  \end{condition}

 Now we have the following theorem.

\begin{theorem}
	\label{new Theorem existence and uniqueness of entrance measure}
	Suppose Assumptions \ref{A3} and \ref{New A2} hold. Then
	\begin{equation}
		\label{1127-6}
		\sup_{t,s\in \mathbb{R}, t\geq s}\mathbf{E}\big[|X_t^{s,x}|^2\big]\leq e^{2g(\Delta)}|x|^2+\frac{1}{1-e^{2\bar{\alpha}}}e^{2g(\Delta)}(2g(\Delta)+d\Gamma_1 \Delta).
	\end{equation}
	If in addition Assumption \ref{A6} holds, there exists a unique continuous entrance measure $\mu_{\cdot}$ of SDE (\ref{SDE}) in $\mathcal{M}$. 
\end{theorem}

\begin{proof}
	For any $t\geq s$, assume that $n\Delta\leq |t-s|\leq (n+1)\Delta$ for some integer $n\geq 0$, Assumption \ref{New A2} shows that
	\begin{equation}
		\label{1127-7}
		\int_{s}^{t}\alpha_rdr\leq n\bar{\alpha}+g(\Delta)\leq g(\Delta).
	\end{equation}
	Hence from \eqref{1127-7} and Assumption \ref{New A2}
	\begin{equation}
		\label{1127-8}
		\begin{split}
			\int_{s}^{t}e^{2\int_{u}^{t}\alpha_rdr}(2\Lambda_u+d\Gamma_1)du\leq& \int_{t-(n+1)\Delta}^{t}e^{2\int_{u}^{t}\alpha_rdr}(2\Lambda_u+d\Gamma_1)du\\
			=& \sum_{i=0}^n\int_{t-(i+1)\Delta}^{t-i\Delta}e^{2\int_{u}^{t}\alpha_rdr}(2\Lambda_u+d\Gamma_1)du\\
			\leq& \sum_{i=0}^n e^{2i\bar{\alpha}+2g(\Delta)}(2g(\Delta)+d\Gamma_1 \Delta)\\
			\leq& \frac{1}{1-e^{2\bar{\alpha}}}e^{2g(\Delta)}(2g(\Delta)+d\Gamma_1 \Delta).
		\end{split}
	\end{equation}
	Then \eqref{1127-6} follows from \eqref{* in Theorem of uniformly bounded solution}, \eqref{1127-7} and \eqref{1127-8}.

	Now we are in position to prove the well-posedness of entrance measure. 
	Let $k_0>0$ be a integer such that $(k_0-1)\bar{\alpha}+g(\Delta)<0$ and
	  \begin{equation*}
		  \tilde{t}_{2n}=t_{nk_0}, \ \tilde{t}_{2n+1}=t_{nk_0}-\Delta, \ n=0,1,2,\cdots.
	  \end{equation*}
	Consider the Lyapunov function $V(x)=|x|^2$, then \eqref{* in Theorem of uniformly bounded solution} in Theorem \ref{Theorem of uniformly bounded solution} implies that
	  \begin{equation}\label{1005-1}
		  P(t,s)V(x)\leq e^{2\int_{s}^{t}\alpha_rdr}V(x)+e^{2g(t-s)}(2g(t-s)+d\Gamma_1|t-s|).
	  \end{equation}
	  Note that $\tilde{t}_{2n+1}-\tilde{t}_{2n+2}=t_{nk_0}-\Delta-t_{nk_0+k_0}\geq (k_0-1)\Delta$, then \eqref{1127-7} shows that $\int_{\tilde{t}_{2n+2}}^{\tilde{t}_{2n+1}}\alpha_rdr\leq (k_0-1)\bar{\alpha}+g(\Delta)<0$. Note also that $\int_{\tilde{t}_{2n+1}}^{\tilde{t}_{2n}}\alpha_rdr=e^{2\bar{\alpha}}<0$.
	  Hence Assumption \ref{A1} {\bf (i)} follows from \eqref{1127-6} and \eqref{1005-1} by choosing 
	  \begin{equation*}
		  \gamma(t,s)=
		  \begin{cases}
			e^{2\bar{\alpha}},  &(t,s)=(\tilde{t}_{2n},\tilde{t}_{2n+1}),\\
			1,  &(t,s)=(\tilde{t}_{2n+1},\tilde{t}_{2n+2}),\\
			e^{2g(\Delta)},  &\text{otherwise},
		  \end{cases} \  
		  K(t,s)=\frac{1}{1-e^{2\bar{\alpha}}}e^{2g(\Delta)}(2g(\Delta)+d\Gamma_1 \Delta).
	  \end{equation*}  
	  Note that if Assumption \ref{A6} holds, then for all $t\in [\tilde{t}_{2n+1},\tilde{t}_{2n}]$, $\sigma(t,x)$ satisfies \eqref{Ineq of non-degenerate diffusion 2}. Thus by \eqref{* in Theorem of lower bound density unsmooth coefficient} in Theorem \ref{Theorem of lower bound density unsmooth coefficient}, we know that for any $R>0$,
	  \begin{equation}
		\label{1214-2}
		  \begin{split}
			&\inf_{V(x)\leq R, |y|\leq 1}p(\tilde{t}_{2n},\tilde{t}_{2n+1},x,y)\\
			&\geq \eta_1\Delta^{-d/2}\exp\{-\eta_2(1+R^{(d+1)\kappa})(1+(\sqrt{R}+1)^{2\kappa})-\eta_3\Delta^{-1}(1+(\sqrt{R}+1)^2)\}\\
			&=:\bar{\eta}(R, \Delta)>0,
		  \end{split}
	  \end{equation}
	  where $\eta_1,\eta_2$ and $\eta_3$ are positive numbers depending only on $(d,\Gamma_1,\Gamma_2,\kappa,g(\Delta),\Delta)$. Hence for any $\Gamma\in \mathcal{B}(\mathbb{R}^d)$ we have
	  \begin{equation*}
		  \begin{split}
			\inf_{V(x)\leq R}P(\tilde{t}_{2n},\tilde{t}_{2n+1},x,\Gamma)&\geq \inf_{V(x)\leq R}\int_{\Gamma\cap B_1}p(\tilde{t}_{2n},\tilde{t}_{2n+1},x,y)dy\\
			&\geq \bar{\eta}(R,\Delta){\rm Leb} (\Gamma\cap B_1),
		  \end{split}
	  \end{equation*}
	  where $B_1$ is the unit ball in $\mathbb{R}^d$ and Leb$(\cdot)$ is the Lebesgue measure on $\mathbb{R}^d$. Note that ${\rm Leb}(\cdot\cap B_1)$ is a probability measure on $\mathbb{R}^d$, then Assumption \ref{A1} {\bf (ii)} holds by choosing $\nu:={\rm Leb}(\cdot\cap B_1)$ and
	  \begin{equation*}
		\eta(t,s):=
		\begin{cases}
			\bar{\eta}(R,\Delta), & [t,s]=[\tilde{t}_{2n},\tilde{t}_{2n+1}] \text{ for some } n,\\
			0, &\text{otherwise},
		\end{cases}
	  \end{equation*}
	  where $R$ will be choosen later. Since $\gamma(\tilde{t}_{2n},\tilde{t}_{2n+1})= e^{2\bar{\alpha}}<1$ and
	  $$\max_{n\geq 1}\gamma(\tilde{t}_{n-1},\tilde{t}_{n})\leq 1:=\gamma, \ \max_{n\geq 1}K(\tilde{t}_{n-1},\tilde{t}_{n})=\frac{1}{1-e^{2\bar{\alpha}}}e^{2g(\Delta)}(2g(\Delta)+d\Gamma_1 \Delta):=K,$$
	  so if we choose $R>\frac{4K}{1-e^{2\bar{\alpha}}}$ and $\delta=\bar{\eta}(R,\Delta)$, then
	  \begin{equation*}
		\bar{\gamma}_n^{\delta}=e^{2\bar{\alpha}}:=\gamma^*<1 \ \text{ and } \ \inf_{n\geq 1}\frac{n^{\delta}}{n}=\frac{1}{2}>\frac{(\gamma-1)^+R+2K}{(\gamma-1)^+R+(1-\gamma^*)^+R}.
	  \end{equation*}
	  Hence all conditions in Theorem \ref{Theorem 1214} hold and there exists a unique entrance measure $\mu_{\cdot}$ of SDE (\ref{SDE}) in $\mathcal{M}_{\mathcal{A}}$ where $\mathcal{A}=\{\tilde{t}_{n}\}_{n\geq 0}$. By \eqref{1127-6}, it is easy to show that $\mu_{\cdot}$ is also in $\mathcal{M}$. Finally, the continuity of $\mu_{\cdot}$ is similar as that in the proof of Theorem \ref{Theorem 1205}.
\end{proof}

In the following, with slightly stronger condition, we get the (uniformly) exponential convergence to the entrance measure in the following theorem.
\begin{theorem}
	\label{Theorem existence and uniqueness of entrance measure}
	Suppose Assumptions \ref{A3}, \ref{A3 new} and \ref{New A2} hold. SDE \eqref{SDE} has a unique continuous entrance measure $\mu_{\cdot}$ in $\mathcal{M}$. Moreover, there exist positive constants $C,\lambda, \beta$ depending only on $(d,\Gamma_1,\Gamma_2,\kappa, \bar{\alpha}, g(\Delta), \Delta)$ such that for all $x\in \mathbb{R}^d,\ t\geq s$,
	\begin{equation}
		\label{* in Theorem existence and uniqueness of entrance measure}
		\rho_{\beta}\bigl(P(t,s,x,\cdot), \mu_t\bigr)\leq C(1+|x|^2)e^{-\lambda(t-s)}.
	\end{equation}
\end{theorem}
  \begin{proof}
	  Similar to the proof of Theorem \ref{new Theorem existence and uniqueness of entrance measure}, it can be proved that all conditions in Proposition \ref{Theorem of geometric ergodic of time-inhoogeneous} hold. Hence there exists a unique continuous entrance measure $\mu_{\cdot}$ of SDE (\ref{SDE}) in $\mathcal{M}$ such that \eqref{* in Theorem existence and uniqueness of entrance measure} holds.
  \end{proof}

\begin{example}
	\label{eg1}
	Consider the following one-dimensional SDE:
\begin{equation}
	\label{0730-1}
	dX_t=\big(X_t-\sin^+(t)X_t^3\big)dt+dW_t,
\end{equation}
where $W_t, t\in \mathbb{R}$, is a two-sided one-dimensional Brownian motion. 
It is obvious that the drift term $b(t,x)=x-\sin^+(t)x^3$ satisfies $|b(t,x)|\leq 1+|x|^3$.
Note that
\begin{equation*}
	xb(t,x)\leq \big(1-2\pi\sin^+(t)\big)|x|^2+\pi^2.
\end{equation*}
and for any $t\in \mathbb{R}$,
\begin{equation}\label{1018-1}
	\int_{t}^{t+2\pi}(1-2\pi\sin^+(r)\big)dr=-2\pi<0,
\end{equation}
Theorem \ref{Theorem existence and uniqueness of entrance measure} implies that SDE \eqref{0730-1} has a unique entrance measure $\mu_{\cdot}$ with geometric convergence rate. Moreover, $\mu_{\cdot}$ is a periodic measure with periods $2\pi$.
\end{example}
\begin{example}
	Instead of SDE \eqref{0730-1}, we consider the following SDE with degenerate noise:
	\begin{equation}\label{SDE 1017}
		dX_t=\big(X_t-\sin^+(t)X_t^3\big)dt+sin^+\Big(\frac{1}{4}t\Big)dW_t.
	\end{equation}
	Let $\Delta=2\pi$, \eqref{1018-1} indicates that Assumption \ref{New A2} holds. Note that $sin^+\big(\frac{1}{4}t\big)\geq \frac{\sqrt{2}}{2}$ on $[\pi-8k\pi,3\pi-8k\pi]$ for all $k\geq 0$, then Assumption \ref{A6} holds by setting $t_n=3\pi-8n\pi$. Thus it follows from Theorem \ref{new Theorem existence and uniqueness of entrance measure} that SDE \eqref{SDE 1017} has a unique entrance measure $\mu_{\cdot}$ with geometric convergence rate. Moreover, $\mu_{\cdot}$ is a periodic measure with periods $8\pi$. 

	Note that $sin^+\big(\frac{1}{4}t\big)=0$ on $[-4\pi-8k\pi,-8k\pi]$ for all $k\geq 0$, so on these time intervals, the noise is degenerate. On the interval $[-4\pi, 4\pi]$, only on $[0, \pi]$ and $[2\pi, 3\pi]$, both the noise is nondegenerate and the drift is weakly dissipative simultaneously. It is remarkable that the equation has a unique periodic measure.
\end{example}

\section{Examples satisfying Theorem \ref{Theorem 1214} but beyond results in Section \ref{Sec 4}}
Now we give several examples that go beyond Theorems \ref{Theorem 1205}, \ref{corollary 1226}, \ref{new Theorem existence and uniqueness of entrance measure} and \ref{Theorem existence and uniqueness of entrance measure}. But Theorem \ref{Theorem 1214} still applies and they all have a unique entrance measure. Note in these systems, the lengths of the time intervals where the systems expand are getting larger and larger without bounds.

\begin{example}\label{2022a}
	\label{eg2}
	We consider the following SDE:
	\begin{equation}
		\label{sde 1214}
		dX_t=\big(X_t-\sin^+(\sqrt{|t|})X_t^3\big)dt+dW_t.
	\end{equation}
	Then $b(t,x)=x-\sin^+(\sqrt{|t|})x^3$ is of polynomial growth. If 
	\begin{equation*}
		xb(t,x)=x^2-\sin^+(\sqrt{|t|})x^4\leq \alpha_t|x|^2+\Lambda_t
	\end{equation*}
	for some $\alpha_t$ and $\Lambda_t$, we conclude that $\alpha_t$ must be 1 for all $t\in [(4k^2-4k+1)\pi^2,4k^2\pi^2], \ k=1,2,\cdots$. It breaks conditions \eqref{0110-3}, \eqref{0110-4}, \eqref{0110-1}, \eqref{0110-2} and Assumption \ref{New A2}. Hence Theorems \ref{Theorem 1205}, \ref{corollary 1226}, \ref{new Theorem existence and uniqueness of entrance measure} and \ref{Theorem existence and uniqueness of entrance measure} cannot apply!
	
	Notice that
	\begin{equation}
		\label{1214-1}
		xb(t,x)\leq \big(1-16\sin^+(\sqrt{|t|})\big)|x|^2+64:=\alpha_t|x|^2+64.
	\end{equation}
	In the following, we will show that the Markovian semigroup generated by the solution of SDE \eqref{sde 1214} satisfies the assumptions in Theorem \ref{Theorem 1214}, which indicates that SDE \eqref{sde 1214} has a unique entrance measure in $\mathcal{M}_A$ where $A=\{T_k: k\geq 0\}$ and $T_k$ being defined in \eqref{0919-1}.

	To see this, set
	\begin{equation}\label{0919-1}
		S_k=-\biggl(\frac{5\pi}{6}+2k\pi\biggr)^2, \ \ T_k=-\biggl(\frac{\pi}{6}+2k\pi\biggr)^2, \ \ k=1,2,\cdots.
	\end{equation}
	Obviously for all $k\geq 1$,
	\begin{equation}
		\label{1230-1}
		\alpha_t\leq -7 \text{ on } [S_k,T_k], \ \text{ and } \ \alpha_t\leq 1 \text{ on } [T_{k+1},S_k].
	\end{equation}
	Note that 
	\begin{equation}
		\label{1230-2}
		T_k-S_k=(4k+1)\frac{2}{3}\pi^2 \ \text{ and } \ S_k-T_{k+1}=(8k+6)\frac{2}{3}\pi^2.
	\end{equation}
	Let $\Delta=\frac{\pi^2}{3}$. For any $k\geq 1$, we consider the following partition $t^k_j$ of $[T_{k+1},T_k]$: 
	\begin{equation}\label{1114-1}
		t^k_j=T_k-j\Delta, \ 0\leq j\leq 4k+1, \ t^k_{4k+2}=T_{k+1}.
	\end{equation}
	Consider the Lyapunov function $V(x)=|x|^2$. It follows from \eqref{* in Theorem of uniformly bounded solution} and \eqref{1214-1} that
	\begin{equation*}
		P(t,s)V(x)\leq e^{2\int_{s}^{t}\alpha_rdr}V(x)+129\int_{s}^{t}e^{2\int_{u}^{t}\alpha_rdr}du=:\gamma(t,s)V(x)+K(t,s).
	\end{equation*}
	Then by a simple calculation, we know that for any $k\geq 1$, $0\leq j\leq 4k+1$,
	\begin{equation*}
		\gamma(t^k_{j},t^k_{j+1})\leq e^{-14\Delta}<1 \ \text{ and } \  K(t^k_j,t^k_{j+1})\leq 10.
	\end{equation*}
	Now let $A=\{t_n\}_{n\geq 0}=\cup_{k\geq 1}\{t^k_j: 0\leq j\leq 4k+1\}$ be a decreasing sequence. Then we have $\gamma=\max_{n\geq 0}\gamma(t_n,t_{n+1})=e^{-14\Delta}<1$ and $K=\max_{n\geq 0}K(t_n,t_{n+1})=10$. Hence $A$ is a well-controlled partition. Since $\alpha_t\leq 1$ for all $t\in \mathbb{R}$, then by Theorem \ref{Theorem of lower bound density unsmooth coefficient}, we know that for any $R>0$
	\begin{equation*}
		\begin{split}
		  \inf_{t\in \mathbb{R}, V(x)\leq R}P(t+\Delta,t,x,\cdot)\geq \bar{\eta}(R,\Delta)\nu(\cdot),
		\end{split}
	\end{equation*}
	where $\bar{\eta}(R,\Delta)$ is defined as in \eqref{1214-2} and $\nu(\cdot)={\rm Leb} (\cdot\cap B_1)$. Now let us choose $R>\frac{40}{1-e^{-14\Delta}}$ and define
	\begin{equation*}
		\eta(t,s)=
		\begin{cases}
			\bar{\eta}(R,\Delta), & \text{if } |t-s|=\Delta,\\
			0, & \text{otherwise}.
		\end{cases}
	\end{equation*}
	Then for a given $0<\delta<\bar{\eta}(R,\Delta)$, we have
	\begin{equation*}
	    \sup_{n\geq 1}\bar{\gamma}^{\delta}_n=e^{-14\Delta}=:\gamma^*<1 \text{ and } \inf_{n\geq 1}\frac{n^{\delta}}{n}>\frac{1}{2}>\frac{(\gamma-1)^+R+2K}{(\gamma-1)^+R+(1-\gamma^*)^+R}.
	\end{equation*}
	Moreover, similar to \eqref{1114-2}, we know that for any $j> i$,
	\begin{equation*}
		P(t_i,t_j)V(x)=P(t_i,t_{i+1})\circ\cdots P(t_{j-1},t_j)V(x)\leq V(x)+\frac{10}{1-e^{-14\Delta}}.
	\end{equation*}
	Thus all the conditions in Theorem \ref{Theorem 1214} hold and there exists a unique entrance measure $\mu_{\cdot}$ of SDE \eqref{sde 1214} in $\mathcal{M}_{A}$.
	
	On the other hand, for any $t_n\in [T_{k+1},T_k]$,
	\begin{equation*}
		n\geq \sum_{j=1}^{k-1}(4k+1)\geq 2k(k-1), \ \ |t_n|\leq |T_{k+1}|=\biggl(\frac{\pi}{6}+2(k+1)\pi\biggr)^2,
	\end{equation*}
	thus for any $t\in \mathbb{R}$,
	\begin{equation*}
		\liminf_{n\to \infty}\frac{n}{|t-t_n|}=\liminf_{n\to \infty}\frac{n}{|t_n|}\geq \liminf_{k\to \infty}\frac{2k(k-1)}{|T_{k+1}|}=\frac{1}{2\pi^2}>0.
	\end{equation*}
	Then by Theorem \ref{Corollary 1101}, for any $\mu_1,\mu_2\in \mathcal{P}(\mathbb{R})$, there exist $\beta>0,\lambda>0$ such that for all $t\in \mathbb{R}$ and $t\geq t_n$,
	\begin{equation}\label{1114-3}
		\rho_{\beta}\bigl(P(t,t_n)\mu_1,P(t,t_n)\mu_2\bigr)\leq C_te^{-\lambda(t-t_n)}\rho_{\beta}(\mu_1,\mu_2).
	\end{equation}
	Since for any $s\in [t_{n+1},t_n]$, it is easy to show that $\gamma(t_n,s)\leq 1$ and $K(t_n,s)\leq K(t_n,t_{n+1})\leq 10$. Hence for the unique entrance measure $\mu_{\cdot}$ in $\mathcal{M}_A$,
	\begin{equation}\label{1114-4}
		\rho_{\beta}\bigl(P(t_n,s,x,\cdot),\mu_{t_n}\bigr)\leq 2+\beta P(t_n,s)V(x)+\beta \int_{\mathbb{R}}V(x)\mu_{t_n}(dx)\leq C_{\beta}(1+V(x)).
	\end{equation}
	Note that for $\lim_{s\in [t_{n+1},t_n], \ n\to \infty}\frac{t-t_n}{t-s}=1$,
	by \eqref{1114-3} and \eqref{1114-4}, we know that for any $t\geq s$,
	\begin{equation}\label{1028-2}
		\rho_{\beta}\bigl(P(t,s,x,\cdot),\mu_{t}\bigr)\leq  C_t(1+V(x))e^{-\lambda(t-s)}.
	\end{equation}

	By the same argument as above, the following SDE with degenerate noise:
	\begin{equation*}\label{SDE 0929}
		dX_t=\big(X_t-\sin^+(\sqrt{|t|})X_t^3\big)dt+\sin^+(\sqrt{|t|})dW_t
	\end{equation*}
	also has a unique entrance measure in $\mathcal{M}_A$ and \eqref{1028-2} holds.
\end{example}
\begin{remark}\label{2022b}
	If we replace $\sin^+(\sqrt{|t|})$ by $(\sin(\sqrt{|t|})-\alpha)^+$ for some $\alpha\in (0,1)$ in SDE \eqref{sde 1214}, since we can choose $C$ large enough such that
	\begin{equation*}
		xb(t,x)\leq \big(1-C(\sin(\sqrt{|t|})-\alpha)^+\big)|x|^2+\frac{C^2}{4},
	\end{equation*} 
	then arguing as in Example \ref{eg2}, we can prove that SDE \eqref{sde 1214} still has a unique entrance measure in $\mathcal{M}_A$ with geometric convergence rate where $A$ can be choosen by 
	\begin{equation*}
		A=\big\{-(\arcsin (\alpha')+2n\pi)^2: n\in \mathbb{N}\big\}, \text{ for some } \alpha<\alpha'<1.
	\end{equation*}
	It is noted that in this case there are more proportion of times at which the system expands. Actually if we take $\alpha=\frac{\sqrt{3}}{2}$, then it is easy to see that the system will be contracting when $t\in \big(-(\frac{2\pi}{3}+2k\pi)^2, -(\frac{\pi}{3}+2k\pi)^2\big)$ and expanding when $t\in \big[-(\frac{7\pi}{3}+2k\pi)^2, -(\frac{2\pi}{3}+2k\pi)^2\big]$ for all $k\in \mathbb{N}$. Hence on every time integer $\big[-(\frac{7\pi}{3}+2k\pi)^2, -(\frac{\pi}{3}+2k\pi)^2\big]$, the ratio of times at which the system expands to the times that the system is contracting is greater than $5:1$. In fact, the time intervals at which the system contracts are relatively shorter, or even much shorter when $\alpha$ is close to $1$. However over a long time horizon, the system still contracts overall.	
\end{remark}

Now we give an example which only has subgeometric contraction property and satisfies Theorem \ref{Theorem 1214}.

\begin{example}
	We consider the following 1-dimensional linear SDE:
	\begin{equation}\label{linear SDE}
		dX_t=f_{\epsilon}(t)X_tdt+dW_t,
	\end{equation}
	where $\epsilon\in (0,1)$ and $f_{\epsilon}$ is defined as follows: for all $i\geq 1$
	\begin{equation*}
		f_{\epsilon}(t)=
		\begin{cases}
			\text{locally Lebesgue integrable}, & t> -1,\\
			-1, & -i^2-i^{\epsilon}\leq t\leq -i^2,\\
			i^{-1}, & -(i+1)^2<t<-i^2-i^{\epsilon}.
		\end{cases}
	\end{equation*}
	We will prove \eqref{linear SDE} has a unique entrance measure $\mu_{\cdot}$ which is a Gaussian measure-valued function and there exist positive numbers $c_t,C_t,\lambda_1,\lambda_2,\beta$ such that
	\begin{equation}\label{1101-3}
		c_t|x|e^{-\lambda_1 |t-s|^{\frac{1+\epsilon}{2}}}\leq \rho_{\beta}\bigl(P(t,s,x,\cdot), \mu_t\bigr)\leq C_t(1+|x|^2)e^{-\lambda_2 |t-s|^{\frac{1+\epsilon}{2}}}.
	\end{equation}
	To see this, it is easy to obtain first that for all $t\geq s$, the solution $X_t^{s,x}$ has the following form
	\begin{equation}
		X_t^{s,x}=e^{\int_{s}^{t}f_{\epsilon}(r)dr}x+\int_{s}^{t}e^{\int_{u}^{t}f_{\epsilon}(r)dr}dW_u\sim \mathcal{N}\biggl(e^{\int_{s}^{t}f_{\epsilon}(r)dr}x, \int_{s}^{t}e^{2\int_{u}^{t}f_{\epsilon}(r)dr}du\biggr).
	\end{equation}
	Now let $V(x)=|x|^2$ be the Lyapunov function, we have
	\begin{equation}\label{1028-3}
		P^*(t,s)V(x)=\mathbf{E}\bigl[|X_t^{s,x}|^2\bigr]=e^{2\int_{s}^{t}f_{\epsilon}(r)dr}V(x)+\int_{s}^{t}e^{2\int_{u}^{t}f_{\epsilon}(r)dr}du=:\gamma(t,s)V(x)+K(t,s).
	\end{equation}
	For any $i\geq 1$, we consider the following partition of $[-(i+1)^2,-i^2]$: Let $m_i=\left\lceil \frac{i^{\epsilon}}{2}\right\rceil$ be the smallest integer greater than or equal to $\frac{i^{\epsilon}}{2}$. Define
	\begin{equation*}
		t^i_j:=-i^2-j,\ 0\leq j\leq m_i, \ t^i_{m_i+1}=-(i+1)^2.
	\end{equation*}
	For any $1\leq j\leq m_i$, it can be easy to check that $\gamma(t^i_{j-1},t^i_{j})=e^{-2}<1$ and $K(t^i_{j-1},t^i_{j})\leq \frac{1}{2}$. On the other hand, $\gamma(t^i_{m_i},t^i_{m_i+1})\leq e^{-i^{\epsilon}+6}$ and
	\begin{equation}\label{1028-4}
		\begin{split}
			K(t^i_{m_i},t^i_{m_i+1})&=\int_{-i^2-i^{\epsilon}}^{t^i_{m_i}}e^{-2(t^i_{m_i}-u)}du+e^{-2(i^{\epsilon}-m_i)}\int_{t^i_{m_i+1}}^{-i^2-i^{\epsilon}}e^{2i^{-1}(-i^2-i^{\epsilon}-u)}du\\
			&\leq \frac{1}{2}+\frac{ie^{-i^{\epsilon}}e^6}{2}\leq \frac{1+\epsilon^{-\frac{1}{\epsilon}}e^{-\frac{1}{\epsilon}+6}}{2}.
		\end{split}
	\end{equation}
	Let $A:=\{t_n\}_{n\geq 0}=\cup_{i=1}^{\infty}\{t^i_j: 0\leq j\leq m_i\}$ (note that $t^i_{m_i+1}=t^{i+1}_0$), then the above argument implies that $\{t_n\}_{n\geq 0}$ is a well-controlled partition with $\gamma=e^6$ and $K=\frac{1+\epsilon^{-\frac{1}{\epsilon}}e^{-\frac{1}{\epsilon}+6}}{2}<\infty$.

	Fix $R>\frac{K}{1-e^{-2}}$. Now Theorem \ref{Theorem of lower bound density unsmooth coefficient} yields that for any $1\leq j\leq m_i$
	\begin{equation*}
		\begin{split}
		  \inf_{V(x)\leq R}P(t^i_{j-1},t^i_j,x,\cdot)\geq \bar{\eta}(R,1)\nu(\cdot),
		\end{split}
	\end{equation*}
	where $\bar{\eta}(R,1)$ is defined by \eqref{1214-2}. Let us define
	\begin{equation*}
		\eta(t,s)=
		\begin{cases}
			\bar{\eta}(R,1), & t=t^i_{j-1}, s=t^i_{j} \text{ for some } i\geq 1, \ 1\leq j\leq m_i,\\
			0, & \text{otherwise}.
		\end{cases}
	\end{equation*}
	Choose $\delta<\bar{\eta}(R,1)$, then we have
	\begin{equation*}
		\bar{\gamma}_n^{\delta}=e^{-2}=:\gamma^*<1 \ \text{ and } \ \lim_{n\to \infty}\frac{n^{\delta}}{n}=1>\frac{(\gamma-1)^+R+2K}{(\gamma-1)^+R+(1-\gamma^*)R}.
	\end{equation*}
	Note that for all $t^i_{m_i}\leq t\leq t^i_0$, $\gamma(t,t^{i+1}_0)\leq e^{-i^{\epsilon}+6}$ and similar to \eqref{1028-4},  $K(t,t^{i+1}_0)\leq K$. Note also that $\gamma_i:=\gamma(t^i_0,t^{i+1}_0)\leq e^{-2i^{\epsilon}+4}\leq e^{-2}$ for all $i\geq 3$, then similar to \eqref{1114-2}, for any $i'>i\geq 3$,
	\begin{equation*}
		\begin{split}
			P(t^i_0,t^{i'}_0)V(x)&= P(t^i_0,t^{i+1}_0)\circ P(t^{i+1}_0,t^{i+2}_0)\circ \cdots P(t^{i'-1}_0,t^{i'}_0)V(x)\\
			&\leq \gamma_i\cdots\gamma_{i'-1}V(x)+K_0(1+\gamma_i+\gamma_i\gamma_{i+1}+\cdots+\gamma_i\cdots\gamma_{i'-2})\\
			&\leq e^{-2(i'-i)}V(x)+\frac{K}{1-e^{-2}}.
		\end{split}
	\end{equation*}
	Hence \eqref{0918-4} follows by
	\begin{equation*}
		P(t^i_j,t^{i'}_{j'})V(x)=P(t^i_j,t^{i+1}_0)\circ P(t^{i+1}_0,t^{i'}_{0})\circ P(t^{i'}_{0},t^{i'}_{j'})V(x)\leq e^{12}V(x)+\Bigl(1+e^6+\frac{1}{1-e^{-2}}\Bigr)K.
	\end{equation*}
	Then by Theorem \ref{Theorem 1214}, there exist a unique entrance measure $\mu_{\cdot}$ in $\mathcal{M}_{A}$ and it can be seen that the entrance measure $\mu_t$ is given by
	\begin{equation*}
		\mu_t(\cdot)=\mathcal{N}\biggl(0, \int_{-\infty}^{t}e^{2\int_{u}^{t}f_{\epsilon}(r)dr}du\biggr)(\cdot).
	\end{equation*}
	Note that for $t_n=t^i_j$, we have
	\begin{equation*}
		n=\sum_{k=1}^{i-1}(m_k+1)+(j+1)> \frac{1}{2}\sum_{k=1}^{i-1}k^{\epsilon}\geq \frac{1}{2(1+\epsilon)}(i-1)^{1+\epsilon}.
	\end{equation*}
	Set $\phi(t)=t^{\frac{1+\epsilon}{2}}$ for $t\geq 0$. Then it is easy to check that for any $t\in \mathbb{R}$,
	\begin{equation*}
		\liminf_{n\to \infty}\frac{n}{\phi(t-t_n)}\geq \frac{1}{2(1+\epsilon)}>0.
	\end{equation*}
	By \eqref{1027-2} in Theorem \ref{Corollary 1101}, there exist $\beta>0,C_t>0,\lambda>0$ such that for all $t\geq t_n$ and $\mu_1,\mu_2\in \mathcal{P}(\mathbb{R})$,
	\begin{equation}\label{1028-5}
		\rho_{\beta}\big(P^*(t,t_n)\mu_1,P^*(t,t_n)\mu_2\big)\leq C_t e^{-\lambda |t-t_n|^{\frac{1+\epsilon}{2}}}\rho_{\beta}(\mu_1,\mu_2).
	\end{equation}
	Note that for any $s\in (t_{n+1},t_{n})$, we have $\gamma(t_n,s)\leq e^4$ and $K(t_n,s)\leq K$. 
	Then 
	\begin{equation*}
		\begin{split}
			\rho_{\beta}\bigl(P(t_n,s,x,\cdot), \mu_{t_n}\bigr)\leq 2+\beta \gamma(t_n,s)|x|^2+K(t_n,s)+\beta\int_{\mathbb{R}}|x|^2\mu_{t_n}(dx)
			\leq C_{\beta}(1+|x|^2),
		\end{split}
	\end{equation*}
	where $\mu_{\cdot}$ is the unique entrance measure in $\mathcal{M}_A$. Hence for any $s\in (t_{n+1},t_n]$
	\begin{equation*}
		\begin{split}
			\rho_{\beta}\bigl(P(t,s,x,\cdot), \mu_{t}\bigr)&=\rho_{\beta}\bigl(P^*(t,t_n)P(t_n,s,x,\cdot), P^*(t,t_n)\mu_{t_n}\bigr)\\
			&\leq C_t e^{-\lambda |t-t_n|^{\frac{1+\epsilon}{2}}}\rho_{\beta}\bigl(P(t_n,s,x,\cdot), \mu_{t_n}\bigr)\\
			&\leq C_t (1+|x|^2) e^{-\lambda |t-t_n|^{\frac{1+\epsilon}{2}}}.
		\end{split}
	\end{equation*}
	Since 
	\begin{equation*}
		\lim_{s\in (t_{n+1},t_n], n\to \infty}\frac{t-t_n}{t-s}=1,
	\end{equation*}
	then there exists $0<\lambda_2<\lambda$ such that for all $t\geq s$,
	\begin{equation}\label{1101-1}
		\rho_{\beta}\bigl(P(t,s,x,\cdot), \mu_{t}\bigr)\leq C_t (1+|x|^2) e^{-\lambda_2 |t-s|^{\frac{1+\epsilon}{2}}}.
	\end{equation}

	In the following, we will show that the order $\frac{1+\epsilon}{2}$ is sharp! Similar to \eqref{1226-1}, we have
	\begin{equation*}
		\begin{split}
			\rho_{\beta}\big(P(t,s,x,\cdot),\mu_t\big)
			\geq 2\sqrt{\beta}|x|e^{\int_{s}^{t}f_{\epsilon}(r)dr}.
		\end{split}
	\end{equation*}
	Now for any $t\in \mathbb{R}$, there exists $i_t\geq 1$ such that $t\geq t^{i_t}_0=-i_t^2$. If $s\in [-(i+1)^2,-i^2]$ for some $i\geq i_t$, we have
	\begin{equation*}
		\int_{s}^{t}f_{\epsilon}(r)dr\geq\int_{-i_t^2}^{t}f_{\epsilon}(r)dr-\sum_{j=i_t}^{i-1}j^{\epsilon}\geq\int_{-i_t^2}^{t}f_{\epsilon}(r)dr-\frac{1}{1+\epsilon}(i^{1+\epsilon}-i_t^{1+\epsilon}).
	\end{equation*}
	Since $i^2\leq |s|\leq |i+1|^2$, we have
	\begin{equation}
		\liminf_{s\to -\infty}\frac{\int_{s}^{t}f_{\epsilon}(r)dr}{|t-s|^{\frac{1+\epsilon}{2}}}\geq -\frac{1}{1+\epsilon}.
	\end{equation}
	Hence there exist $c_t>0$ and $\lambda_1>\frac{1}{1+\epsilon}$ such that for all $t\geq s$,
	\begin{equation}\label{1101-2}
		\rho_{\beta}\big(P(t,s,x,\cdot),\mu_t\big)\geq c_t|x|e^{-\lambda_1|t-s|^{\frac{1+\epsilon}{2}}}.
	\end{equation}
	From \eqref{1101-1} and \eqref{1101-2}, the claim \eqref{1101-3} follows.
\end{example}

\section{Ergodicity of SDEs with quasi-periodic coefficients}\label{zhao22a}

\subsection{Existence and uniqueness of quasi-periodic measures and reparameterizations}

  Quasi-periodicity is a main topic among the study of dynamical systems (e.g. KAM theory). As an important class of time-inhomogeneous Markov processes, random quasi-periodic process has begun to attract attentions of mathematicians (\cite{feng2021random,Liu-Lu2022}). This kind of processes is not only theoretically interesting, but also appears in nature, e.g. temperature variants of the day-night periodicity and the season periodicity, let alone the large scale climate change of glacial and ice-age climates that also happens periodically in the sense of random periodicity (\cite{benzi1982stochastic,feng2019existence,Feng-Zhao-Zhong2021}). We begin with the definition of quasi-periodic measure first given in \cite{feng2021random}. In the following we only study the case of two periods for simplicity of the descriptions. All the results obtained here are still true if more periods are considered. 
  
  In the following, unless otherwise specified, a probability measure-valued function being continuous always means that it is continuous in the sense of weak convergence.

  \begin{definition}
	We say a measure-valued map $\mu: \mathbb{R}\rightarrow \mathcal{P}(\mathbb{X})$ on a metric space $\mathbb{X}$ is a quasi-periodic measure of a time-inhomogeneous Markovian semigroup $P$ with periods $\tau_1, \tau_2$, where the reciprocals of $\tau_1,\tau_2$ are rationally linearly independent, if it is an entrance measure and there exists a continuous measure-valued function $\tilde{\mu}:\mathbb{R}\times \mathbb{R}\rightarrow \mathcal{P}(\mathbb{X})$ such that for all $t,t_1,t_2\in \mathbb{R}$,
	\begin{equation}
		\label{* in Definition of quasi-periodic measure}
		\tilde{\mu}_{t,t}=\mu_t \ \text{ and } \ \tilde{\mu}_{t_1+\tau_1,t_2}=\tilde{\mu}_{t_1,t_2}=\tilde{\mu}_{t_1,t_2+\tau_2}.
	\end{equation}
	Moreover, if the corresponding Markovian transition function $P(t,s,x,\cdot)$ is the transition kernel of an SDE, we also say that this SDE has a quasi-periodic measure $\mu$.
\end{definition}


  In this section, we always assume the coefficients $b,\sigma$ satisfy the following quasi-periodic condition:
  \begin{condition}
  \label{A4}
    The coefficients $b,\sigma$ are quasi-periodic with periods $\tau_1,\tau_2$, where the reciprocals of $\tau_1,\tau_2$ are rationally linearly independent, i.e., there exist continuous $\tilde{b}:\mathbb{R}\times \mathbb{R}\times \mathbb{R}^d\to \mathbb{R}^d$ and $\tilde{\sigma}:\mathbb{R}\times \mathbb{R}\times \mathbb{R}^d\to \mathbb{R}^{d\times d}$ such that 
		\begin{equation*}
			\label{Quasi-periodic condition}
			\tilde{b}(t,t,x)=b(t,x),\ \ \tilde{\sigma}(t,t,x)=\sigma(t,x), \ \text{ for all } t\in \mathbb{R},\ x\in \mathbb{R}^d,
		\end{equation*}
		and $\tilde{b}(t_1,t_2,x),\ \tilde{\sigma}(t_1,t_2,x)$ are periodic in $(t_1,t_2)$ with periods $\tau_1,\tau_2$ respectively.

		Moreover, when Assumption \ref{A3} holds, $\alpha,\Lambda$ given in \eqref{Ineq weakly coercivity} are also quasi-periodic with periods $\tau_1,\tau_2$, i.e., there exist continuous $\tilde{\alpha}: \mathbb{R}\times\mathbb{R}\to \mathbb{R}$ and $\tilde{\Lambda}: \mathbb{R}\times\mathbb{R}\to \mathbb{R}^+$ such that $\tilde{\alpha}_{t,t}=\alpha_t, \tilde{\Lambda}_{t,t}=\Lambda_t$ and $\tilde{\alpha}_{t_1,t_2}, \tilde{\Lambda}_{t_1,t_2}$ are periodic with periods $\tau_1,\tau_2$ respectively.
  \end{condition}
  
  For convenience, $\tilde b, \tilde \sigma$ etc. in the above sense are called unfolding multi-time parent functions or parent functions of $b,\sigma $ etc. respectively.
  
  \begin{lemma}\label{Remark 0919-1}
	If Assumptions \ref{A3} and \ref{A4} hold, then $\tilde{b}(t_1,t_2,\cdot), \tilde{\sigma}(t_1,t_2,\cdot)$ satisfy \eqref{Ineq of non-degenerate diffusion}-\eqref{Ineq polynomial growth} with the same $\Gamma_1, \Gamma_2, \kappa$ and satisfy \eqref{Ineq weakly coercivity} with $\tilde{\alpha}_{t_1,t_2}, \tilde{\Lambda}_{t_1,t_2}$, where $\tilde{\alpha}_{t_1,t_2}, \tilde{\Lambda}_{t_1,t_2}$ satisfy \eqref{1228-1} with the same function $g$. If moreover Assumption \ref{A3 new} holds, $\tilde{\sigma}(t_1,t_2,\cdot)$ satisfies \eqref{Ineq of non-degenerate diffusion 2} with the same $\Gamma_1$.
  \end{lemma}
  \begin{proof}
	Since $\tau_1,\tau_2$ are rationally linearly independent, then $\{(t\mod \tau_1, t\mod \tau_2): t\in \mathbb{R}^+\}$ is dense in $[0,\tau_1)\times [0,\tau_2)$. Then for any $t_1,t_2\in \mathbb{R}$, there exist a sequence $\{s_n\}_{n\geq 1}\subset \mathbb{R}^+$ such that $(s_n\mod \tau_1, s_n\mod \tau_2)\to (t_1\mod \tau_1, t_2\mod \tau_2)$. By the periodicity and continuity of $\tilde{\sigma}$, we have for all $x, \xi\in \mathbb{R}^d$
	\begin{equation*}
	  \begin{split}
		  \langle \tilde{\sigma}\tilde{\sigma}^{\top}(t_1,t_2,x)\xi, \xi \rangle
		  &=\langle \tilde{\sigma}\tilde{\sigma}^{\top}(t_1\mod \tau_1, t_2\mod \tau_2,x)\xi, \xi \rangle\\
		  &=\lim_{n\to \infty}\langle \tilde{\sigma}\tilde{\sigma}^{\top}(s_n\mod \tau_1, s_n\mod \tau_2,x)\xi, \xi \rangle\\
		  &=\lim_{n\to \infty}\langle \tilde{\sigma}\tilde{\sigma}^{\top}(s_n, s_n,x)\xi, \xi \rangle\\
		  &=\lim_{n\to \infty}\langle \sigma\sigma^{\top}(s_n,x)\xi, \xi \rangle.
	  \end{split}
	\end{equation*}
	Hence we have
	\begin{equation*}
	  \Gamma_1^{-1}|\xi|^2\leq\langle \tilde{\sigma}\tilde{\sigma}^{\top}(t_1,t_2,x)\xi, \xi \rangle\leq \Gamma_1|\xi|^2.
	\end{equation*}
	and
	\begin{equation*}
	  \|\tilde{\sigma}(t_1,t_2,x)-\tilde{\sigma}(t_1,t_2,y)\|\leq \Gamma_1|x-y|.
	\end{equation*}
	Similarly, the periodicity and continuity of $\tilde{b}, \tilde{\alpha}$ and $\tilde{\Lambda}$ yield that $\tilde{b}(t_1,t_2,x)$ is also locally Lipschitz continuous in $x$ uniformly in $(t_1,t_2)$ and
	\begin{equation*}
		|\tilde{b}(t_1,t_2,x)|\leq \Gamma_2(1+|x|^{\kappa}),
	\end{equation*}
	\begin{equation*}
	  \langle x, \tilde{b}(t_1,t_2,x)\rangle \leq \tilde{\alpha}_{t_1,t_2} |x|^2+\tilde{\Lambda}_{t_1,t_2}.
  \end{equation*}
  Moreover, for all $s\leq t, r_1,r_2\in \mathbb{R}$,
  \begin{equation*}
	  \int_{s}^{t}(\tilde{\alpha}^+_{r+r_1,r+r_2}+\tilde{\Lambda}_{r+r_1,r+r_2})dr\leq g(t-s).
  \end{equation*}
  \end{proof}

  \begin{lemma}\label{Remark 0919-2}
	  Assumption \ref{A1 1203} is equivalent to
	  \begin{equation*}
		\int_{0}^{\tau_1}\int_{0}^{\tau_2}\tilde{\alpha}_{t_1,t_2}dt_1dt_2<0.
	  \end{equation*}
	  Moreover, under this assumption, there exist $\Delta>0$ and $\bar{\alpha}<0$ such that
	  \begin{equation}\label{1017-2}
		\int_{t}^{t+\Delta}\tilde{\alpha}_{r+r_1,r+r_2}dr\leq \bar{\alpha}, \ \text{ for all } t,r_1,r_2\in \mathbb{R}.
	  \end{equation}
  \end{lemma}
  
  \begin{proof}
	  Since $\tau_1,\tau_2$ are rationally linearly independent and $\tilde{\alpha}$ is continuous, by Definition 5.1 in \cite{walters2000introduction}, $T_t: [0,\tau_1)\times [0,\tau_2) \rightarrow [0,\tau_1)\times [0,\tau_2)$ defined by 
		  \begin{equation}\label{0919-2}
			T_t(s_1,s_2)=(t+s_1\mod \tau_1,\ t+s_2\mod \tau_2), \text{ for all } s_1,s_2\in [0,\tau_1)\times [0,\tau_2)
		  \end{equation}
          is a minimal rotation. Then applying Theorem 6.20 in \cite{walters2000introduction}, we know that $\frac{1}{\tau_1\tau_2}{\rm Leb}$ is a unique ergodic probability measure on $[0,\tau_1)\times [0,\tau_2)$, where ${\rm Leb}(\cdot)$ represents the Lebesgue measure on $[0,\tau_1)\times [0,\tau_2)$. Note that $\tilde{\alpha}$ is continuous, then by Birkhoff's ergodic theorem, Assumption \ref{New A5} is equivalent to the following:
	  \begin{equation}\label{1005-2}
		\frac{1}{\tau_1\tau_2}\int_{0}^{\tau_1}\int_{0}^{\tau_2}\tilde{\alpha}_{t_1,t_2}dt_1dt_2=\lim_{T\to \infty}\frac{1}{T}\int_{0}^{T}\alpha_tdt=\lim_{T\to \infty}\frac{1}{T}\int_{-T}^{0}\alpha_tdt<0.
	  \end{equation}
	  Moreover, it can be shown that if \eqref{1005-2} holds true, then $\int_{t}^{t+\Delta}\tilde{\alpha}_{r+r_1,r+r_2}dr\leq \bar{\alpha}$ for some $\Delta>0, \bar{\alpha}<0$ and for all $t,r_1,r_2\in \mathbb{R}$. In fact, set $$\alpha^*=\frac{1}{3\tau_1\tau_2}\int_{0}^{\tau_1}\int_{0}^{\tau_2}\tilde{\alpha}_{t_1,t_2}dt_1dt_2<0.$$
	  By the periodicity and continuity of $\tilde{\alpha}$, the restriction of $\tilde{\alpha}$ on $[0,\tau_1)\times [0,\tau_2)$ is continuous under the metric $d_0((r_1,r_2),(r_1',r_2'))=d_1(r_1,r_1')+d_2(r_2,r_2')$, where $d_i$ defined as follows:
	  \begin{equation}
		\label{Metric on torus}
		d_i(r_i,r_i')=\min(|r_i-r_i'|, \tau_i-|r_i-r_i'|), \text{ for all } r_i,r_i'\in [0,\tau_i), i=1,2.
	  \end{equation}
	  Hence there exists $\epsilon>0$ such that $|\tilde{\alpha}_{r_1,r_2}-\tilde{\alpha}_{r_1',r_2'}|\leq |\alpha^*|$ when $d_0((r_1,r_2),(r_1',r_2'))<\epsilon$.
	  Note that $[0,\tau_1)\times [0,\tau_2)$ under $d_0$ is compact (a torus), there exists a finite $\epsilon$-net $\{(r_1^k,r_2^k)\}_{k=1}^{N}$, i.e., for all $(r_1,r_2)\in [0,\tau_1)\times [0,\tau_2)$, there exists $1\leq k\leq N$ such that $d_0((r_1,r_2),(r_1^k,r_2^k))<\epsilon$. 
	  
	  Using the minimal rotation $T_t$ and Birkhoff's ergodic theorem again, there exists $\Delta>0$ big enough such that for all $1\leq k\leq N$,
	  \begin{equation*}
		  \frac{1}{\Delta}\int_{0}^{\Delta}\tilde{\alpha}_{r+r_1^k,r+r_2^k}dr\leq \frac{2}{3\tau_1\tau_2}\int_{0}^{\tau_1}\int_{0}^{\tau_2}\tilde{\alpha}_{t_1,t_2}dt_1dt_2=2\alpha^*.
	  \end{equation*}
	  For any $t,r_1,r_2\in \mathbb{R}$, let $r_i':=t+r_i\mod \tau_i, i=1,2$. Then there exists $1\leq k'\leq N$ such that $d_0((r_1',r_2'),(r_1^{k'},r_2^{k'}))<\epsilon$. Hence
	  \begin{equation*}
		\begin{split}
			&\ \ \ \ \frac{1}{\Delta}\int_{t}^{t+\Delta}\tilde{\alpha}_{r+r_1,r+r_2}dr=\frac{1}{\Delta}\int_{0}^{\Delta}\tilde{\alpha}_{r+r_1',r+r_2'}dr\\
			&\leq \frac{1}{\Delta}\int_{0}^{\Delta}|\tilde{\alpha}_{r+r_1',r+r_2'}-\tilde{\alpha}_{r+r_1^{k'},r+r_2^{k'}}|dr+\frac{1}{\Delta}\int_{0}^{\Delta}\tilde{\alpha}_{r+r_1^{k'},r+r_2^{k'}}dr\\
			&\leq |\alpha^*|+2\alpha^*=\alpha^*.
		\end{split}
	  \end{equation*}
	  Therefore \eqref{1017-2} holds by letting $ \bar{\alpha}:=\alpha^*\Delta$.
  \end{proof}

  Now we have the following theorem. 

  \begin{theorem}
	\label{Theorem of Existence and uniqueness of quasi-periodic measure}
	  Under Assumptions \ref{A3}, \ref{A3 new}, \ref{A1 1203} and \ref{A4}, SDE \eqref{SDE} has a unique continuous quasi-periodic measure $\mu$ in $\mathcal{M}$ with geometric convergence rate in the $\rho_{\beta}$ distance for some $\beta>0$. Moreover, there exists a unique continuous $\tilde{\mu}:\mathbb{R}\times \mathbb{R}\to \mathcal{P}(\mathbb{R}^d)$ satisfying \eqref{* in Definition of quasi-periodic measure}. 
  \end{theorem}

  We have already obtained the entrance measure with geometric convergence rate. It remains to prove this entrance measure is also a quasi-periodic measure. The key is to study the following reparameterized stochastic differential equation
	\begin{equation}
		\label{Equation K_r_1,r_2}
		\begin{split}
		K^{r_1,r_2}(t,s,x)
		=x+\int_s^t\tilde{b}^{r_1,r_2}(v,K^{r_1,r_2}(v,s,x))dv+\int_s^t\tilde{\sigma}^{r_1,r_2}(v,K^{r_1,r_2}(v,s,x))dW_v,
		\end{split}
	\end{equation}
	where $\tilde{b}^{r_1,r_2}(t,x):=\tilde{b}(t+r_1,t+r_2,x), \ \tilde{\sigma}^{r_1,r_2}(t,x):=\tilde{\sigma}(t+r_1,t+r_2,x)$. Equation \eqref{Equation K_r_1,r_2} was introduced in Theorem 3.5 in \cite{feng2021random}. By Lemma \ref{Remark 0919-1}, we know that if Assumptions \ref{A3} and \ref{A4} hold, then $\tilde{b}^{r_1,r_2}, \ \tilde{\sigma}^{r_1,r_2}$ satisfy Assumption \ref{A3} for all fixed $(r_1, r_2)$. Hence SDE \eqref{Equation K_r_1,r_2} has a unique solution $K^{r_1,r_2}(t,s,x)$ with any initial condition $(s,x)$. 
	
	Denote by $P^{r_1,r_2}(t,s,x,\cdot)$ the transition function of the unique solution $K^{r_1,r_2}(t,s,x)$, i.e.
	\begin{equation*}
		P^{r_1,r_2}(t,s,x,A):=\mathbf{P}\{K^{r_1,r_2}(t,s,x)\in A\}, \ \text{ for all } t\geq s, \ x\in \mathbb{R}^d, \ A\in \mathcal{B}(\mathbb{R}^d).
	\end{equation*}
	Similarly denote by $P^{r_1,r_2}(t,s)$ the kernel acting as a linear operator on functions $f: \mathbb{R}^d\to \mathbb{R}$ and $P^{r_1,r_2,*}(t,s)$ the dual operator of $P^{r_1,r_2}(t,s)$ acting on measures in $\mathcal{P}(\mathbb{R}^d)$.

	Before giving the proof of Theorem \ref{Theorem of Existence and uniqueness of quasi-periodic measure}, we first state the following lemma.

	\begin{lemma}
		\label{Lemma of convergence in probability of K}
		Suppose Assumptions \ref{A3}, \ref{A4} hold. Then for any $f\in C_b(\mathbb{R}^d)$ and $t\geq s$,  we have 
		\begin{equation*}
			\label{* in Lemma of convergence in probability of K}
			\lim_{(r'_1,r'_2,x')\to (r_1,r_2,x)}\mathbf{E}\left[ f\left( K^{r'_1,r'_2}(t,s,x') \right) \right]=\mathbf{E}\left[ f\left( K^{r_1,r_2}(t,s,x) \right) \right].
		\end{equation*}
	\end{lemma}

	\begin{proof}
		Without loss of generality, we always assume that 
		\begin{equation}
			\label{1 in Lemma of convergence in probability of K}
			(r'_1,r'_2,x')\in [r_1-1,r_1+1]\times [r_2-1,r_2+1]\times B_1(x),
		\end{equation}
		where $B_1(x):=\{y\in \mathbb{R}^d: |y-x|\leq 1\}$. Denote
		\begin{equation*}
			\label{2 in Lemma of convergence in probability of K}
			T_N^{s,x'}:=\inf\{t\geq s: |K^{r'_1,r'_2}(t,s,x')|> N\} \text{ and } T_N^{s,x}:=\inf\{t\geq s: |K^{r_1,r_2}(t,s,x)|> N\}.
		\end{equation*}
		Since $\tilde{b}^{r_1,r_2}, \ \tilde{\sigma}^{r_1,r_2}$ satisfy Assumption \ref{A3} for all fixed $(r_1, r_2)$, then by \eqref{Estimate of stopping time T_N} in Theorem \ref{Theorem of uniformly bounded solution} and \eqref{1 in Lemma of convergence in probability of K} we know that
		\begin{equation}
			\label{3 in Lemma of convergence in probability of K}
			\mathbf{P}\{T_N^{s,y}\leq t\}\leq \frac{1}{N^{2}}e^{2g(t-s)}\bigg( (|x|+1)^2+\int_{s}^{t}e^{-2\int_{s}^{u}\alpha_rdr}(2\Lambda_u+d\Gamma_1)du \bigg), \ \text{ for all } y\in B_1(x).
		\end{equation}
    Now let 
    \begin{align*}
      \hat{b}_N^{r',r}:&=\sup_{t\in \mathbb{R}, |x|\leq N}|\tilde{b}^{r'_1,r'_2}(t,x)-\tilde{b}^{r_1,r_2}(t,x)|\\
      &= \sup_{t\in \mathbb{R}, |x|\leq N}|\tilde{b}(t+r'_1,t+r'_2,x)-\tilde{b}(t+r_1,t+r_2,x)|,
    \end{align*}
    and
    \begin{align*}
      \hat{\sigma}_N^{r',r}:&=\sup_{t\in \mathbb{R}, |x|\leq N}\|\tilde{\sigma}^{r'_1,r'_2}(t,x)-\tilde{\sigma}^{r_1,r_2}(t,x)\|\\
      &= \sup_{t\in \mathbb{R}, |x|\leq N}\|\tilde{\sigma}(t+r'_1,t+r'_2,x)-\tilde{\sigma}(t+r_1,t+r_2,x)\|.
    \end{align*}
		Since $\tilde{b}(t_1,t_2,x)\in C(\mathbb{R}\times\mathbb{R}\times\mathbb{R}^d)$, we know that for any fixed $N>0$, $\lim_{(r'_1,r'_2)\to (r_1,r_2)}\hat{b}_N^{r',r}=0$. Similarly $\tilde{\sigma}$ is continuous and hence we also obtain that $\lim_{(r'_1,r'_2)\to (r_1,r_2)}\hat{\sigma}_N^{r',r}=0$.

		Let $T_N^{s,x',x}:=T_N^{s,x'}\wedge T_N^{s,x}$. Then it is easy to check that
		\begin{equation*}
			\mathbf{P}\{T_N^{s,x',x}\leq t\}=\mathbf{P}\left( \{T_N^{s,x'}\leq t\}\cup \{T_N^{s,x}\leq t\} \right)\leq \mathbf{P}\{T_N^{s,x'}\leq t\}+\mathbf{P}\{T_N^{s,x}\leq t\}.
		\end{equation*}
		Thus by \eqref{3 in Lemma of convergence in probability of K}, for any $\epsilon>0$, we can choose $N$ big enough such that $\mathbf{P}\{T_N^{s,x',x}\leq t\}<\epsilon$. Hence
		\begin{equation}
			\label{4 in Lemma of convergence in probability of K}
			\begin{split}
				&\quad \ \Big|\mathbf{E}\left[ f\left( K^{r'_1,r'_2}(t,s,x') \right) \right]-\mathbf{E}\left[ f\left( K^{r_1,r_2}(t,s,x) \right) \right]\Big|\\
				&\leq \mathbf{E}\left[ \Big|f\left( K^{r'_1,r'_2}(t,s,x')\right)-f\left( K^{r_1,r_2}(t,s,x)  \right)\Big|1_{\{T_N^{s,x',x}> t\}}  \right]\\
				&\quad \ + \mathbf{E}\left[ \Big|f\left( K^{r'_1,r'_2}(t,s,x')\right)-f\left( K^{r_1,r_2}(t,s,x)  \right)\Big|1_{\{T_N^{s,x',x}\leq t\}}  \right]\\
				&\leq \mathbf{E}\left[ \Big|f\left( K^{r'_1,r'_2}\big(t\wedge T_N^{s,x',x},s,x'\big)\right)-f\left( K^{r_1,r_2}\big(t\wedge T_N^{s,x',x},s,x\big)  \right)\Big| \right]+2|f|_{\infty}\epsilon.
			\end{split}
		\end{equation}
		Since $\tilde{b}(t_1,t_2,x)$ is locally Lipschitz continuous in $x$ uniformly in $(t_1,t_2)$, denote by $L_N$ the Lipschitz constant on $B_N:=\{x\in \mathbb{R}^d: |x|\leq N\}$, i.e. for all $x, y\in B_N$ and $t_1,t_2\in \mathbb{R}$, $|\tilde{b}(t_1,t_2,x)-\tilde{b}(t_1,t_2,y)|\leq L_N|x-y|$. Then by standard estimate of It$\rm \hat{o}$ process we have
		\begin{align*}
			&\quad \ \mathbf{E}\left[ \Big|K^{r'_1,r'_2}\big(t\wedge T_N^{s,x',x},s,x'\big)-K^{r_1,r_2}\big(t\wedge T_N^{s,x',x},s,x\big)\Big|^2 \right]\\
			&\leq \left( 3|x'-x|^2+6(\hat{b}_N^{r',r})^2(t-s)^2+6(\hat{\sigma}_N^{r',r})^2(t-s) \right)\\
			&\quad \ +6\left( L_N^2(t-s)+\Gamma_1^2 \right)\int_s^t\mathbf{E}\left[ \Big|K^{r'_1,r'_2}\big(u\wedge T_N^{s,x',x},s,x'\big)-K^{r_1,r_2}\big(u\wedge T_N^{s,x',x},s,x\big)\Big|^2\right]du.
		\end{align*}
		Thus by the Gronwall inequality,
		\begin{equation}
			\label{5 in Lemma of convergence in probability of K}
			\begin{split}
				&\ \ \ \ \mathbf{E}\left[ \Big|K^{r'_1,r'_2}\big(t\wedge T_N^{s,x',x},s,x'\big)-K^{r_1,r_2}\big(t\wedge T_N^{s,x',x},s,x\big)\Big|^2\right]\\
				&\leq \left( 3|x'-x|^2+6(\hat{b}_N^{r',r})^2(t-s)^2+6(\hat{\sigma}_N^{r',r})^2(t-s) \right)e^{6\left( L_N^2(t-s)+\Gamma_1^2 \right)(t-s)}\\
				&\to 0,
			\end{split}
		\end{equation}
		as $(r'_1,r'_2,x')\to (r_1,r_2,x)$.
		Note that $f$ is uniformly continuous on $B_N$, hence there exists $\delta>0$ such that $|f(x)-f(y)|<\epsilon$ for all $x,y\in B_N$ with $|x-y|<\delta$. Then it follows from \eqref{4 in Lemma of convergence in probability of K}, \eqref{5 in Lemma of convergence in probability of K} and the Chebyshev inequality that
		\begin{align*}
			&\quad \ \limsup_{(r'_1,r'_2,x')\to (r_1,r_2,x)}\Big|\mathbf{E}\left[ f\left( K^{r'_1,r'_2}(t,s,x') \right) \right]-\mathbf{E}\left[ f\left( K^{r_1,r_2}(t,s,x) \right) \right]\Big|\\
      &\leq (1+2|f|_{\infty})\epsilon\\
      &\quad \ +\frac{2|f|_{\infty}}{\delta^2}\limsup_{(r'_1,r'_2,x')\to (r_1,r_2,x)}\mathbf{E}\left[ \Big|K^{r'_1,r'_2}\big(t\wedge T_N^{s,x',x},s,x'\big)-K^{r_1,r_2}\big(t\wedge T_N^{s,x',x},s,x\big)\Big|^2\right]\\
			&\leq (1+2|f|_{\infty})\epsilon.
		\end{align*}
		Then the desired result follows from the arbitrary choice of $\epsilon>0$.
	\end{proof}

	Now we give the proof of Theorem \ref{Theorem of Existence and uniqueness of quasi-periodic measure}.

	\begin{proof}[Proof of Theorem \ref{Theorem of Existence and uniqueness of quasi-periodic measure}]
		Uniqueness: Since any quasi-periodic measure of SDE \eqref{SDE} will be an entrance measure, uniqueness follows from the uniqueness of entrance measure in Theorem \ref{Theorem existence and uniqueness of entrance measure}.

		Existence: It remains to prove that the entrance measure $\mu$ in Theorem \ref{Theorem existence and uniqueness of entrance measure} is indeed a quasi-periodic measure. First similar to the proof of Theorem 3.5 in \cite{feng2021random}, we know that for all $r,r_1,r_2\in \mathbb{R}, \ t\geq s, x\in \mathbb{R}^d$,
		\begin{equation*}
			K^{r,r}(t,s,x)=X_{t+r}^{s+r,x}\circ \theta_{-r}, \ \mathbf{P}-a.s.,
		\end{equation*}
		and
		\begin{equation*}
			K^{r_1+\tau_1,r_2}(t,s,x)=K^{r_1,r_2}(t,s,x)=K^{r_1,r_2+\tau_2}(t,s,x), \ \mathbf{P}-a.s.,
		\end{equation*}
		where $\theta_{\cdot}$ is the Brownian shift, i.e. $(\theta_r\omega)(s)=\omega(r+s)-\omega(r), \ s,r\in \mathbb{R}$. Then we have
		\begin{equation}
			\label{1 in Theorem of Existence and uniqueness of quasi-periodic measure}
			P^{r,r}(t,s,x,\cdot)=P(t+r,s+r,x,\cdot), 
		\end{equation}
		and
		\begin{equation}
			\label{2 in Theorem of Existence and uniqueness of quasi-periodic measure}
			P^{r_1+\tau_1,r_2}(t,s,x,\cdot)=P^{r_1,r_2}(t,s,x,\cdot)=P^{r_1,r_2+\tau_2}(t,s,x,\cdot).
		\end{equation}
		By Lemmas \ref{Remark 0919-1}, \ref{Remark 0919-2}, we know that $\tilde{b}^{r_1,r_2}, \ \tilde{\sigma}^{r_1,r_2}$ satisfy Assumptions \ref{A3} and \ref{A3 new} with parameters $(\Gamma,\kappa,\tilde{\alpha}^{r_1,r_2}, \tilde{\Lambda}^{r_1,r_2})$ where for all $r_1,r_2\in \mathbb{R}$, 
		\begin{equation*}
			\tilde{\alpha}^{r_1,r_2}_t:=\tilde{\alpha}_{t+r_1,t+r_2}, \ \tilde{\Lambda}^{r_1,r_2}_t:=\tilde{\Lambda}_{t+r_1,t+r_2}
		\end{equation*}
		are dominated by the same function $g$ and $\tilde{\alpha}^{r_1,r_2}$ satisfies Assumption \ref{New A2} with the same $\Delta, \bar{\alpha}$,
		then by Theorem \ref{Theorem existence and uniqueness of entrance measure} and Remark \ref{New remark}, we know that for any $r_1,r_2\in \mathbb{R}$, SDE \eqref{Equation K_r_1,r_2} has a unique entrance measure $\mu^{r_1,r_2}$ and there exist $C>0,\lambda>0$ such that for all $r_1,r_2\in \mathbb{R}$, $t\geq s$,
		\begin{equation}
			\label{3 in Theorem of Existence and uniqueness of quasi-periodic measure}
			\|P^{r_1,r_2}(t,s,x,\cdot)-\mu^{r_1,r_2}_t\|_{TV}\leq C(1+|x|^2)e^{-\lambda(t-s)}.
		\end{equation}
		Hence it follows from \eqref{* in Theorem existence and uniqueness of entrance measure}, \eqref{1 in Theorem of Existence and uniqueness of quasi-periodic measure}, \eqref{3 in Theorem of Existence and uniqueness of quasi-periodic measure} and Remark \ref{New remark} that for all $t,r\in \mathbb{R}$,
		\begin{equation*}
			\label{4 in Theorem of Existence and uniqueness of quasi-periodic measure}
			\begin{split}
				\|\mu^{r,r}_t-\mu_{t+r}\|_{TV}&\leq \|P^{r,r}(t,s,x,\cdot)-\mu^{r,r}_t\|_{TV}+\|P(t+r,s+r,x,\cdot)-\mu_{t+r}\|_{TV}\\
			    &\leq 2C(1+|x|^2)e^{-\lambda(t-s)}. 
			\end{split}
		\end{equation*}
		Taking $s\to -\infty$, we know that $\mu^{r,r}_t=\mu_{t+r}$. Moreover, \eqref{2 in Theorem of Existence and uniqueness of quasi-periodic measure} and \eqref{3 in Theorem of Existence and uniqueness of quasi-periodic measure} imply that for all $t,r_1,r_2\in \mathbb{R}$,
		\begin{equation*}
			\mu^{r_1+\tau_1,r_2}_t=\mu^{r_1,r_2}_t=\mu^{r_1,r_2+\tau_2}_t.
		\end{equation*}
		Now define $\tilde{\mu}_{t_1,t_2}:=\mu_0^{t_1,t_2}$, it is easy to see that $\tilde{\mu}$ satisfies \eqref{* in Definition of quasi-periodic measure}.

		Next we will show that $\tilde{\mu}$ is continuous. Now fix $f\in C_b(\mathbb{R}^d)$. Then it follows from \eqref{3 in Theorem of Existence and uniqueness of quasi-periodic measure} that for any $s>0$,
		\begin{align*}
      &\quad \ |\langle f, \tilde{\mu}_{t'_1,t'_2}\rangle-\langle f, \tilde{\mu}_{t_1,t_2}\rangle|\\
      &\leq |\langle f, \mu_0^{t'_1,t'_2}-P^{t'_1,t'_2}(0,-s,0,\cdot)\rangle|+|\langle f, P^{t_1,t_2}(0,-s,0,\cdot)-\mu_0^{t_1,t_2}\rangle|\\
			&\quad \ +|\langle f, P^{t'_1,t'_2}(0,-s,0,\cdot)-P^{t_1,t_2}(0,-s,0,\cdot)\rangle|\\
			&\leq 2|f|_{\infty}Ce^{-\lambda s}+\Big|\mathbf{E}\left[ f\left( K^{t'_1,t'_2}(0,-s,0) \right) \right]-\mathbf{E}\left[ f\left( K^{t_1,t_2}(0,-s,0) \right) \right]\Big|
		\end{align*}
		By Lemma \ref{Lemma of convergence in probability of K}, it follows that for any fixed $s>0$, 
		\begin{equation*}
			\Big|\mathbf{E}\left[ f\left( K^{t'_1,t'_2}(0,-s,0) \right) \right]-\mathbf{E}\left[ f\left( K^{t_1,t_2}(0,-s,0) \right) \right]\Big|\to 0 \ \text{ as } \ (t_1',t_2')\to (t_1,t_2),
		\end{equation*}
		so
		\begin{equation*}
			\limsup_{(t'_1,t'_2)\to (t_1,t_2)}|\langle f, \tilde{\mu}_{t'_1,t'_2}\rangle-\langle f, \tilde{\mu}_{t_1,t_2}\rangle|\leq 2|f|_{\infty}Ce^{-\lambda s}.
		\end{equation*}
		Hence $\tilde{\mu}$ is continuous as $s$ can be an arbitrary positive number.

		Finally we show that the continuous $\tilde{\mu}$ is unique. If there exists another continuous $\tilde{\mu}'_{t_1,t_2}$ which is periodic in $(t_1,t_2)$ with periods $(\tau_1,\tau_2)$ such that $\tilde{\mu}'_{t,t}=\mu_t$ for all $t\in \mathbb{R}$, then for all $t\in \mathbb{R}$,
		$$\tilde{\mu}'_{t\mod \tau_1, t\mod \tau_2}=\tilde{\mu}_{t\mod \tau_1, t\mod \tau_2}.$$
		Since the reciprocals of $\tau_1$ and $\tau_2$ are rationally linearly independent, then 
		$$\{(t\mod \tau_1, t\mod \tau_2): t\in \mathbb{R}\}$$
		is dense in $[0,\tau_1)\times [0,\tau_2)$. Hence $\tilde{\mu}'_{t_1,t_2}=\tilde{\mu}_{t_1,t_2}$ for all $(t_1,t_2)\in [0,\tau_1)\times [0,\tau_2)$ by the continuity of $\tilde{\mu}'$ and $\tilde{\mu}$. Therefore $\tilde{\mu}'=\tilde{\mu}$ by the periodicity of $\tilde{\mu}'$ and $\tilde{\mu}$.
	\end{proof}

	\begin{example}
		\label{eg3}
		Consider the following one-dimensional SDE:
    \begin{equation}
    	\label{0730-2}
    	dX_t=(X_t-X_t^3+C_1\cos (w_1 t)+C_2\cos (w_2 t))dt+dW_t,
    \end{equation}
	which is a special case of \eqref{sde 1108} where $f(t)=C_1\cos (w_1 t)+C_2\cos (w_2 t)$ with constants $w_1>0,w_2>0, C_1, C_2\in \mathbb{R}$. Here assume that $w_1,w_2$ are rationally independent. It is obvious that the drift term of SDE \eqref{0730-2},
	$$b(t,x)=x-x^3+C_1\cos (w_1 t)+C_2\cos (w_2 t)),$$
	is a quasi-periodic function with periods $\frac{2\pi}{w_1}$ and $\frac{2\pi}{w_2}$. Moreover, for all $t,x\in \mathbb{R}$,
    $$xb(t,x)\leq -2x^2+(|C_1|+|C_2|)^2+4.$$
	\end{example}
    Then Theorem \ref{Theorem of Existence and uniqueness of quasi-periodic measure} implies the following result.
    \begin{corollary}
    	If $C_1C_2\neq 0$, SDE  \eqref{0730-2} has a unique continuous quasi-periodic measure with periods $\frac{2\pi}{w_1}$ and $\frac{2\pi}{w_2}$. If $C_1C_2=0$, SDE  \eqref{0730-2} has a unique periodic measure of period $\frac{2\pi}{w_1}$ when $C_1\neq 0$, a unique periodic measure of period $\frac{2\pi}{w_2}$ when $C_2\neq 0$ and a unique invariant measure when $C_1=C_2=0$.
    \end{corollary}

The case when $C_1C_2=0$ and one of them is nonzero is 
the well-known Benzi-Parisi-Sutera-Vulpiani's stochastic resonance model of climate change modelling the transition of ice-age and interglacial period. The existence and uniqueness of the periodic measure was obtained in Feng-Zhao-Zhong \cite{feng2019existence} with convergence in the total variation distance and it is noted that the uniqueness of the periodic measure implies the transitions between the two climates of ice-age and interglacial period. 
In the case when $C_1=C_2=0$, the existence and uniqueness of an invariant measure is a well-known result e.g. see Varadhan \cite{Varadhan}.

Note that SDE \eqref{0730-2} is uniformly weakly dissipative, now we give an example of SDE that is not weakly dissipative.

	\begin{example}
		Let us consider the following one-dimensional SDE:
    \begin{equation}
    	\label{0730}
    	dX_t=\big(C_1|\sin(w_1 t)|X_t-C_2\sin^+(w_2 t)X_t^3+C_3\big)dt+dW_t,
    \end{equation}
	where $W_t, t\in \mathbb{R}$, is a two-sided one-dimensional Brownian motion, $w_1,w_2, C_1, C_2, C_3$ are given positive constant. Here assume that $w_1,w_2$ are rationally independent. It is obvious that the drift term of SDE \eqref{0730},
	$$b(t,x)=C_1|\sin(w_1 t)|x-C_2\sin^+(w_2 t)x^3+C_3,$$
	is a quasi-periodic function with periods $\frac{2\pi}{w_1}$ and $\frac{2\pi}{w_2}$. 
	Note that for any $a>0$,
	\begin{equation*}
		xb(t,x)\leq \big(
			C_1+C_3-C_2a\sin^+(w_2 t)\big)|x|^2+C_2a^2+C_3.
	\end{equation*}
	Choosing $a=\frac{(C_1+C_3+1)\pi}{C_2}$, we have
	\begin{equation}
		\begin{split}
			xb(t,x)\leq& \big(C_1+C_3-(C_1+C_3+1)\pi\sin^+(w_2 t)\big)|x|^2+\frac{(C_1+C_3+1)^2\pi^2}{C_2}+C_3\\
			=:&\alpha_{t}|x|^2+\Lambda.
		\end{split}
	\end{equation}
	If we let
	\begin{equation*}
		\tilde{\alpha}_{t_1,t_2}=C_1+C_3-(C_1+C_3+1)\pi\sin^+(w_2 t_2).
	\end{equation*}
	It is easy to see that $\alpha_{t}$ is a quasi-periodic function (actually a periodic function) with periods $\frac{2\pi}{w_1}$ and $\frac{2\pi}{w_2}$ such that $\tilde{\alpha}_{t,t}=\alpha_t$.
	Moreover, 
	$$\int_{0}^{\frac{2\pi}{w_1}}\int_{0}^{\frac{2\pi}{w_2}}\tilde{\alpha}_{t_1,t_2}dt_1dt_2=-\frac{4\pi^2}{w_1w_2}<0.$$
	Then Theorem \ref{Theorem of Existence and uniqueness of quasi-periodic measure} yields that SDE \eqref{0730} has a unique continuous quasi-periodic measure with periods $\frac{2\pi}{w_1}$ and $\frac{2\pi}{w_2}$.
	\end{example}


    \subsection{Invariant measures and ergodicity through lifting}

    Consider the cylinder $\hat{\mathbb{X}}=[0,\tau_1)\times [0,\tau_2)\times \mathbb{R}^d$ with the following natural metric
    $$\hat{d}(\hat{x},\hat{y}):=d_0((t_1, t_2),(s_1, s_2))+|x-y|, \text{ for all } \hat{x}=(t_1,t_2,x),\hat{y}=(s_1,s_2,y)\in \hat{\mathbb{X}},$$
    where $d_0$ defined in \eqref{Metric on torus} which makes $\mathbb{T}^2:=[0,\tau_1)\times [0,\tau_2)$ a torus. Denote by $\mathcal{B}(\hat{\mathbb{X}})$ the Borel measurable set on $\hat{\mathbb{X}}$ deduced by the metric $\hat{d}$. 
  
    Now we lift the solution semiflow of SDE \eqref{SDE} on the cylinder $\hat{\mathbb{X}}$ by the following:
    $$\hat{\Phi}(t,\omega)(s_1,s_2,x)=(t+s_1\mod \tau_1,\ t+s_2 \mod\tau_2,\ K^{s_1,s_2}(t,0,x,\omega)),$$
    where $K^{r_1,r_2}$ is the solution of \eqref{Equation K_r_1,r_2}. Then by Lemma 3.12 in \cite{feng2021random}, $\hat{\Phi}: \mathbb{R}^+\times \hat{\mathbb{X}}\times\Omega\rightarrow \hat{\mathbb{X}}$ is a cocycle on $\hat{\mathbb{X}}$ over the metric dynamical system $(\Omega, \mathcal{F}, \mathbf{P}, (\theta_t)_{t\in \mathbb{R}})$.
  
    Consider the Markovian transition $\hat{P}:\mathbb{R}^+\times \hat{\mathbb{X}}\times \mathcal{B}(\hat{\mathbb{X}})\rightarrow [0,1]$ generated by the cocycle $\hat{\Phi}$, i.e.,
      $$\hat{P}(t, (s_1,s_2,x), \hat{A})=\mathbf{P}(\omega: \hat{\Phi}(t,\omega)(s_1,s_2,x)\in \hat{A}),$$
      for all $t\in \mathbb{R}^+, (s_1,s_2,x)\in \hat{\mathbb{X}}, \hat{A}\in \mathcal{B}(\hat{\mathbb{X}})$. Similarly, denote by $\hat{P}_t$ and $\hat{P}^*_t$ the linear operator on functions $f: \hat{\mathbb{X}}\to \mathbb{R}$ and its dual operator on $\mathcal{P}(\hat{\mathbb{X}})$. The following proposition says $\hat{P}_t$ is a Feller semigroup.

      \begin{proposition}
        \label{Feller property of P^*}
        Under Assumptions \ref{A3}, \ref{A4}, the semigroup $\hat{P}_t$ is Feller, i.e. for all $f\in C_{b}(\hat{\mathbb{X}})$, $\hat{P}_tf\in C_b(\hat{\mathbb{X}})$.
      \end{proposition}
      
      \begin{proof}
        Consider $f\in C_b(\hat{\mathbb{X}})$. Obviously $|\hat{P}_tf|_{\infty}\leq |f|_{\infty}$, then we just need to prove that $\hat{P}_tf$ is continuous. It is sufficient to prove that for any sequence $\hat{x}_n=(r_1^n,r_2^n,x_n), \hat{x}=(r_1,r_2,x)\in \hat{\mathbb{X}}$ with $\hat{x}_n\xrightarrow{n\rightarrow \infty} \hat{x}$, we have $\hat{P}_tf(\hat{x}_n)\xrightarrow{n\rightarrow \infty} \hat{P}_tf(\hat{x})$. Since
        \begin{equation*}
        \begin{split}
        \hat{P}_tf(\hat{x})
        &=\int_{[0, \tau_1) \times [0, \tau_2)\times\mathbb{R}^d}\hat{P}(t,(r_1,r_2,x),ds_1\times ds_2\times dy)f(s_1,s_2,y)\\
        &=\int_{[0, \tau_1) \times [0, \tau_2)\times\mathbb{R}^d}\mathbf{P}\{\hat{\Phi}(t,\cdot)(r_1,r_2,x) \in ds_1\times ds_2\times dy\}f(s_1,s_2,y)\\
        &=\int_{\mathbb{R}^d}\mathbf{P}\{K^{r_1,r_2}(t,0,x) \in dy\}f(t+r_1\mod \tau_1,t+r_2\mod \tau_2,y)\\
        &=\mathbf{E}f(t+r_1\mod \tau_1,t+r_2\mod \tau_2,K^{r_1,r_2}(t,0,x)).
        \end{split}
        \end{equation*}
        Let $f_t(r_1,r_2,x):=f(t+r_1\mod \tau_1,t+r_2\mod \tau_2, x)$. Then we have
        \begin{equation*}
        \begin{split}
        |\hat{P}_tf(\hat{x}_n)-\hat{P}_tf(\hat{x})|
        &=\big|\mathbf{E}f_t(r_1^n,r_2^n,K^{r_1^n,r_2^n}(t,0,x_n))-\mathbf{E}f_t(r_1,r_2,K^{r_1,r_2}(t,0,x))\big|\\
        &\leq \big|\mathbf{E}f_t(r_1^n,r_2^n,K^{r_1^n,r_2^n}(t,0,x_n))-\mathbf{E}f_t(r_1,r_2,K^{r_1^n,r_2^n}(t,0,x_n))\big|\\
        &\quad \ +\big|\mathbf{E}f_t(r_1,r_2,K^{r_1^n,r_2^n}(t,0,x_n))-\mathbf{E}f_t(r_1,r_2,K^{r_1,r_2}(t,0,x))\big|\\
        &=:A_1^n+A_2^n.
        \end{split}
        \end{equation*}
        Note that $\hat{x}_n\xrightarrow{n\rightarrow \infty} \hat{x}$ implies that $x_n\xrightarrow{n\rightarrow \infty} x$. Hence there exists $M>0$ such that $\sup_{n\in \mathbb{N}}|x_n|\leq M$. Denote 
        $$T_N^{x_n}:=\inf\{t\geq 0: |K^{r_1^n,r_2^n}(t,0,x_n)|> N\}.$$
        Since $\tilde{b}^{r_1^n,r_2^n},\tilde{\sigma}^{r_1^n,r_2^n}$ satisfies Assumption \ref{A4}, then it follows from \eqref{Estimate of stopping time T_N} in Theorem \ref{Theorem of uniformly bounded solution} that
        \begin{align*}
          \mathbf{P}\{T_N^{x_n}\leq t\}\leq \frac{1}{N^{2}}e^{2g(t)}\bigg( M^2+\int_{0}^{t}e^{-2\int_{0}^{u}\alpha_rdr}(2\Lambda_u+d\Gamma_1)du \bigg).
        \end{align*}
        For arbitrary $\epsilon>0$, choose $N$ big enough such that $\mathbf{P}\{T_N^{x_n}\leq t\}<\epsilon$, then
        \begin{align*}
          A_1^n&\leq \mathbf{E}\left[ \big|f_t\big(r_1^n,r_2^n,K^{r_1^n,r_2^n}(t,0,x_n)\big)-f_t\big(r_1,r_2,K^{r_1^n,r_2^n}(t,0,x_n)\big)\big|1_{\{T_N^{x_n}> t\}} \right]\\
          &\quad \ +\mathbf{E}\left[ \big|f_t\big(r_1^n,r_2^n,K^{r_1^n,r_2^n}(t,0,x_n)\big)-f_t\big(r_1,r_2,K^{r_1^n,r_2^n}(t,0,x_n)\big)\big|1_{\{T_N^{x_n}\leq t\}} \right]\\
          &\leq \mathbf{E}\left[ \big|f_t\big(r_1^n,r_2^n,K^{r_1^n,r_2^n}(t\wedge T_N^{x_n},0,x_n)\big)-f_t\big(r_1,r_2,K^{r_1^n,r_2^n}(t\wedge T_N^{x_n},0,x_n)\big)\big| \right]+2|f|_{\infty}\epsilon\\
          &\leq \sup_{y\in B_N}|f_t(r_1^n,r_2^n,y)-f_t(r_1,r_2,y)|+2|f|_{\infty}\epsilon,
        \end{align*}
        where we use the fact that $\big|K^{r_1^n,r_2^n}(t\wedge T_N^{x_n},0,x_n)\big|\leq N$. Since $f\in C_b(\hat{\mathbb{X}})$, then $f_t\in C_b(\hat{\mathbb{X}})$ and hence uniformly continuous in $\mathbb{T}^2\times B_N$. Note that $d_0((r_1^n,r_2^n),(r_1,r_2))\to 0$, then $$\lim_{n\to \infty}\sup_{y\in B_N}|f_t(r_1^n,r_2^n,y)-f_t(r_1,r_2,y)|=0.$$
        Then $\limsup_{n\to \infty}A_1^n\leq 2|f|_{\infty}\epsilon$ and hence $\lim_{n\to \infty}A_1^n=0$ as $\epsilon>0$ is arbitrary.
        
        Note that $f\in C_b(\hat{\mathbb{X}})$ also implies $f_t(r_1,r_2,\cdot)\in C_b(\mathbb{R}^d)$. Lemma \ref{Lemma of convergence in probability of K} shows that $A_2^n\to 0$ as $n\to \infty$. Hence $\lim_{n\to \infty}\hat{P}_tf(\hat{x}_n)=\hat{P}_tf(\hat{x})$.
      \end{proof}
      
      Now we are ready to find the invariant measure of $\hat{P}_t$. We first recall the following lemma.
      
      \begin{lemma}[{\cite[Theorem 3.13]{feng2021random}}]
        If $\mu: \mathbb{R}\rightarrow \mathcal{P}(\mathbb{R}^d)$ is a quasi-periodic measure of SDE \eqref{SDE}, then $\hat{\mu}: \mathbb{R}\rightarrow \mathcal{P}(\hat{\mathbb{X}})$ defined by 
        \begin{align}
          \label{Quasi-periodic measure on cylinder}
          \hat{\mu}_t=\tilde{\hat{\mu}}_{t,t} \text{ with } 
          \tilde{\hat{\mu}}_{t_1,t_2}=\delta_{t_1\mod \tau_1}\times \delta_{t_2\mod \tau_2}\times \tilde{\mu}_{t_1,t_2}, \text{ for all } t,t_1,t_2\in \mathbb{R}.
        \end{align}
        is an quasi-periodic measure of semigroup $\hat{P}^*$, where $\tilde{\mu}_{t_1,t_2}$ defined in \eqref{* in Definition of quasi-periodic measure} and $\delta_t(\cdot)$ is the $\delta$-measure on $\mathbb{R}$.
      \end{lemma}
      
      For the quasi-periodic measure $\hat{\mu}$ given by \eqref{Quasi-periodic measure on cylinder}, set
      $$\bar{\hat{\mu}}_T:=\frac{1}{T}\int_{0}^{T}\hat{\mu}_tdt$$
      and
      \begin{equation}
       \label{Tight measure set}
        \widehat{\mathcal{M}}:=\{\bar{\hat{\mu}}_T: T\in \mathbb{R}^+\}.
      \end{equation}
      
      We also have the following lemma.
      \begin{lemma}[{\cite[Lemma 3.14]{feng2021random}}]
        \label{Lemma of tight measure set}
        Under Assumptions \ref{A3}, \ref{A3 new}, \ref{A1 1203} and \ref{A4}, $\widehat{\mathcal{M}}$ is tight and hence is weakly compact in $\mathcal{P}(\hat{\mathbb{X}})$.
      \end{lemma}
      
      For convenience, we still use $\mathcal{M}$ to denote the set of all probability measures $\nu\in \mathcal{P}(\hat{\mathbb{X}})$ satisfying
      \begin{equation*}
        \int_0^{\tau_1}\int_0^{\tau_2}\int_{\mathbb{R}^d}|x|^2\nu(dt_1\times dt_2\times dx)<\infty.
      \end{equation*}

      Using Lemma \ref{Lemma of tight measure set} and Proposition \ref{Feller property of P^*}, we can prove the existence and uniqueness of invariant measure under $\hat{P}^*$.
      
      \begin{theorem}
        \label{Theorem of existence and uniqueness of invariant measure}
        Under Assumptions \ref{A3}, \ref{A3 new}, \ref{A1 1203} and \ref{A4}, there exists a unique invariant probability measure in $\mathcal{M}$ with respect to the semigroup $\hat{P}^*$ which is given by
        \begin{equation}
          \label{* in Theorem of existence and uniqueness of invariant measure}
          \frac{1}{\tau_1\tau_2}\int_0^{\tau_1}\int_0^{\tau_2}\delta_{t_1}\times \delta_{t_2}\times \tilde{\mu}_{t_1,t_2}dt_1dt_2.
        \end{equation}
        Moreover, this invariant measure is ergodic with respect to the semigroup $\hat{P}^*$
      \end{theorem}
      \begin{proof}
        From Lemma \ref{Lemma of tight measure set}, we know that $\widehat{\mathcal{M}}$ defined by (\ref{Tight measure set}) is tight and hence weakly compact. This means that there exists a sequence $\{T_n\}_{n\geq 1}$ with $T_n\uparrow \infty$ as $n\rightarrow \infty$ and a probability measure $\bar{\hat{\mu}}\in \mathcal{P}(\hat{\mathbb{X}})$ such that $\bar{\hat{\mu}}_{T_n}\xrightarrow{\mathcal{W}} \bar{\hat{\mu}}$. Moreover, for any fixed $t>0$, since 
        \begin{equation}
          \label{0716-2}
          \begin{split}
            \hat{P}^*_t\bar{\hat{\mu}}_{T_n}-\bar{\hat{\mu}}_{T_n}
          &=\frac{1}{T_n}\int_{0}^{T_n}\hat{P}^*_t\hat{\mu}_sds-\frac{1}{T_n}\int_{0}^{T_n}\hat{\mu}_sds\\
          &=\frac{1}{T_n}\int_{0}^{T_n}\hat{\mu}_{t+s}ds-\frac{1}{T_n}\int_{0}^{T_n}\hat{\mu}_sds\\
          &=\frac{1}{T_n}\int_{t}^{t+T_n}\hat{\mu}_sds-\frac{1}{T_n}\int_{0}^{T_n}\hat{\mu}_sds\\
          &=\frac{1}{T_n}\int_{T_n}^{t+T_n}\hat{\mu}_sds-\frac{1}{T_n}\int_{0}^{t}\hat{\mu}_sds,\\
          \end{split}
        \end{equation}
        thus
        \begin{equation*}
          \begin{split}
          \limsup_{n\rightarrow \infty}\|\hat{P}^*_t\bar{\hat{\mu}}_{T_n}-\bar{\hat{\mu}}_{T_n}\|_{TV}
           &\leq \limsup_{n\rightarrow \infty}\frac{1}{T_n}(\int_0^t\|\hat{\mu}_s\|_{TV}ds+\int_{T_n}^{T_n+t}\|\hat{\mu}_s\|_{TV}ds)\\
           &\leq \limsup_{n\rightarrow \infty}\frac{2t}{T_n}=0.
          \end{split}
        \end{equation*}
        Hence $\hat{P}^*_t\bar{\hat{\mu}}_{T_n}\xrightarrow{\mathcal{W}} \bar{\hat{\mu}}$. On the other hand, for any $f\in C_b(\hat{\mathbb{X}})$, by Proposition \ref{Feller property of P^*}, we have $\hat{P}_tf\in C_b(\hat{\mathbb{X}})$, and therefore
        \begin{equation}
          \label{0716-3}
          \begin{split}
          \lim_{n\rightarrow \infty}\int_{\hat{\mathbb{X}}}f(\hat{x})\hat{P}^*_t\bar{\hat{\mu}}_{T_n}(d\hat{x})
           &=\lim_{n\rightarrow \infty}\int_{\hat{\mathbb{X}}}\hat{P}_tf(\hat{x})\bar{\hat{\mu}}_{T_n}(d\hat{x})\\
           &=\int_{\hat{\mathbb{X}}}\hat{P}_tf(\hat{x})\bar{\hat{\mu}}(d\hat{x})\\
           &=\int_{\hat{\mathbb{X}}}f(\hat{x})\hat{P}^*_t\bar{\hat{\mu}}(d\hat{x}).
          \end{split}
        \end{equation}
          This means $\hat{P}^*_t\bar{\hat{\mu}}_{T_n}\xrightarrow{\mathcal{W}} \hat{P}^*_t\bar{\hat{\mu}}$.  Summarizing above we have that $\hat{P}^*_t\bar{\hat{\mu}}=\bar{\hat{\mu}}$. 
      
        Next we will show that this invariant measure will actually equal to \eqref{* in Theorem of existence and uniqueness of invariant measure}. First we will prove that $\tilde{\hat{\mu}}_{t_1,t_2}$ defined by \eqref{Quasi-periodic measure on cylinder} is continuous in $(t_1,t_2)$. For any $f\in C_{b,Lip}(\hat{\mathbb{X}})$ with Lipschitz constant $L_f$ and $t_1,t_2,t'_1,t'_2\in \mathbb{R}$, we have
        \begin{equation}
          \label{1 in Theorem of existence and uniqueness of invariant measure}
          \begin{split}
            |\langle f, \tilde{\hat{\mu}}_{t'_1,t'_2}\rangle-\langle f, \tilde{\hat{\mu}}_{t_1,t_2}\rangle|&=\Big|\int_{\mathbb{R}^d}f_{t'_1,t'_2}(x)\tilde{\mu}_{t'_1,t'_2}(dx)-\int_{\mathbb{R}^d}f_{t_1,t_2}(x)\tilde{\mu}_{t_1,t_2}(dx)\Big|\\
          &\leq \Big|\int_{\mathbb{R}^d}f_{t'_1,t'_2}(x)\tilde{\mu}_{t'_1,t'_2}(dx)-\int_{\mathbb{R}^d}f_{t_1,t_2}(x)\tilde{\mu}_{t'_1,t'_2}(dx)\Big|\\
          &\quad \ +\Big|\int_{\mathbb{R}^d}f_{t_1,t_2}(x)\tilde{\mu}_{t'_1,t'_2}(dx)-\int_{\mathbb{R}^d}f_{t_1,t_2}(x)\tilde{\mu}_{t_1,t_2}(dx)\Big|,
          \end{split}
        \end{equation}
        where $f_{t_1,t_2}(x):=f(t_1\mod \tau_1, t_2\mod \tau_2, x)$. By \eqref{Metric on torus} the definition of $d_i, i=1,2$, we know that
        \begin{equation*}
          d_i(t'_i\mod \tau_i, t_i\mod \tau_i)\leq |t'_i-t_i|, \ i=1,2.
        \end{equation*}
        Hence
        \begin{equation}
          \label{2 in Theorem of existence and uniqueness of invariant measure}
          \begin{split}
            \Big|\int_{\mathbb{R}^d}f_{t'_1,t'_2}(x)\tilde{\mu}_{t'_1,t'_2}(dx)-\int_{\mathbb{R}^d}f_{t_1,t_2}(x)\tilde{\mu}_{t'_1,t'_2}(dx)\Big|\leq L_f(|t'_1-t_1|+|t'_2-t_2|).
          \end{split}
        \end{equation}
        Note that $f\in C_{b,Lip}(\hat{\mathbb{X}})$ implies that $f_{t_1,t_2}(\cdot)\in C_b(\mathbb{R}^d)$. Then it follows from \eqref{1 in Theorem of existence and uniqueness of invariant measure}, \eqref{2 in Theorem of existence and uniqueness of invariant measure} and the continuity of $\tilde{\mu}$ in Theorem \ref{Theorem of Existence and uniqueness of quasi-periodic measure} that
        \begin{equation*}
          \lim_{(t'_1,t'_2)\to (t_1,t_2)}|\langle f, \tilde{\hat{\mu}}_{t'_1,t'_2}\rangle-\langle f, \tilde{\hat{\mu}}_{t_1,t_2}\rangle|=0.
        \end{equation*}
        Hence $\tilde{\hat{\mu}}$ is continuous.
        
         Note that the continuity of $\tilde{\hat{\mu}}$ yields that for all $f\in C_b(\hat{\mathbb{X}})$, $\langle f, \tilde{\hat{\mu}}_{t_1,t_2}\rangle$ is continuous in $(t_1,t_2)$. Since $T_t$ defined by \eqref{0919-2} is a minimal rotation with the unique ergodic probability measure $\frac{1}{\tau_1\tau_2}{\rm Leb}$ on $[0,\tau_1)\times [0,\tau_2)$, then by Birkhoff's ergodic theory,
          \begin{equation}
          \label{0716-1}
          \begin{split}
              \langle f, \bar{\hat{\mu}}_T\rangle&=\frac{1}{T}\int_{0}^{T}\langle f, \hat{\mu}_t\rangle dt\\
                 &=\frac{1}{T}\int_{0}^{T}\langle f, \tilde{\hat{\mu}}_{T_t(0,0)}\rangle dt\\
                 &\xrightarrow{T\rightarrow \infty}\int_{[0, \tau_1) \times [0, \tau_2)}\langle f, \tilde{\hat{\mu}}_{t_1,t_2}\rangle \frac{1}{\tau_1\tau_2}dt_1dt_2,
          \end{split}
          \end{equation}
        which implies that $\bar{\hat{\mu}}_T\xrightarrow{\mathcal{W}}\frac{1}{\tau_1\tau_2}\int_{[0, \tau_1) \times [0, \tau_2)}\tilde{\hat{\mu}}_{t_1,t_2}dt_1dt_2$. Then
        \begin{equation*}
          \bar{\hat{\mu}}=\frac{1}{\tau_1\tau_2}\int_{[0, \tau_1) \times [0, \tau_2)}\tilde{\hat{\mu}}_{t_1,t_2}dt_1dt_2=\frac{1}{\tau_1\tau_2}\int_0^{\tau_1}\int_0^{\tau_2}\delta_{t_1}\times \delta_{t_2}\times \tilde{\mu}_{t_1,t_2}dt_1dt_2
        \end{equation*}
        is an invariant measure with respect to $\hat{P}^*$.

		Note also that the solution $K^{r_1,r_2}(t,s,x)$ of SDE \eqref{Equation K_r_1,r_2} satisfies \eqref{* in Theorem of uniformly bounded solution}, then by \eqref{3 in Theorem of Existence and uniqueness of quasi-periodic measure} in the proof of Theorem \ref{Theorem of Existence and uniqueness of quasi-periodic measure}, we know that 
        \begin{equation*}
          \begin{split}
			\sup_{t_1,t_2\in \mathbb{R}}\int_{\mathbb{R}^d}|x|^2\tilde{\mu}_{t_1,t_2}(dx)&=\sup_{t_1,t_2\in \mathbb{R}}\int_{\mathbb{R}^d}|x|^2\mu_0^{t_1,t_2}(dx)\\
			&\leq \sup_{t_1,t_2\in \mathbb{R}}\limsup_{s\to -\infty}\int_{\mathbb{R}^d}|x|^2P^{t_1,t_2}(0,s,0,dx)\\
			&\leq \sup_{t_1,t_2\in \mathbb{R}, s\leq 0}\mathbf{E}\bigl[|K^{t_1,t_2}(0,s,0)|^2\bigr] <\infty.
		  \end{split}
        \end{equation*}
        Thus
        \begin{equation*}
          \int_{\hat{\mathbb{X}}}|x|^2\bar{\hat{\mu}}(d\hat{x})=\frac{1}{\tau_1\tau_2}\int_0^{\tau_1}\int_0^{\tau_2}\int_{\mathbb{R}^d}|x|^2 \tilde{\mu}_{t_1,t_2}(dx)dt_1dt_2< \infty.
        \end{equation*}
        Hence $\bar{\hat{\mu}}\in \mathcal{M}$.

	Uniqueness: It remains to prove that for any invariant probability measure $\nu$ in $\mathcal{M}$, we have $\nu=\bar{\hat{\mu}}$. By Lemma 2.9 in \cite{feng2021random}, we only need to prove that for any open set $\hat{\mathcal{O}}\in \mathcal{B}(\hat{\mathbb{X}})$, we have $\nu(\hat{\mathcal{O}})\geq \bar{\hat{\mu}}(\hat{\mathcal{O}})$. Define, for any $r_1,r_2\in \mathbb{R}$,
  $$\hat{\mathcal{O}}^{r_1,r_2}=\{x\in \mathbb{R}^d: (r_1\mod \tau_1, \ r_2\mod \tau_2, x)\in \hat{\mathcal{O}}\},$$
  Note that by Remark 3.6 or (1.3) in \cite{feng2021random}, we know that for all $r_1,r_2,r\in \mathbb{R}$ and $t\geq s$
\begin{equation*}
  K^{r_1,r_2}(t+r,s+r,x)=K^{r_1+r,r_2+r}(t,s,x)\circ \theta_r, \ \mathbf{P}-a.s..
\end{equation*} 
Then
\begin{equation}
  \label{Equation of invariant measure v}
    \begin{split}
    \nu\left(\hat{\mathcal{O}}\right)
    &=\lim_{T\rightarrow \infty}\frac{1}{T}\int_0^T\hat{P}^*_t\nu\left(\hat{\mathcal{O}}\right)dt\\
    &=\lim_{T\rightarrow \infty}\frac{1}{T}\int_0^T\int_{\hat{\mathbb{X}}}\hat{P}\left(t,(s_1,s_2,x),\hat{\mathcal{O}}\right)\nu(d\hat{x})dt\\
    &=\lim_{T\rightarrow \infty}\int_{\hat{\mathbb{X}}} \frac{1}{T}\int_0^T \mathbf{P}\left\{K^{s_1,s_2}(t,0,x,\cdot)\in \hat{\mathcal{O}}^{t+s_1,t+s_2}\right\}dt\nu\left(d\hat{x}\right)\\
    &=\lim_{T\rightarrow \infty}\int_{\hat{\mathbb{X}}} \frac{1}{T}\int_0^T\mathbf{P}\left\{K^{t+s_1,t+s_2}(0,-t,x,\cdot)\in \hat{\mathcal{O}}^{t+s_1,t+s_2}\right\}dt\nu(d\hat{x})\\
    &=\lim_{T\rightarrow \infty}\int_{\hat{\mathbb{X}}} \frac{1}{T}\int_0^TP^{t+s_1,t+s_2}(0,-t,x,\hat{\mathcal{O}}^{t+s_1,t+s_2})dt\nu(d\hat{x}).
    \end{split}
  \end{equation}
  Since $\nu\in \mathcal{M}$, then it turns out from \eqref{Equation of invariant measure v}, \eqref{3 in Theorem of Existence and uniqueness of quasi-periodic measure} and Fatou's lemma that
  \begin{equation*}
      \begin{split}
      \nu\left(\hat{\mathcal{O}}\right)&\geq \liminf_{T\rightarrow \infty}\int_{\hat{\mathbb{X}}} \frac{1}{T}\int_0^T\tilde{\mu}_{t+s_1,t+s_2}(\hat{\mathcal{O}}^{t+s_1,t+s_2})dt\nu(d\hat{x})\\
      &\quad \ -\limsup_{T\rightarrow \infty}\frac{C}{\lambda T}\int_{\hat{\mathbb{X}}} (1+|x|^2)\nu(d\hat{x})\\
      &= \liminf_{T\rightarrow \infty}\int_{\hat{\mathbb{X}}} \frac{1}{T}\int_0^T\tilde{\hat{\mu}}_{t+s_1,t+s_2}(\hat{\mathcal{O}})dt\nu(d\hat{x})\\
      &\geq \int_{\hat{\mathbb{X}}} \left( \liminf_{T\rightarrow \infty}\frac{1}{T}\int_0^T\tilde{\hat{\mu}}_{t+s_1,t+s_2}(\hat{\mathcal{O}})dt \right)\nu(d\hat{x}).
      \end{split}
  \end{equation*}
  Again by Birkhoff's ergodic theory, we know that for all $(s_1,s_2)\in \mathbb{R}^2$
  $$\frac{1}{T}\int_0^T\tilde{\hat{\mu}}_{t+s_1,t+s_2}dt\xrightarrow[T\rightarrow\infty]{\mathcal{W}} \bar{\hat{\mu}}.$$
  Then since $\hat{\mathcal{O}}$ is open, and by Proposition 2.4 in \cite{ikeda2014stochastic}, we obtain that $\nu(\hat{\mathcal{O}})\geq \bar{\hat{\mu}}(\hat{\mathcal{O}})$.

  Finally, we prove that $\bar{\hat{\mu}}$ is ergodic. By Proposition 3.2.7 in \cite{da1996ergodicity}, we just need to prove that $\bar{\hat{\mu}}$ is an extremal point of the set of all invariant measures for the semigroup $\hat{P}^*$. If there exists two invariant measures $\nu_1, \nu_2$ and $\beta\in (0,1)$ such that 
  \begin{equation*}
    \bar{\hat{\mu}}=\beta \nu_1+(1-\beta)\nu_2.
  \end{equation*}
  Since $\bar{\hat{\mu}}\in \mathcal{M}$, then it is easy to check that $\nu_1,\nu_2\in \mathcal{M}$. Hence $\nu_1=\nu_2=\bar{\hat{\mu}}$ which implies that $\bar{\hat{\mu}}$ is extremal.
\end{proof}

\begin{remark}
Following the convergence of $\frac{1}{T}\int_{0}^{T}\hat{\mu}_t dt$ in \eqref{0716-1} by Birkhoff's ergodic theorem, we can use the same argument as from \eqref{0716-2} to \eqref{0716-3} to prove the limit is an invariant measure. This enables us to avoid using the tightness argument given in Lemma \ref{Lemma of tight measure set}. However we include the tightness approach here as this is a more universal argument and does not rely on the geometric structure of the torus.
\end{remark}

\section{Acknowledgements and funding}
A preliminary version of this paper was presented in the Durham Symposium: Stochastic Dynamics, Nonlinear Probability, and Ergodicity, 22-26 August 2022. 
The authors would like to thank the participants for conversations and comments on the results which have helped to the completion of the final version. 
The project was supported by an EPSRC grant (ref EP/S005293/2).






\addtolength{\itemsep}{-1.5 em} 
\setlength{\itemsep}{-3pt}
\footnotesize

\addcontentsline{toc}{section}{References}

\bibliographystyle{siam}
\bibliography{ref}

\begin{thebibliography}{10}

\bibitem{Bates-Lu-Wang2014}
{\sc P.~W. Bates, K.~Lu, and B.~Wang}, {\em Attractors of non-autonomous
  stochastic lattice systems in weighted spaces}, Phys. D, 289 (2014),
  pp.~32--50.

\bibitem{baxendale2011te}
{\sc P.~Baxendale}, {\em T. {E}. {H}arris's contributions to recurrent {M}arkov
  processes and stochastic flows}, Ann. Probab., 39 (2011), pp.~417--428.

\bibitem{benth2007volatility}
{\sc F.~E. Benth and J.~\v{S}altyt\.{e} Benth}, {\em The volatility of
  temperature and pricing of weather derivatives}, Quant. Finance, 7 (2007),
  pp.~553--561.

\bibitem{benzi1982stochastic}
{\sc R.~Benzi, G.~Parisi, A.~Sutera, and A.~Vulpiani}, {\em A theory of
  stochastic resonance in climatic change}, SIAM J. Appl. Math., 43 (1983),
  pp.~565--478.

\bibitem{bogachev2015fokker}
{\sc V.~I. Bogachev, N.~V. Krylov, M.~R\"{o}ckner, and S.~V. Shaposhnikov},
  {\em Fokker-{P}lanck-{K}olmogorov equations}, vol.~207 of Mathematical
  Surveys and Monographs, American Mathematical Society, Providence, RI, 2015.

\bibitem{Branicki-Uda2021}
{\sc M.~Branicki and K.~Uda}, {\em Time-periodic measures, random periodic
  orbits, and the linear response for dissipative non-autonomous stochastic
  differential equations}, Res. Math. Sci., 8 (2021), pp.~Paper No. 42, 62.

\bibitem{Cheban-Liu2020}
{\sc D.~Cheban and Z.~Liu}, {\em Periodic, quasi-periodic, almost periodic,
  almost automorphic, {B}irkhoff recurrent and {P}oisson stable solutions for
  stochastic differential equations}, J. Differential Equations, 269 (2020),
  pp.~3652--3685.

\bibitem{Chekroun2011}
{\sc M.~D. Chekroun, E.~Simonnet, and M.~Ghil}, {\em Stochastic climate
  dynamics: random attractors and time-dependent invariant measures}, Phys. D,
  240 (2011), pp.~1685--1700.

\bibitem{Cherubini-nonlinearity2017}
{\sc A.~M. Cherubini, J.~S.~W. Lamb, M.~Rasmussen, and Y.~Sato}, {\em A random
  dynamical systems perspective on stochastic resonance}, Nonlinearity, 30
  (2017), pp.~2835--2853.

\bibitem{da1996ergodicity}
{\sc G.~Da~Prato and J.~Zabczyk}, {\em Ergodicity for infinite-dimensional
  systems}, vol.~229 of London Mathematical Society Lecture Note Series,
  Cambridge University Press, Cambridge, 1996.

\bibitem{delarue2010density}
{\sc F.~Delarue and S.~Menozzi}, {\em Density estimates for a random noise
  propagating through a chain of differential equations}, J. Funct. Anal., 259
  (2010), pp.~1577--1630.

\bibitem{Dong-Zhang-Zheng2020}
{\sc Z.~Dong, W.~Zhang, and Z.~Zheng}, {\em Random periodic solutions of
  non-autonomous stochastic differential equations}, arXiv preprint
  arXiv:2104.01423,  (2021).

\bibitem{dynkin1978}
{\sc E.~B. Dynkin}, {\em Sufficient statistics and extreme points}, Ann.
  Probab., 6 (1978), pp.~705--730.

\bibitem{Engel-cmp2021}
{\sc M.~Engel and C.~Kuehn}, {\em A {R}andom {D}ynamical {S}ystems
  {P}erspective on {I}sochronicity for {S}tochastic {O}scillations}, Comm.
  Math. Phys., 386 (2021), pp.~1603--1641.

\bibitem{Feng-Liu-Zhao2021}
{\sc C.~Feng, Y.~Liu, and H.~Zhao}, {\em Ergodic numerical approximation to
  periodic measures of stochastic differential equations}, J. Comput. Appl.
  Math., 398 (2021), pp.~Paper No. 113701, 23.

\bibitem{Feng-LiuYj-Zhao2023}
{\sc C.~Feng, Y.~Liu, and H.~Zhao}, {\em Periodic measures and {W}asserstein
  distance for analysing periodicity of time series datasets}, Commun.
  Nonlinear Sci. Numer. Simul., 120 (2023), pp.~Paper No. 107166, 31.

\bibitem{feng2021random}
{\sc C.~Feng, B.~Qu, and H.~Zhao}, {\em Random quasi-periodic paths and
  quasi-periodic measures of stochastic differential equations}, J.
  Differential Equations, 286 (2021), pp.~119--163.

\bibitem{feng2020random}
{\sc C.~Feng and H.~Zhao}, {\em Random periodic processes, periodic measures
  and ergodicity}, J. Differential Equations, 269 (2020), pp.~7382--7428.

\bibitem{Feng-Zhao-Zhong2021}
{\sc C.~Feng, H.~Zhao, and J.~Zhong}, {\em Expected exit time for time-periodic
  stochastic differential equations and applications to stochastic resonance},
  Phys. D, 417 (2021), pp.~Paper No. 132815, 18.

\bibitem{feng2019existence}
\leavevmode\vrule height 2pt depth -1.6pt width 23pt, {\em Existence of
  geometric ergodic periodic measures of stochastic differential equations}, J.
  Differential Equations, 359 (2023), pp.~67--106.

\bibitem{flandoliintroduction}
{\sc F.~Flandoli and E.~Tonello}, {\em An introduction to random dynamical
  systems for climate}, Preprint,  (2021).

\bibitem{Gao-Yan2018}
{\sc L.~Gao and L.~Yan}, {\em On random periodic solution to a neutral
  stochastic functional differential equation}, Math. Probl. Eng.,  (2018),
  pp.~Paper No. 8353065, 9.

\bibitem{Gao-Liu-Sun-Zheng2022}
{\sc P.~Gao, Y.~Liu, Y.~Sun, and Z.~Zheng}, {\em Large deviations principle for
  stationary solutions of stochastic differential equations with multiplicative
  noise}, arXiv preprint arXiv:2206.02356,  (2022).

\bibitem{hairer2011yet}
{\sc M.~Hairer and J.~C. Mattingly}, {\em Yet another look at {H}arris' ergodic
  theorem for {M}arkov chains}, in Seminar on {S}tochastic {A}nalysis, {R}andom
  {F}ields and {A}pplications {VI}, vol.~63 of Progr. Probab.,
  Birkh\"{a}user/Springer Basel AG, Basel, 2011, pp.~109--117.

\bibitem{HIS76}
{\sc J.~D. Hays, J.~Imbrie, and N.~J. Shackleton}, {\em Variations in the
  earth's orbit: Pacemaker of the ice ages}, Science, 194 (1976),
  pp.~1121--1132.

\bibitem{huang2016ergodic}
{\sc W.~Huang, Z.~Lian, and K.~Lu}, {\em Ergodic theory of random anosov
  systems mixing on fibers}, arXiv preprint arXiv:1612.08394,  (2016).

\bibitem{ikeda2014stochastic}
{\sc N.~Ikeda and S.~Watanabe}, {\em Stochastic differential equations and
  diffusion processes}, vol.~24 of North-Holland Mathematical Library,
  North-Holland Publishing Co., Amsterdam; Kodansha, Ltd., Tokyo, second~ed.,
  1989.

\bibitem{Liu-Lu2021}
{\sc R.~Liu and K.~Lu}, {\em Statistical properties of 2d stochastic
  navier-stokes equations with time-periodic forcing and degenerate stochastic
  forcing}, arXiv preprint arXiv:2105.00598,  (2021).

\bibitem{Liu-Lu2022}
\leavevmode\vrule height 2pt depth -1.6pt width 23pt, {\em Exponential mixing
  and limit theorems of quasi-periodically forced 2d stochastic navier-stokes
  equations in the hypoelliptic setting}, arXiv preprint arXiv:2205.14348,
  (2022).

\bibitem{Majda-cpam2001}
{\sc A.~J. Majda, I.~Timofeyev, and E.~Vanden~Eijnden}, {\em A mathematical
  framework for stochastic climate models}, Comm. Pure Appl. Math., 54 (2001),
  pp.~891--974.

\bibitem{majid2021}
{\sc N.~R. Majid and M.~R\"{o}ckner}, {\em The structure of entrance laws for
  time-inhomogeneous {O}rnstein--{U}hlenbeck processes with {L}\'{e}vy noise in
  {H}ilbert spaces}, Infin. Dimens. Anal. Quantum Probab. Relat. Top., 24
  (2021), pp.~Paper No. 2150011, 23.

\bibitem{mattingly1999}
{\sc J.~C. {Mattingly}}, {\em {Ergodicity of 2D Navier-Stokes equations with
  random forcing and large viscosity}}, {Commun. Math. Phys.}, 206 (1999),
  pp.~273--288.

\bibitem{menozzi2021density}
{\sc S.~Menozzi, A.~Pesce, and X.~Zhang}, {\em Density and gradient estimates
  for non degenerate {B}rownian {SDE}s with unbounded measurable drift}, J.
  Differential Equations, 272 (2021), pp.~330--369.

\bibitem{meyn1992stability}
{\sc S.~P. Meyn and R.~L. Tweedie}, {\em Stability of {M}arkovian processes.
  {I}. {C}riteria for discrete-time chains}, Adv. in Appl. Probab., 24 (1992),
  pp.~542--574.

\bibitem{Milankovitch1930}
{\sc M.~{Milankovitch}}, {\em {Mathematische Klimalehre und astronomische
  Theorie der Klimaschwankungen.}}
\newblock {IV +176 S. Berlin, Gebr. Borntr\"ager (Handbuch der Klimatologie Bd.
  I, Teil A)}, 1930.

\bibitem{Oksendal2003}
{\sc B.~\O~ksendal}, {\em Stochastic differential equations}, Universitext,
  Springer-Verlag, Berlin, sixth~ed., 2003.
\newblock An introduction with applications.

\bibitem{Raynaud2020almost}
{\sc P.~Raynaud~de Fitte}, {\em Almost periodicity and periodicity for
  nonautonomous random dynamical systems}, Stoch. Dyn., 21 (2020), pp.~Paper
  No. 2150034, 34.

\bibitem{Schmalfuss2001}
{\sc B.~{Schmalfuss}}, {\em {Lyapunov functions and non-trivial stationary
  solutions of stochastic differential equations}}, {Dyn. Syst.}, 16 (2001),
  pp.~303--317.

\bibitem{Song-Song-Zhang2019}
{\sc J.~Song, X.~Song, and Q.~Zhang}, {\em Nonlinear {F}eynman-{K}ac formulas
  for stochastic partial differential equations with space-time noise}, SIAM J.
  Math. Anal., 51 (2019), pp.~955--990.

\bibitem{Uda2021}
{\sc K.~Uda}, {\em Averaging principle for stochastic differential equations in
  the random periodic regime}, Stochastic Process. Appl., 139 (2021),
  pp.~1--36.

\bibitem{Varadhan}
{\sc S.~S. Varadhan}, {\em Lectures on diffusion problems and partial
  differential equations}, Lectures on Mathematics and Physics, Tata Institute
  of Fundamental Research, Bombay, 1980.

\bibitem{walters2000introduction}
{\sc P.~Walters}, {\em An introduction to ergodic theory}, vol.~79 of Graduate
  Texts in Mathematics, Springer-Verlag, New York-Berlin, 1982.

\bibitem{Wang-bifurcation2014}
{\sc B.~Wang}, {\em Existence, stability and bifurcation of random complete and
  periodic solutions of stochastic parabolic equations}, Nonlinear Anal., 103
  (2014), pp.~9--25.

\bibitem{Zhao-Zheng2009}
{\sc H.~Zhao and Z.-H. Zheng}, {\em Random periodic solutions of random
  dynamical systems}, J. Differential Equations, 246 (2009), pp.~2020--2038.

\end{thebibliography}
\end{document}